\documentclass[11pt,a4paper]{article}
\usepackage[top=1in, bottom=1in, left=1in, right=1in]{geometry}
\usepackage[utf8]{inputenc}
\usepackage[english]{babel}
\usepackage{amssymb}
\usepackage{graphicx}
\usepackage[colorlinks=true, allcolors=blue]{hyperref}
\usepackage{tikz}
\usepackage{amsthm}
\usepackage{amsmath}
\usepackage{cancel}
\usepackage{tikz-cd}
\usepackage{mathtools}
\usepackage{tikz}
\usepackage{pgf}
\usepackage{blindtext}
\usepackage{hyperref}
\usetikzlibrary{matrix,arrows,decorations.pathmorphing}
\tikzset{commutative diagrams/.cd}
\usetikzlibrary{cd}
\usepackage{cancel}
\usepackage{nicefrac}
\usepackage{centernot}
\usepackage{mathtools}
\usepackage{ stmaryrd }
\usepackage{wrapfig}
\usepackage{csquotes}
\usepackage{hyperref}
\hypersetup{
    colorlinks=true,       
    linkcolor=BlueGreen,          
    citecolor=BlueGreen,       
    filecolor=BlueGreen,      
    urlcolor=BlueGreen          
}
\usepackage{soul}
\usepackage[dvipsnames]{xcolor}
\usepackage[sorting=nty,style=alphabetic]{biblatex}
\theoremstyle{definition}
\newtheorem{thm}{Theorem}[section]
\newtheorem{lem}[thm]{Lemma}
\newtheorem{prop}[thm]{Proposition}
\newtheorem{cor}[thm]{Corollary}

\newtheorem{quest}[thm]{Question}
\newtheorem*{thm*}{Theorem}
\newtheorem*{cor*}{Corollary}

\newtheorem{theoremalph}{Theorem}
\newtheorem{coralph}[theoremalph]{Corollary}
\newtheorem{conjalph}[theoremalph]{Conjecture}

\newtheorem{rthm}{Theorem}

\newenvironment{reptheorem}[1]
  {\renewcommand\therthm{\ref*{#1}}\rthm}
  {\endrthm}

\theoremstyle{definition}
\newtheorem{defn}[thm]{Definition}
\newtheorem*{defn*}{Definition}
\newtheorem{ex}[thm]{Example}

\theoremstyle{remark}
\newtheorem{rmk}[thm]{Remark}
\theoremstyle{remark}

\newtheorem{nt}[thm]{Notation}

\makeatletter
\DeclareFontFamily{OMX}{MnSymbolE}{}
\DeclareSymbolFont{MnLargeSymbols}{OMX}{MnSymbolE}{m}{n}
\SetSymbolFont{MnLargeSymbols}{bold}{OMX}{MnSymbolE}{b}{n}
\DeclareFontShape{OMX}{MnSymbolE}{m}{n}{
    <-6>  MnSymbolE5
   <6-7>  MnSymbolE6
   <7-8>  MnSymbolE7
   <8-9>  MnSymbolE8
   <9-10> MnSymbolE9
  <10-12> MnSymbolE10
  <12->   MnSymbolE12
}{}
\DeclareFontShape{OMX}{MnSymbolE}{b}{n}{
    <-6>  MnSymbolE-Bold5
   <6-7>  MnSymbolE-Bold6
   <7-8>  MnSymbolE-Bold7
   <8-9>  MnSymbolE-Bold8
   <9-10> MnSymbolE-Bold9
  <10-12> MnSymbolE-Bold10
  <12->   MnSymbolE-Bold12
}{}

\let\llangle\@undefined
\let\rrangle\@undefined
\DeclareMathDelimiter{\llangle}{\mathopen}%
                     {MnLargeSymbols}{'164}{MnLargeSymbols}{'164}
\DeclareMathDelimiter{\rrangle}{\mathclose}%
                     {MnLargeSymbols}{'171}{MnLargeSymbols}{'171}
\makeatother

\makeatletter
\newcommand\ackname{Acknowledgements}
\if@titlepage
  \newenvironment{acknowledgements}{%
      \titlepage
      \null\vfil
      \@beginparpenalty\@lowpenalty
      \begin{center}%
        \bfseries \ackname
        \@endparpenalty\@M
      \end{center}}%
     {\par\vfil\null\endtitlepage}
\else
  \newenvironment{acknowledgements}{%
      \if@twocolumn
        \section*{\abstractname}%
      \else
        \small
        \begin{center}%
          {\bfseries \ackname\vspace{-.5em}\vspace{\z@}}%
        \end{center}%
        \quotation
      \fi}
      {\if@twocolumn\else\endquotation\fi}
\fi
\makeatother

\newcommand{\altfrac}[2]{\ifmmode\def\tmp{$}\else\def\tmp{}\fi\mbox{%
    {\raisebox{.24\ht\strutbox}{\tmp#1\tmp}}%
    \kern-2.2pt\scalebox{1.6}[1.5]{/}\kern-1.8pt%
    {\tmp#2\tmp}%
    }}

\newcommand{\frpp}{\mathop{\scalebox{1.5}{\raisebox{-0.2ex}{$\ast$}}}}%
\newcommand{\Sf}{\mathcal{S}^{\textit{f}}}
\newcommand{\Sn}{\mathfrak{S}_n}
\newcommand{\A}{\mathrm{A}}
\newcommand{\B}{\mathcal{B}}

\newcommand{\W}{\mathrm{W}}
\newcommand{\VB}{\mathcal{VB}}
\newcommand{\VA}{\mathrm{VA}}
\newcommand{\KVA}{\mathrm{KVA}}

\newcommand{\Prod}{\mathrm{Prod}}
\newcommand{\Stab}{\mathrm{Stab}}
\newcommand{\id}{\mathrm{id}}
\newcommand{\hGamma}{\widehat{\Gamma}}
\newcommand{\hGammam}{\widehat{\Gamma_m}}

\newcommand{\hmbg}{\widehat{m}_{\beta,\gamma}}
\newcommand{\D}{\mathrm{D}}
\newcommand{\vD}{\mathrm{vD}}
\newcommand{\cX}{\mathcal{X}}
\newcommand{\hDm}{\D_{\hGammam}}
\newcommand{\idK}{\mathrm{id}_{\mathrm{K}}}
\newcommand{\idW}{\mathrm{id}_{\mathrm{W}}}
\newcommand{\ev}{\mathrm{ev}}
\newcommand{\cl}{\mathfrak{cl}}
\newcommand{\CR}{\mbox{\tiny{$CR$}}}

\usepackage[dvipsnames]{xcolor}

\title{A Deligne complex for virtual Artin groups}
\author{Federica Gavazzi \& Alexandre Martin }
\date{}

\begin{document}

\maketitle
\begin{abstract}
    \noindent Virtual Artin groups were recently introduced by Bellingeri--Paris--Thiel as a generalisation of virtual braid groups. In this article, we initiate a geometric study of these groups by constructing an analogue of the Deligne complex for virtual Artin groups, and we prove that it is CAT(0) for all locally reducible defining graphs (a class that contains in particular two-dimensional graphs and graphs without any label $3$, and which is generic in the sense of Goldsborough--Vaskou).

    \noindent As applications, we classify finite subgroups of locally reducible virtual Artin groups, showing that such groups are conjugated into an isomorphic copy of the corresponding Coxeter subgroup. We also prove an analogue of the $K(\pi, 1)$-conjecture for locally reducible virtual Artin groups: we show that these groups are virtually torsion-free and admit a cocompact model of classifying space for proper actions of minimal dimension, equal to the  virtual cohomological dimension of the group. 
\end{abstract}

\section{Introduction}\label{sectionintro}

\textcolor{red}{}

\noindent Artin groups, also known as Artin-Tits groups, are a widely studied class of groups defined by a presentation with generators and relations. Introduced by Jacques Tits in the sixties, Artin groups first appear as a generalization of another class of interesting groups: Coxeter groups. Both these classes of groups are defined through a presentation with generators and relations.\medskip

\noindent If $m\geq 2$ is an integer and $a,b$ are two letters, we denote by $\Prod_R(s,t;m)$ and $\Prod_L(s,t;m)$ respectively the alternating word $\cdots st$ and the alternating word $st\cdots$, of length $m$. A \textit{Coxeter graph} $\Gamma$ is a finite, simplicial graph $\Gamma$ with a finite vertex set $S$ and edges $\{s,t\}$ labelled by integers $m_{s,t}=m_{t,s}\geq 2$, for $s,t\in S$. Given $\Gamma$ a Coxeter graph on a vertex set $S$, the \textit{Artin group} $\A[\Gamma]$ is the group generated by $S$ with relations $\Prod_R(s,t;m_{s,t})=\Prod_R(t,s;m_{s,t})$ for all $s,t \in S$ such that $s$ and $t$ are connected by an edge. The \textit{Coxeter group} $\W[\Gamma]$ is the quotient of $\A[\Gamma]$ under the normal closure of the relations $\{s^2=\idW,\; \forall\; s\in S\}$. Sometimes, when $s,t\in S$ are not connected by an edge in $\Gamma$, we may say that $m_{s,t}=\infty$. The \textit{rank} of a Coxeter or Artin group associated with $\Gamma$ is $|S|$. If $|S|=2$, then $\Gamma$ is called a \textit{dihedral} Coxeter graph, and it is denoted by $\Gamma=\Gamma_m$, where $m=m_{s,t}$. \medskip \\
\noindent Many natural questions on Artin groups, as determining the torsion, the centre, the classifying space or the word problem, are still open problems (see \cite{Charney2008PROBLEMSRT} for a survey on open problems regarding these groups). Certain classes of Artin groups are better understood than others thanks to approaches of geometric group theory or combinatorial group theory. For instance, \textit{spherical type} Artin groups, which are those whose associated Coxeter group $\W[\Gamma]$ is finite, are well-understood thanks to their Garside structure. Among spherical type (sometimes called \textit{finite type}) Artin groups, we find the well-known braid group on $n$ strands, denoted by $\B_n$.\medskip
\\
A related structure of interest in this work is that of virtual braids, which were introduced by Kauffman in \cite{Kauff} alongside virtual knots and links. The virtual braid group on $n$ strands, denoted here by $\VB_n$, has two families of generators: the classical generators $\{\sigma_i\mid 1\leq i\leq n-1\}$, which satisfy the usual braid relations, and the \textit{virtual} generators $\{\tau_{i}\mid 1\leq i\leq n-1\}$, which satisfy the defining relations of the symmetric group $\Sn$. These two families are linked by a collection of mixed relations, yielding a group with a natural topological interpretation in terms of braid diagrams equipped with an additional type of crossing, called a \textit{virtual} crossing. \medskip
\\
\noindent A generalization of Artin groups appears in a recent article by Bellingeri, Paris and Thiel: virtual Artin groups \cite{BellParThiel}. This new class of groups  extends classical Artin groups in the same way that virtual braid groups extend  classical braid groups. 

\begin{defn}\cite{BellParThiel}\label{def:VA}
Given $\Gamma$ a Coxeter graph with vertex set $S$, consider the two sets: $\mathcal{S}=\{\sigma_s\,|\,s\in S\}$ and $\mathcal{T}=\{\tau_s\,|\,s\in S\}$, both in one-to-one correspondence with $S$. The \textit{virtual Artin group} associated with $\Gamma$ is the abstract group $\VA[\Gamma]$, defined by the presentation with generating set $\mathcal{S} \cup \mathcal{T}$ and the following relations:
\begin{enumerate}
\item[(V1)]\label{v1} $\Prod_R(\sigma_s,\sigma_t;m_{s,t})=\Prod_R(\sigma_t,\sigma_s;m_{s,t})$ $\forall s,t \in S$ such that $\{s,t\}$ is an edge of $\Gamma$;
\item[(V2)]\label{v2} $\Prod_R(\tau_s,\tau_t;m_{s,t})=\Prod_R(\tau_t,\tau_s;m_{s,t})$ $\forall s,t \in S$ such that $\{s,t\}$ is an edge of $\Gamma$; and $\tau_s^2=1$  $\forall s \in S$;
\item [(V3)]\label{v3} $\Prod_R(\tau_s,\tau_t;m_{s,t}-1) \,\sigma_s=\sigma_r \,\Prod_R(\tau_s,\tau_t;m_{s,t}-1)$ $\forall s,t \in S$ such that $\{s,t\}$ is an edge of $\Gamma$, where $r=s$ if $m_{s,t}$ is even and $r=t$ if $m_{s,t}$ is odd.
\end{enumerate}
\end{defn}
\noindent  We emphasize that the term \textit{virtual} in virtual Artin groups derives from the context of virtual braids and \textbf{should not be interpreted as meaning that the group is virtually an Artin group}, except in the special cases described in Subsection \ref{subs-decomp-VA}.\\
\\
\noindent It was shown in \cite{BellParThiel} that, in the virtual Artin group $\VA[\Gamma]$, the sets $\mathcal{S}$ and $\mathcal{T}$ generate isomorphic copies of  $\A[\Gamma]$ and $\W[\Gamma]$ respectively.
Moreover, $\VA[\Gamma]$ admits a semidirect product structure
$\VA[\Gamma]\cong \A[\hGamma]\rtimes\W[\Gamma]$, where the normal subgroup $\A[\hGamma]$ is an Artin group associated with an explicitly defined Coxeter graph $\hGamma$ associated to the root system of $\W[\Gamma]$ (see Subsection \ref{subs-decomp-VA}). This observation provides strong motivation for studying virtual Artin groups, as it allows one to reinterpret and potentially reformulate classical problems in Artin groups theory within this broader algebraic framework.\\
\\
\noindent Classical Artin groups have often been studied via simplicial complexes naturally associated with them. One of the most prominent examples is the well-known \textit{(modified) Deligne complex} of Charney-Davis \cite{CharDav95}, which generalises a construction of Deligne in his solution of the $K(\pi,1)$-conjecture for Artin groups of spherical type \cite{Deligne1972}. The $K(\pi,1)$-conjecture for a given Artin group is equivalent to the contractility of its Deligne complex by \cite{CharDav95}. In \cite{CharDav95}, Ruth Charney and Michael W. Davis proved that, for certain classes of Artin groups, namely \textit{FC-type} Artin groups and \textit{2-dimensional} Artin groups, the Deligne complex $\D_{\Gamma}$ admits a CAT(0) metric. The authors further conjectured that the Moussong metric on $\D_{\Gamma}$ (see Subsection \ref{subs-moussongmetric}) is CAT(0) for every Coxeter graph $\Gamma$. In  \cite{Charney2000}, Charney proved this conjecture result for the class of \textit{locally reducible} Artin groups (see Subsection \ref{subs-from2tolocallyred} for a definition), which   includes in particular 2-dimensional Artin groups and Artin groups with no graph label $3$. The question of whether the classical Deligne complex admits a CAT(0) metric remains one of the central open problems in the theory of Artin groups, with recent progress for three-dimensional Artin groups~\cite{HuangPrzyty}.\\
\\
\noindent  Many questions arising for classical Artin groups can also be formulated for virtual Artin groups, and are sometimes easier to address in this broader setting. For instance, in \cite[Corollary 3.4]{BellParThiel} the authors show
that for any Coxeter graph $\Gamma$, the centre of $\VA[\Gamma]$ is trivial.
Another property that is known to hold for irreducible virtual Artin groups, but is only conjectured to be true for irreducible classical Artin groups, is \textit{indecomposability}, which is studied by the first author in \cite{Gav26}.\\
\\
\noindent In this work, we initiate a geometric study of virtual Artin groups, and introduce an analogue of the Deligne complex for virtual Artin groups, which we call  the \emph{virtual Deligne complex}. As in the Artin case, the virtual Deligne complex $\vD_{\Gamma}$ is a simplicial complex constructed from a poset of cosets of suitable \textit{standard parabolic subgroups} of $\VA[\Gamma]$. Recently, a family of subgroups of $\VA[\Gamma]$ playing the role of standard parabolic subgroups of a virtual Artin group was proposed in \cite{GaGaPa26}, where the authors show that this family of subgroups satisfy an analogue of van der Lek's lemmas (see \cite[Chapter II]{van1983homotopy}). It is possible to equip the virtual Deligne complex with an analogue of the Moussong metric, and as in the Coxeter and Artin case, we expect the following:

\begin{conjalph}\label{conj-vD-CAT0}
For any Coxeter graph $\Gamma$, the virtual Deligne complex $\vD_{\Gamma}$, equipped with the Moussong metric, is CAT(0).
\end{conjalph}

\noindent The main result of this work is a proof of this conjecture in the case of a locally reducible Coxeter graph (see Definition~\ref{def-locred}), a class that contains in particular all the 2-dimensional Coxeter graphs, and all Coxeter graphs with no label equal to $3$, and which is ``generic'' in a suitable probabilistic sense \cite{GoldsboroughVaskou25}. 

\begin{theoremalph}[See Corollary~\ref{cor-CAT0-locred}]
  Let $\Gamma$ be a locally reducible Coxeter graph. Then the virtual Deligne complex~$\vD_{\Gamma}$, equipped with the Moussong metric, is a complete CAT(0) metric space.  
\end{theoremalph}

\noindent In related but independent work, given a virtual Artin group $\VA[\Gamma]$, G\'alvez Mateos and Paris \cite{GalMatPar} construct families of simplicial and cubical complexes~$\Xi_\mathcal{D}$ (for suitable families $\mathcal{D}$ of subsets of $V(\Gamma)$) on which $\VA[\Gamma]$ acts. When $\mathcal{D}$ is the family of subsets of $V(\Gamma)$ of spherical type, the virtual Deligne complex $\vD_\Gamma$ we define in this article coincides with $\Xi_\mathcal{D}$ with its simplicial structure. Although some of our constructions overlap, the perspectives and applications differ: We focus on the Moussong metric on $\vD_\Gamma$, which is conjectured to be CAT(0) in all cases, while G\'alvez Mateos and Paris focus on a different metric on $\Xi_\mathcal{D}$,  characterising precisely when $\Xi_\mathcal{D}$ with its cubical structure is a CAT(0) cube complex. In particular, they show that the virtual Deligne complex $\vD_\Gamma$ with its cubical structure is a CAT(0) cube complex if and only if $\Gamma$ is of FC type (see for instance \cite{CharDav95} for the definition of this class). \\\\
\noindent The existence of a complete CAT(0) metric on the virtual Deligne complex of $\VA[\Gamma]$ has significant consequences for the structure of $\VA[\Gamma]$, which we mention below.\\\\
\noindent We first use this action to classify the finite subgroups of $\VA[\Gamma]$. As observed above, the virtual Artin group $\VA[\Gamma]$ contains an embedded copy, denoted $\iota_{\W}(\W[\Gamma])$, of the associated Coxeter group, and therefore it necessarily contains torsion elements. We recall in Subsection \ref{subs-decomp-VA} that $\VA[\Gamma]$ admits a semidirect product decomposition
$\VA[\Gamma]\cong \A[\hGamma]\rtimes\W[\Gamma]$, where the normal subgroup $\A[\hGamma]$ is an Artin group associated with an explicitly defined Coxeter graph $\hGamma$. Artin groups are conjectured to be torsion-free for every Coxeter graph. (Note that $\A[\hGamma]$ is known to be torsion-free whenever $\Gamma$ is either of spherical or affine type; see Corollary 6.5 and Theorem 6.6 in \cite{BellParThiel}.) Assuming this conjecture, any finite subgroup of $\VA[\Gamma]$ would embed into the corresponding Coxeter group $\W[\Gamma]$. A natural question is whether this embedding can be strengthened to a conjugacy statement. More precisely, we conjecture the following:

\begin{conjalph}[Torsion conjecture for virtual Artin groups]\label{torsion-conjecture}
  Let $\Gamma$ be a Coxeter graph and let $\VA[\Gamma]$ be the associated virtual Artin group. Then every finite subgroup of $\VA[\Gamma]$ is conjugated into a subgroup of $\iota_{\W}(\W[\Gamma])$.  
\end{conjalph}

\noindent Note that the above is unknown even in the case of virtual braid groups. 
In this article, we answer this question affirmatively for locally reducible virtual Artin groups.

\begin{theoremalph}[see Theorem~\ref{thm-finitesubgroups}]
    Let $\Gamma$ be a locally reducible Coxeter graph. Then every finite subgroup of $\VA[\Gamma]$ is conjugated into a subgroup of $\iota_{\W}(\W[\Gamma])$.  
\end{theoremalph}

\noindent A consequence  is the following.
\begin{coralph}[see Corollary~\ref{thm-A[gammahat]-torfree-locred}]
    If $\Gamma$ is locally reducible, then the Artin group $\A[\hGamma]$ is torsion-free and $\VA[\Gamma]$ is virtually torsion-free.
\end{coralph}

\noindent We now turn to finiteness properties of virtual Artin groups. Let $\cl^f(\Gamma)$ denote the \textit{spherical clique number} of $\Gamma$, that is, the number of vertices in a maximal spherical clique of $\Gamma$. The $K(\pi,1)$-conjecture for the Artin group $\A[\Gamma]$ implies in particular the existence of a compact model of Eilenberg-MacLane space for $\A[\Gamma]$ of dimension $\cl^f(\Gamma)$, namely the \textit{Salvetti complex} of $\A[\Gamma]$. In the case of virtual Artin groups, since these groups contain torsion, a natural replacement is to study the existence of a cocompact model of \textit{classifying space for proper actions} of minimal dimension. In particular, we conjecture the following: 

\begin{conjalph}[Minimal classifying space for proper actions]\label{conj-class-properactions}
Let $\Gamma$ be a  Coxeter graph. Then $\VA[\Gamma]$ admits a cocompact model of classifying space for proper actions of dimension $\cl^f(\Gamma)$.  
\end{conjalph}
\noindent
In this article, we solve this conjecture in the locally reducible case: 

\begin{theoremalph}[see Theorem~\ref{thm:cocompact_EG_bar} and Corollary~\ref{cor_vcd-VA}]
Let $\Gamma$ be a locally reducible Coxeter graph. Then $\VA[\Gamma]$ admits a cocompact model of classifying space for proper actions of dimension~$\cl^f(\Gamma)$, and we have 
$$ \underline{\mathrm{gd}}(\VA[\Gamma]) = \mathrm{vcd}(\VA[\Gamma]) =  \cl^f(\Gamma)$$
where $\underline{\mathrm{gd}}$ denotes the minimal dimension of a classifying space for proper actions, and $\mathrm{vcd}$ denotes the virtual cohomological dimension of a group.
\end{theoremalph}

\noindent Note that the virtual cohomological dimension of virtual Artin groups is known in the spherical case by \cite[Theorem 6.3]{BellParThiel}. By contrast, it remains open in more general settings, such as the affine case, although the same authors establish an upper bound for this dimension.\\

\noindent \textbf{Strategy and structure of the article.} The article is organised as follows. Preliminary Sections~\ref{sec-prelimVA} and~\ref{sec-prelim-complexes} introduce known results and terminology about virtual Artin groups and complexes of groups respectively.\\

\noindent In Section \ref{sec-virtualDeligne} we introduce the virtual Deligne complex $\vD_{\Gamma}$. As in the classical setting, we define $\vD_{\Gamma}$ as the universal cover of a strictly developable simple complex of groups built from the family of standard parabolic subgroups of $\VA[\Gamma]$ (see Subsection \ref{subs-defnofVD}), which we equip with the Moussong metric. In Subsection \ref{subs-CAT0-2dim}, we study the geometry of $\vD_{\Gamma}$ when $\Gamma$ has dimension $2$. The proof that $\vD_{\Gamma}$ is locally CAT(0) can be rephrased purely combinatorially. Namely, we prove the following key result, which generalises to the virtual Artin group case a similar statement of Appel--Schupp for Coxeter and Artin groups \cite[Lemma~6]{AppelSchupp83}:

\begin{theoremalph}[``virtual Appel--Schupp Theorem'', see Theorem~\ref{virtualAppelSchupp}]
Let $\Gamma_m$ be a dihedral Coxeter graph with vertex set $S=\{s,t\}$ and $m\geq 2$. Let $\VA_m$ be the associated virtual Artin group, and let $\omega$ be a non-trivial element of $\VA_s\frpp \VA_t$ written in normal form, and representing the identity element of $\VA_m$. Then the syllabic length of $\omega$ satisfies $\|\omega\|\geq 2m$.  
\end{theoremalph}

\noindent This result is the key technical result from this article, whose proof covers Sections \ref{sectionDihedral}, \ref{sec-poly-of-groups-ArtinGammaHat}, and \ref{sec-virtualAppelSchupp}. In Subsection \ref{subs-from2tolocallyred}, we adapt Charney's proof from \cite{Charney2000} to the virtual setting and deduce that the virtual Deligne complex is CAT(0) in the locally reducible case. \\
\\
\noindent In Section \ref{sectionDihedral}, we consider a dihedral Coxeter graph $\Gamma_m$ with vertices $s$ and $t$, together with a non-trivial element in the free product $\VA_s\frpp \VA_t$ that represents the identity of $\VA_m$. In Subsection \ref{subs-fromVAtoAgammaHat}, we exploit the decomposition $\VA_m=\A[\hGammam]\rtimes \W_m$ to deduce that certain equation must hold in $\A[\hGammam]$ (Equation \ref{equationInKVA}). In order to analyse this equation, we first describe the Coxeter graph $\hGammam$ in Subsection \ref{subs-description-gammaHat} (Proposition \ref{propdihedralgammahat}). \\
\\
\noindent In order to study the aforementioned equation in $\A[\hGammam]$, in Section~\ref{sec-poly-of-groups-ArtinGammaHat} we realise
the Artin group $\A[\hGammam]$ as the fundamental group of a new developable simple polygon of groups $G(\mathcal{Q})$ (different from the standard Deligne complex associated with $\hGammam$). The motivation behind this construction is that certain cyclically reduced words representing the identity of $\A[\hGammam]$ correspond to ``loops of polygons'' in the universal cover of this complex of groups (Lemma~\ref{lem-cycred-nonbackloops}). We show that the universal cover of this polygon of groups admits a CAT(0) metric (Proposition~\ref{prop-X'-CAT(0)}), which allows us to use disc diagram arguments to bound the length of such loops. \\
\\
\noindent In Section \ref{sec-virtualAppelSchupp}, we combine the results of Sections \ref{sectionDihedral} and \ref{sec-poly-of-groups-ArtinGammaHat} to prove the virtual Appel--Schupp Theorem \ref{virtualAppelSchupp}. Starting from an element written in normal form in $\VA_s \frpp \VA_t$  representing the identity of $\VA[\Gamma_m]$, we associate to it a word representing the identity of $\A[\hGammam]$. When this word is cyclically reduced, the arguments from Section~\ref{sec-poly-of-groups-ArtinGammaHat} apply and yield the desired lower bound. However, the word may not necessarily be cyclically reduced, and in Subsections \ref{subs-equation-in-A},\ref{nu''-empty} we prove additional results to obtain the required lower bound in this additional case.\\
\\
\noindent In Section \ref{sec-fin-subg}, we study finite subgroups of $\VA[\Gamma]$ when $\Gamma$ is locally reducible, using the action of $\VA[\Gamma]$ on its CAT(0) virtual Deligne complex. By the Bruhat--Tits fixed-point theorem, the study of finite subgroups of $\VA[\Gamma]$ reduces to the study of the finite subgroups of the spherical standard parabolic subgroups of $\VA[\Gamma]$ of dihedral type (Lemma \ref{reduction-to-dihedral-lem}). 
 To study the finite subgroups of a virtual Artin group of dihedral type $\VA_m$, we extend the action of $\A[\hGammam]$ on its CAT(0) Deligne complex $\D_{\hGammam}$ to an action of $\VA_m$. (Proposition \ref{prop-actionofVAm-on-Deligne}). Here again, the study of finite subgroups of $\VA_m$ then reduces to the study of the vertex stabilisers, which are classified in Lemma~\ref{lem-stabilisers}.\\
\\
\noindent In Section \ref{sec-cohom+classifying}, we construct a cocompact model classifying space for proper actions for $\VA[\Gamma]$ when $\Gamma$ is locally reducible. To this end, we first show a ``reduction to spherical parabolic subgroups'' result that holds in full generality: if $\vD_\Gamma$ is CAT(0), then $\VA[\Gamma]$ satisfies Conjecture \ref{conj-class-properactions} if all its spherical standard parabolic subgroups satisfy that conjecture (see Subsection \ref{subs-reduction-parabolic}). In the locally reducible case, the study of the spherical parabolic subgroups essentially boils down to the dihedral case. When $\Gamma$ is of dihedral type, we use the action of $\VA[\Gamma]$ on the Deligne complex $\D_{\hGamma}$ to construct a cocompact model of classifying space for proper actions of dimension~$2$.

\begin{acknowledgements}\noindent
This work was partially supported by the  EPSRC Standard Research Grant UKRI1018. We also thank Jingyin Huang and Luis Paris for useful discussions about this work. 
\end{acknowledgements}

\section{Preliminaries on virtual Artin groups}\label{sec-prelimVA} 
In this section we collect the results on virtual Artin groups and related structures that will be used throughout the paper. As we previously mentioned in the Introduction, the algebraic and geometric study of virtual Artin groups is often eased by their nice structure as semidirect products of better known groups. This decomposition encompasses at the same time classical Artin groups, Coxeter groups, and it mirrors the action of the latter on a crucial structure in the study of these objects: the root system.

\subsection{Root systems} \label{subs-rootsystems}
\noindent 
\noindent Let $\Gamma$ be a Coxeter graph with vertices $V(\Gamma)=S$. Consider the real vector space $V=\bigoplus_{s\in S}\mathbb{R}\cdot\alpha_s$ with basis $\Pi=\{\alpha_s\mid s\in S\}$. The elements $\alpha_s$ are sometimes called the \textit{simple roots}. Define a symmetric bilinear form $\langle\,\cdot,\cdot\rangle$ on $V$ by
\[
\langle\alpha_s,\alpha_t\rangle=\begin{cases}
    -\cos{\left(\frac{\pi}{m_{s,t}}\right)}\qquad& \text{if $\{s,t\}$ is an edge of $\Gamma$,}\\ -1\qquad &  \text{if $\{s,t\}$ is not an edge of $\Gamma$,}\end{cases}\]
and extended bilinearly. Define $\rho:\W[\Gamma]\longrightarrow GL(V)$ as $s\longmapsto \rho_s$ where, for all  $v\in V$, $\rho_s(v):=v-2\langle\alpha_s,v \rangle\, \alpha_s$. This yields a faithful linear representation of $\W[\Gamma]$, called the \textit{canonical linear representation} \cite[V.4]{Bourbaki}. For $w\in \W[\Gamma]$ and $v\in V$, we write $w (v)$ for $\rho_w(v)$. \medskip
\\
\noindent Given $\Gamma$ a Coxeter graph with vertices $S$, the \textit{root system} of $\W[\Gamma]$ is the set of unit vectors
\[
\Phi[\Gamma]=\{w(\alpha_s)\mid s\in S,\,w\in \W[\Gamma]\}\subset V.\]
\noindent A root is \textit{positive} (resp. \textit{negative}) if it is a linear combination of simple roots with non-negative (resp. non-positive) coefficients. Denoting the sets of positive and negative roots respectively by $\Phi^+[\Gamma]$ and $\Phi^-[\Gamma]$, we have $\Phi[\Gamma]=\Phi^+[\Gamma]\sqcup \Phi^-[\Gamma]$ and $\Phi^+[\Gamma]=-\Phi^-[\Gamma]$ (see \cite{deodh82}). \medskip 
\\
\noindent If $\beta=w(\alpha_s)\in\Phi[\Gamma]$, the element $r_{\beta}:=wsw^{-1}$ is called a \textit{reflection} of $\W[\Gamma]$. Indeed, observe that $r_{\beta}$ sends $\beta$ to its opposite vector $-\beta$ and that it fixes pointwise the hyperplane of vectors of $V$ that are orthogonal to $\beta$ with respect to the bilinear symmetric form $\langle\cdot,\cdot \rangle$. The conjugates of the generators of $\W[\Gamma]$ are in bijection with the positive (equivalently, negative) roots.\medskip\\
\noindent The Coxeter group $\W[\Gamma]$ naturally acts on the root system $\Phi[\Gamma]$ through the canonical representation. Moreover, such an action preserves the bilinear form; i.e., $\langle w(\beta),w(\gamma)\rangle=\langle\beta,\gamma\rangle$ for all $w\in \W[\Gamma]$ and $\beta,\gamma\in \Phi[\Gamma]$. Another classical result is that $\Phi[\Gamma]$ is finite if and only if $\W[\Gamma]$ is finite \cite{deodh82}.

\subsection{\texorpdfstring{Semidirect decomposition of $\VA[\Gamma]$}{Semidirect decomposition of VA}}\label{subs-decomp-VA}

\noindent The standard reference on virtual Artin groups is \cite{BellParThiel}, of which we summarize the essential information here. Virtual Artin groups were introduced in Definition~\ref{def:VA}. We recall that, given $\Gamma$ a Coxeter graph, the natural map $\iota_{\W}:\W[\Gamma]\longrightarrow\VA[\Gamma]$ sending the standard generator $s$ to the corresponding virtual generator $\tau_s$ is an injective homomorphism. We often identify $\W[\Gamma]$ with its image $\iota_{\W}(\W[\Gamma])$ in $\VA[\Gamma]$. Now, let $\pi_K:\VA[\Gamma]\longrightarrow \W[\Gamma]$ be the map defined by $\pi_K(\sigma_s)=\idW$ and $\pi_K(\tau_s)=s$ for all $s\in S$. It is possible to check that $\pi_K$ is a surjective group homomorphism projecting the virtual Artin group onto its associated Coxeter group. We call the kernel $\ker(\pi_K)=:\KVA[\Gamma]$. The map $\iota_{\W}$ is a right section of $\pi_K$, which yields a split short exact sequence
\[
\begin{tikzpicture}[node distance=2cm, auto]
  \node (uno) {$\{\id\}$};
  \node (kva) [right of=uno] {$\KVA[\Gamma]$};
  \node (va) [right of=kva] {$\VA[\Gamma]$};
	\node (w) [right of=va] {$\W[\Gamma]$};
	\node (unodue) [right of=w] {$\{\id\}.$};
 \draw[->] (uno) to node {}  (kva);
 \draw[->] (kva) to node {$i$} (va);
\draw[transform canvas={yshift=0.5ex},->] (va) - -(w) node[below,midway] {$\pi_K$};
\draw[->] (w) to[bend right] node[midway,above,inner sep=2pt] {$\iota_{\W}$} (va); 
\draw[->] (w) to node {} (unodue);
\end{tikzpicture} \]
\noindent Therefore, the Coxeter group $\W[\Gamma]$ acts on the kernel $\KVA[\Gamma]$ by conjugacy, and $\VA[\Gamma]$ has the semidirect product structure $\VA[\Gamma]=\KVA[\Gamma]\rtimes \W[\Gamma]$.\\
\begin{nt}\label{notationactionWonKVA}
    For $k\in \KVA[\Gamma]$ and $w\in \W[\Gamma]$, we denote by $k^{w}$ the element $\iota_{\W}(w)\,k\,\iota_{\W}(w)^{-1}$.
\end{nt}
\begin{nt}\label{nt-g=(k,w)}
Let $g$ be in $\VA[\Gamma]=\KVA[\Gamma]\rtimes\W[\Gamma]$. We write $g=(k,w)$ where $k\in\KVA[\Gamma]$ and $w\in \W[\Gamma]$, and we call $k$ and $w$ the \textit{Artin component} and the \textit{Coxeter component} of $g$, respectively. Observe that given $g_1=(k_1,w_1)$, $g_{2}=(k_2,w_2)$, their product is $g_1\,g_2=(k_1,w_1)(k_2,w_2)=(k_1\,k_2^{w_1},w_1w_2)$. With this notation, 
\begin{equation}\label{eq-nt-kw}
  wk=(\idK,w)(k,\idW)=(k^w,w) \in \VA[\Gamma].  
\end{equation}
\noindent The identity $\id\in \VA[\Gamma]$ is denoted by $(\idK,\idW)$.
\end{nt}\bigskip
\noindent The reason why we call $k\in \KVA[\Gamma]$ an ``Artin component" will be clearer after Theorem \ref{thm-KVA-iso-Agammahat}. The normal subgroup $\KVA[\Gamma]$ is crucial for the comprehension of $\VA[\Gamma]$. In the remainder of this section, we recollect the results on $\KVA[\Gamma]$ that will be used throughout the work.\medskip  
\begin{defn}\label{defGammaHat} Let $\Gamma$ be a Coxeter graph on a vertex set $S$. Define a new Coxeter graph $\hGamma$ as follows:
\begin{enumerate}
    \item[(1.)] $V(\hGamma)=\Phi[\Gamma]$;
    \item[(2.)] for $\beta,\gamma\in \Phi[\Gamma]$, there is an edge $\{\beta,\gamma\}$ in $\hGamma$ if and only if there exist $a,b\in S$ and $w\in \W[\Gamma]$ such that $\beta=w(\alpha_a)$ and $\gamma=w(\alpha_b)$, in which case the label $\hmbg$ is $m_{a,b}$. 
\end{enumerate}
\end{defn}
\noindent  By \cite{BellParThiel}, the label $\widehat{m}_{\beta,\gamma}$ does not depend on the choice of $w$, $a$ and $b$ in (2.). When we consider the Artin group associated with $\hGamma$, we denote the generating set of $\A[\hGamma]$ simply by $\{\beta\mid\beta\in\Phi[\Gamma]\}$. Observe that $\A[\hGamma]$ is finitely generated if and only if $\Phi[\Gamma]$ is finite, equivalently, if and only if $\W[\Gamma]$ is a finite group.\\
\\
\noindent Consider now the decomposition $\VA[\Gamma]=\KVA[\Gamma]\rtimes \W[\Gamma]$. For any root $\beta=w(\alpha_s)$, set $\delta_{\beta}:= \sigma_s^w\in \KVA[\Gamma]$. By \cite[Lemma 2.2]{BellParThiel}, the definition of $\delta_{\beta}$ does not depend on the choice of $w$ and $s$ such that $\beta=w(\alpha_s)$. The main result concerning $\KVA[\Gamma]$ is the following.

\begin{thm}\cite[Theorem 2.3]{BellParThiel} \label{thm-KVA-iso-Agammahat}
  The kernel $\KVA[\Gamma]$ is generated by $\{\delta_{\beta}\mid \beta \in \Phi[\Gamma]\}$, and the map $\beta \longmapsto \delta_{\beta}$ induces an isomorphism between $\A[\hGamma]$ and $\KVA[\Gamma]$.  \medskip
\end{thm}

\noindent Henceforth, we often identify $\A[\hGamma]$  and $\KVA[\Gamma]$, and the generators of $\A[\hGamma]$ with $\{\delta_{\beta}\mid \beta \in \Phi[\Gamma]\}$.
With this notation, the action of $\W[\Gamma]$ on these generators coincides with its natural action on the root system. More precisely, for all $w\in \W[\Gamma]$ and all $\beta\in \Phi[\Gamma]$, we have
\begin{equation}\label{actionWonKVA}
\delta_{\beta}^w=\iota_{\W}(w)\,\delta_{\beta}\,\iota_{\W}
(w)^{-1}=\delta_{w(\beta)}.
\end{equation}

\noindent If $k\in \A[\hGamma]$, then write $k=\delta_{\beta_1}^{\varepsilon_1}\cdots \delta_{\beta_p}^{\varepsilon_p}$ with $\beta_i\in \Phi[\Gamma]$ and $\varepsilon_i\in\mathbb Z$ for all $i=1,\ldots,p$. The element $k^{w}$ is $k^w=\delta_{w(\beta_1)}^{\varepsilon_1}\cdots \delta_{w(\beta_p)}^{\varepsilon_p}$.\\
\\
\noindent The virtual Artin group $\VA[\Gamma]$ can therefore be written as $\A[\hGamma]\rtimes\W[\Gamma]$. Observe that, when $\Gamma$ is of spherical type, $\W[\Gamma]$ is finite and the virtual Artin group $\VA[\Gamma]$ \textit{is virtually an Artin group}, meaning that there is a finite index subgroup $\A[\hGamma]$ of $\VA[\Gamma]$ that is an Artin group. For any other $\Gamma$, the kernel $\A[\hGamma]$ has not finite index in $\VA[\Gamma]$.\medskip\\
\noindent A deeper understanding of the Artin group $\A[\hGamma]$ will allow us to derive structural properties of the whole virtual Artin group $\VA[\Gamma]$. Specifically, in Subsection \ref{subs-description-gammaHat} we study the graph $\hGamma$ when $\Gamma$ is the dihedral Coxeter graph with label $m$. This description allows us to see the Artin group $\A[\hGammam]$ as the fundamental group of a certain polygon of groups (see Subsection \ref{sec-poly-of-groups-ArtinGammaHat}).

\section {Preliminaries on complexes of groups}\label{sec-prelim-complexes}
\noindent Coxeter and Artin groups are often studied through their actions on certain remarkable cell complexes. In this section, we describe the \textit{Davis complex} and the \textit{Deligne complex}, which are respectively associated with $\W[\Gamma]$ and $\A[\Gamma]$. Motivated by these constructions, we introduce in Section \ref{sec-virtualDeligne} the notion of a \textit{virtual Deligne complex} for a virtual Artin group $\VA[\Gamma]$. These spaces can be defined in terms of simple complexes of groups, a structure that we recall in the following subsection.

\subsection{Simple complexes of groups}\label{subs-simplecomplexesofgr}
 
\noindent In this subsection, we recall standard results about simple complexes of groups, see \cite[II.12]{BriHaefl} for a more general treatment. We will be considering simple complexes of groups over a partially ordered set (or poset).

\begin{defn}
A \textit{simple complex of groups} $G(\mathcal{Q})=(G_q,\psi_{pq})$ over a poset $(\mathcal{Q},<)$ is the data of:
\begin{itemize}
    \item for each $q\in \mathcal{Q}$, a group $G_{q}$, called the \textit{local group} at $q$;
    \item for each $q < p$, an injective homomorphism $\psi_{pq}:G_{q}\xhookrightarrow{\quad} G_{p}$ such that if $r <q<p$, then $\psi_{pr}=\psi_{pq}\circ \psi_{qr}$.
\end{itemize}
\end{defn}
\noindent Note that in the literature, a complex of groups is sometimes defined using a contravariant convention, i.e. for each $q < p$, an injective homomorphism $\psi_{pq}:G_{q}\xhookrightarrow{\quad} G_{p}$. This is purely a matter of convention, up to considering the opposite poset. The convention we are adopting here is more natural for the complexes of groups we will be considering, and coincides with the convention used in the literature on Artin groups, see for instance \cite{CharDav95}.\medskip\\
\noindent Simple complexes of groups generalise to higher-dimension the notion of amalgamated product of groups, and can be used to encode actions on simplicial complexes with a strict fundamental domain. We recall that given a group $G$ acting on a simplicial complex $X$, a subcomplex $Y\subseteq X$ is called a \textit{strict fundamental domain} if it contains exactly one point from each $G$-orbit. The complexes we will be considering in this article are geometric realisations of posets. We recall that given a poset $\mathcal{Q}$, its \textit{derived complex} $\mathcal{Q}'$ is the abstract simplicial complex whose $k$-simplices correspond to the chains $q_0<\cdots < q_k$ with $q_0, \ldots, q_k \in \mathcal{Q}$. We denote by $|\mathcal{Q}'|$ the \textit{geometric realisation} of this abstract simplicial complex, which we will simply refer to as the geometric realisation of $\mathcal{Q}$. 

\begin{defn}
Let $G(\mathcal{Q})$ be a simple complex of groups over a poset $\mathcal{Q}$ with a simply connected geometric realisation. The \textit{fundamental group} of $G(\mathcal{Q})$ is defined as: \[
\pi_1(G(\mathcal {Q})):=\underset{q \in \mathcal{Q}}{\varinjlim}\,G_{q},
\]
\noindent which is the direct limit of the system of groups and monomorphisms $(G_{q},\psi_{pq})$, that is, the quotient of the free product of all the $G_q$ by adding the relations $a=\psi_{pq}(a)$ for all $q<p$ and $a\in G_q$. For each $q\in \mathcal{Q}$, we denote by $i_q: G_q \rightarrow \pi_1(G(\mathcal {Q}))$ the natural map sending $G_q$ to the image of the corresponding free factor. Note that $i_q$ may not be injective in general. 
\end{defn}

\noindent The following is a powerful theorem  that allows us to construct group actions out of local data.

\begin{thm} \cite[Theorem 12.18 and 12.20]{BriHaefl}\label{thm-strictly-developable-scog}  Let $G(\mathcal{Q})$ be a simple complex of groups over a poset $\mathcal{Q}$ with a simply connected geometric realisation. Suppose that all the morphisms $i_q: G_q\longrightarrow \pi_1(G(\mathcal{Q}))$ are injective. Let $G\coloneqq \pi_1(G(\mathcal{Q}))$. We now identify each $G_q$ with its image in $G$ under $i_q$.\medskip\\
\noindent Let $D(\mathcal{Q}, i)$ be the poset whose elements are of the form $(gG_q, q)$, for $q\in \mathcal{Q}$, $g \in G$  (i.e. $gG_q$ is a left coset of $G_q$ in $G$), and with partial order $(gG_q, q)< (g'G_{q'}, q')$ when $q<q'$ and $g^{-1}g'\in G_{q'}$. The group $G$ acts on $D(\mathcal{Q}, i)$ by left multiplication via $h\cdot (gG_q, q) = (hgG_q, q)$. Let $X$ denote the simplicial realisation of $D(\mathcal{Q}, i)$. Note that $\mathcal{Q}$ embeds in $D(\mathcal{Q}, i)$ via the map $q \mapsto  (G_q, q)$. We therefore identify $\mathcal{Q}$ with a sub-poset of $D(\mathcal{Q}, i)$, and $Y \coloneqq |\mathcal{Q}|$ with a sub-complex of $X$. \medskip\\
\noindent Then $X$ is simply connected, and $G$ acts on $X$ with strict fundamental domain $Y$.  Moreover, for each $q\in \mathcal{Q}\subset D(\mathcal{Q}, i)$, the stabiliser $\mathrm{Stab}_G(q)$ is equal to $G_q$, and for each $q' < q$, the inclusion $\mathrm{Stab}_G(q') \hookrightarrow \mathrm{Stab}_G(q)$ is the local map $\psi_{q,q'}: G_{q'}\rightarrow G_q$.
\end{thm}

\begin{defn} A complex of groups $G(\mathcal{Q})$ satisfying the hypotheses of Theorem~\ref{thm-strictly-developable-scog} is called strictly developable, and $D(\mathcal{Q}, i)$ is called its \textit{universal cover}. Since we are interested in actions on simplicial complexes, with a slight abuse of notation we will often also refer to $X$ as the universal cover of $G(\mathcal{Q})$.
\end{defn}

\noindent The complexes on which we make Coxeter, Artin and virtual Artin groups act, arise as universal covers of strictly developable simple complexes of groups.

\subsection{The Davis complex of a Coxeter group and the Deligne complex of an Artin group}\label{subs-davis-and-deligne-complex}
\noindent In this subsection, we recall the construction of complexes of groups associated to Coxeter and Artin groups. These constructions being very similar, we will adopt a uniform perspective and will use the notation $G[\Gamma]$ with $G = \W$ or $\A$ to refer to either a Coxeter or an Artin group. Given a Coxeter graph $\Gamma$ with set of vertices $S$, if $X\subset S$, the subgroup $G_X[\Gamma]$ of $G[\Gamma]$ generated by $X$ is called a \textit{standard parabolic subgroup} of the Coxeter or Artin group $G[\Gamma]$. A conjugate $wG_X[\Gamma]\,w^{-1}$ of a standard parabolic subgroup is called simply a \textit{parabolic subgroup}. Denote by $\Gamma_X$ the full subgraph of $\Gamma$ spanned by the vertices in $X$. The graph $\Gamma_X$ is itself a Coxeter graph, and we denote its associated Coxeter and Artin groups by $\W[\Gamma_X]$ and $\A[\Gamma_X]$, respectively. \medskip\\
\noindent By the classical theory of Coxeter groups, $\W_X[\Gamma]$ is isomorphic to $ \W[\Gamma_X]$. More specifically, the natural map $\W[\Gamma_X]\longrightarrow \W[\Gamma]$ is an injective homomorphism and its image is $\W_X[\Gamma]$ \cite[Ch IV, 1.8, Thm 2]{Bourbaki}. Like for Coxeter groups, the natural map $\A[\Gamma_X]\longrightarrow \A[\Gamma]$ is injective, i.e. $\A_X[\Gamma]\cong \A[\Gamma_X]$ \cite[Theorem 4.13]{van1983homotopy}. Thus we write $G[\Gamma_X]\xhookrightarrow{}  G[\Gamma]$ for $G\in \{\W,\A\}$. \medskip
\\
\noindent  For the construction of the complexes of our interest, we consider the following family of subsets of $S$:
\[
\Sf:=\{X\subset S\mid \W_X[\Gamma]\;\, \text{is finite}\},
\]
\noindent ordered by inclusion. Set now $K_{\Gamma}:=|(\Sf)'|$. Observe that $\Sf$ has an initial object $\varnothing$, thus $K_{\Gamma}$ is simply connected. We now define a complex of groups over $\Sf$ as follows: 
\begin{itemize}
    \item For each $X\in \Sf$, the associated local group is the standard parabolic subgroup $G_X[\Gamma]$.
    \item The local maps are the inclusions of standard parabolic subgroups (see \cite[Ch IV, 1.8, Thm 2]{Bourbaki} for $G=\W$, see \cite[Theorem 4.13]{van1983homotopy} for $G=\A$).
\end{itemize}
\noindent It follows from the injectivity of the inclusions of parabolic subgroups $G_X[\Gamma]\xhookrightarrow{} G[\Gamma]$ and from Theorem \ref{thm-strictly-developable-scog} that this defines a strictly developable simple complex of groups $G(\Sf)$. Furthermore, $\pi_1(G(\Sf))=G[\Gamma]$. 
Its universal cover is the geometric realisation of the  poset 
\[
G\Sf:=\{gG_X[\Gamma]\mid X\in \Sf,\;g\in G[\Gamma]\},
\]
\noindent ordered by inclusion.\medskip
\\
\noindent For $G=\W$, the complex $\Sigma_{\Gamma}:=|(\W\Sf)'|$ is called the \textit{Davis complex} of the Coxeter group $\W[\Gamma]$. The face preserving action of $\W[\Gamma]$ on $\Sigma_{\Gamma}$ is exactly the order preserving action of $\W[\Gamma]$ on $\W\Sf$ by left multiplication. This action is cocompact since the fundamental domain $K_{\Gamma}$ is compact, and it is proper since the stabilizers of the vertices are the finite parabolic subgroups  $w\W_X[\Gamma]w^{-1}$, with $X\in \Sf$.\medskip
\\
\noindent When $G=\A$, the complex $\D_{\Gamma}:=|(\A\Sf)'|$, called the \textit{(modified) Deligne complex} associated with $\A[\Gamma]$. Thus, the Artin group $\A[\Gamma]$ acts cocompactly and simplicially on $\D_{\Gamma}$, with an action given by the left multiplication on the partially ordered set of cosets. However, the stabilizer of the vertex $a\,\A_X[\Gamma]$ of $\D_{\Gamma}$ is the infinite parabolic subgroup $a\,\A_X[\Gamma]\,a^{-1}$, which implies that the action is not proper. \medskip\\
\noindent Observe that the dimension of $\D_{\Gamma}$ is the maximal cardinality of a subset of $S$ that gives rise to a spherical type standard parabolic subgroup of $\A[\Gamma]$. We call \textit{dimension} of an Artin group $\A[\Gamma]$ (or, with a slight abuse of notation, of the Coxeter graph $\Gamma$) the dimension of the associated Deligne complex. Remarkable Artin groups are those \textit{of dimension 2}, where the maximal-dimensional cells in $\D_{\Gamma}$ are 2-simplices. \medskip
\\
\noindent  The $K(\pi,1)$-conjecture for Artin groups (see \cite{Par14} for a survey) is equivalent to the contractibility of the Deligne complex $\D_\Gamma$ for all Coxeter graphs $\Gamma$, by work of Charney--Davis \cite{CharDav95}. A very common approach to this problem is showing that the Deligne complex $\D_{\Gamma}$ always supports a CAT(0) metric. \medskip

\subsection{\texorpdfstring{The Moussong metric on $\Sigma_\Gamma$ and $\D_\Gamma$.}{The Moussong metric on Sigma and D.}}\label{subs-moussongmetric}

\noindent
Moussong~\cite{Moussong} defined a CAT(0) a piecewise Euclidean metric on the Davis complex $\Sigma_{\Gamma}$. This is a piecewise Euclidean metric obtained by first defining the metric on the strict fundamental domain $K_\Gamma$, and then extending it equivariantly to the whole space. Note that since $K_\Gamma$ is also a fundamental domain for the Deligne complex $\D_\Gamma$, a choice of piecewise Euclidean metric on $K_\Gamma$ similarly induces a metric on the entire $\D_\Gamma$. We follow a similar approach in the next section to define a Moussong metric on the virtual Deligne complex of a virtual Artin group.
\medskip
\\
\noindent For the reader's convenience, we recall here the construction of the Moussong metric. However, we emphasise that the only place in the article where this construction is necessary is the proof of Corollary \ref{cor-CAT0-locred}.\medskip
\\
\noindent The Moussong metric is preserved by the action of the Coxeter group, thus it suffices to describe the piece-wise Euclidean structure on $K_{\Gamma}$. View $K_{\Gamma}$ as a cubical (instead of simplicial) complex, where for $X\in \Sf$, the vertices $Y\subset X$ of $K_{\Gamma}$ span a cube of dimension $|X|$, denoted $cube(X)$. Let $K_{\Gamma_{\geq X}}$ denote the subcomplex of $K_{\Gamma}$ spanned by the vertices $Z\in \Sf$ with $X\subseteq  Z$. The Euclidean metric on $cube(X)$ is described below. \medskip\\
\noindent The finite group $\W_X[\Gamma]$ naturally realizes as a group of orthogonal transformations of $\mathbb{R}^{|X|}$ through the canonical representation described in Subsection \ref{subs-rootsystems}. Each generator $s\in X$ acts as an orthogonal reflection with respect to the hyperplane $\langle\alpha_s\rangle^{\perp}$, and the intersection of all the positive half-spaces with respect to these hyperplanes is a simplicial convex cone denoted by $T_X$. There is a unique point $x_0$ in $T_X$ that is at distance 1 from each such hyperplane. The \textit{Coxeter polytope} $C[X]$ is the convex hull of the $\W_X[\Gamma]$-orbit of $x_0$. For any $Y\subset X$, we can consider the convex hull of all the $\W_Y[\Gamma]$-orbits of $x_0$ together with all its $\W_X[\Gamma]$-translates. The $u$-translate of the $\W_Y[\Gamma]$-orbit of $x_0$ is denoted by $uF(Y)$, for $u\in \W_X[\Gamma]$. By \cite[Lemma 7.3.3]{Davis2008}, the set $\{uF(Y)\mid Y\subset X;\,u\in \W_X[\Gamma]\}$ coincides with the set of faces of $C[X]$. Another fact is that, calling $F^*(Y)$ the face of $T_X$ fixed by $\W_{Y}[\Gamma]$, then $F(Y)$ and $F^*(Y)$ are orthogonal and intersect at a single point $x_Y$.\medskip\\
\noindent The intersection between $C[X]$ and the cone $T_X$ is combinatorially a cube with vertices $\{x_Y\}_{Y\subseteq X}$. The Euclidean structure on $cube(X)$ is defined by identifying it with $C[X]\cap T_X$, so that the vertex corresponding to $Y\subseteq X\in \Sf$ in $K_{\Gamma}$ is identified to the vertex $x_Y\in C[X]\cap T_X$ (see Figure \ref{fig:MoussongMetric}).\medskip\\
\begin{figure}[h!]
\centering
\begin{minipage}{4cm}
  \centering
\includegraphics[width=4cm]{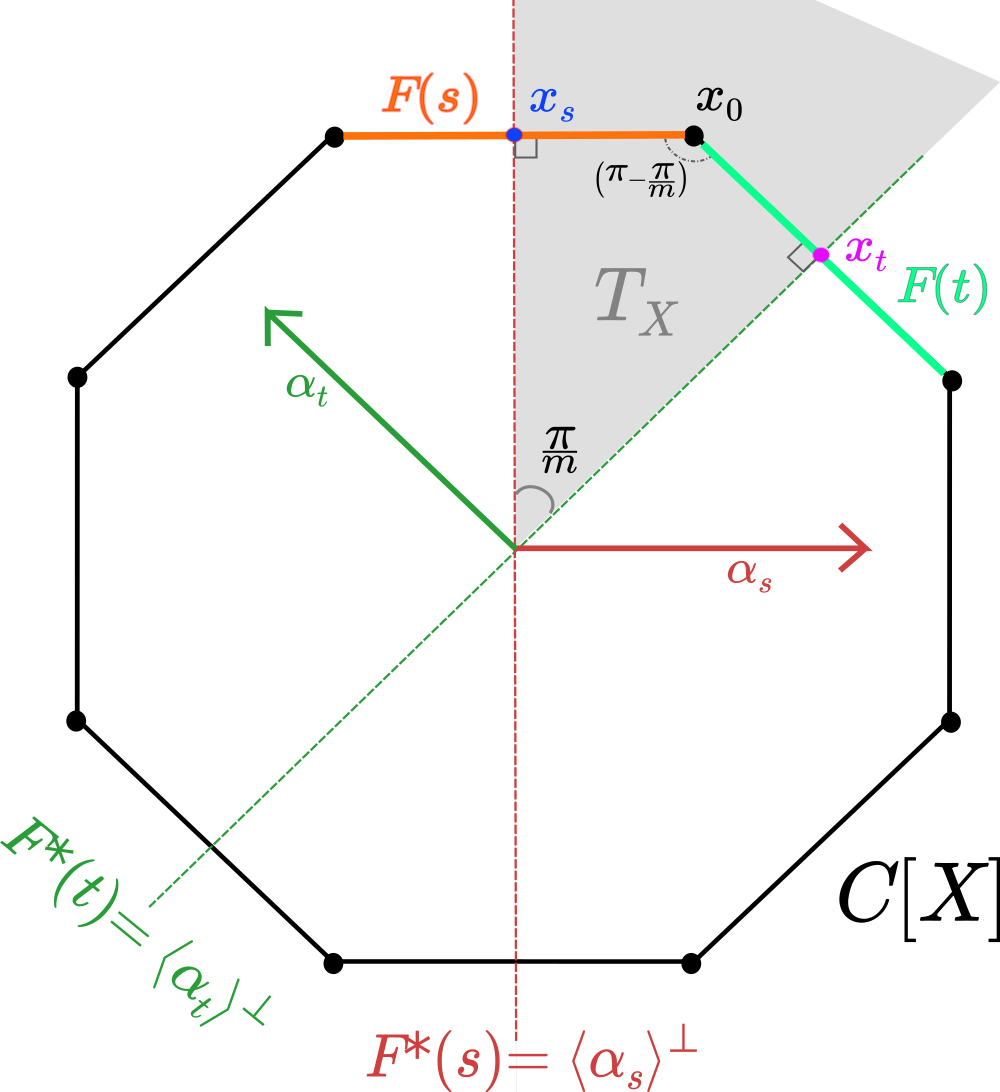}
\end{minipage}
\qquad\qquad\qquad
\begin{minipage}{5cm}
  \centering
\includegraphics[width=4cm]{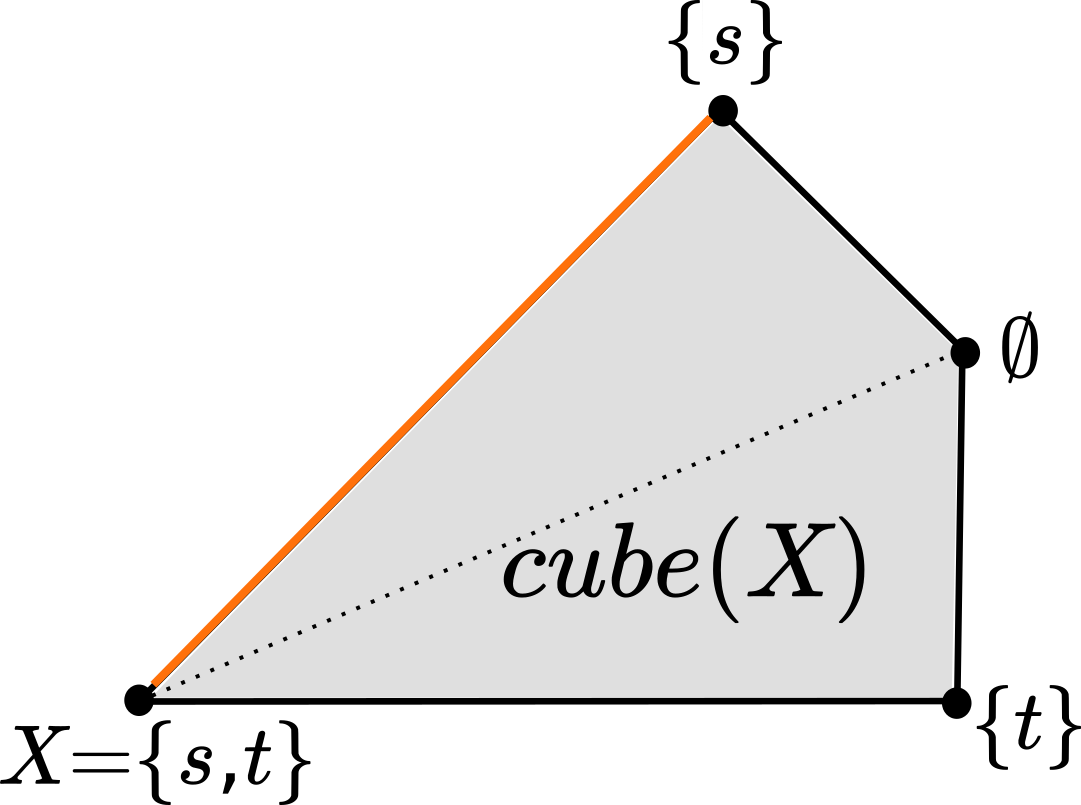}
\end{minipage}
\caption{The Moussong metric on a cube spanned by $\varnothing$, $\{s\}$, $\{t\}$ and $\{s,t\}=X\in \Sf$.}
\label{fig:MoussongMetric}
\end{figure}
\noindent If $Y\subseteq X\in \Sf$, then $F(Y)$ is isometric to the Coxeter cell $C[Y]$, so that the face of $cube(X)$ spanned by $\varnothing$ and $Y$ in $K_{\Gamma}$ is isometric to $cube(Y)$. Then the metrics on the cubes are compatible and they endow $K_{\Gamma}$ of a piece-wise Euclidean structure, which is inherited by the entire Davis complex $\Sigma_{\Gamma}$ under isometric action of the Coxeter group $\W[\Gamma]$. This metric on $K_{\Gamma}$ is called the \textit{Moussong metric} and it will be denoted by $d_M$.\medskip
\\
\noindent The complexes that we introduced in Subsection \ref{subs-davis-and-deligne-complex} are defined as complexes of groups on $\Sf$. The metric that we consider on them (unless differently specified) is the piece-wise Euclidean metric inherited by the Moussong metric $d_M$ on the fundamental domain $K_{\Gamma}$. 
\section{The virtual Deligne complex}\label{sec-virtualDeligne}

In this section we introduce a virtual version of the Deligne complex on which $\VA[\Gamma]$ acts simplicially, cocompactly and without inversions. The metric with which we endow the virtual Deligne complex $\vD_{\Gamma}$ is the Moussong metric described in Subsection  
\ref{subs-moussongmetric}. Inspired by analogous results for classical Artin groups, we show that $\vD_{\Gamma}$ is CAT(0) for a certain class of Coxeter graphs. More specifically, in Subsection \ref{subs-CAT0-2dim} we show that the virtual Deligne complex is CAT(0) when the Coxeter graph has dimension 2 (assuming the Key Theorem~\ref{virtualAppelSchupp}, whose proof is postponed to the following sections), while in Subsection \ref{subs-from2tolocallyred} we extend this result to a wider class of graphs called \textit{locally reducible}, by adapting a proof of Charney \cite{Charney2000}. The consequences of these results are analysed in Sections \ref{sec-fin-subg} and \ref{sec-cohom+classifying}.
\subsection{\texorpdfstring{Definition of $\vD_{\Gamma}$}{Definition of vD}} \label{subs-defnofVD}

\noindent As for the Davis and the Deligne complexes, our complex is defined as the universal cover of a complex of groups over the poset $\Sf$ of spherical subsets
of $S$. For this purpose, we first introduce the notion of standard parabolic subgroups for virtual Artin groups.
\begin{defn}
    Let $X\subset S$, and set $\mathcal{S}_X=\{\sigma_s\;|\; s\in X\}$ and $\mathcal{T}_X=\{\tau_s\;|\; s\in X\}$. The subgroup of $\VA[\Gamma]$ generated by $\mathcal{S}_X\sqcup \mathcal{T}_X$ is called a \textit{standard parabolic subgroup of $\VA[\Gamma]$}, and it is denoted by $\VA_X[\Gamma]$. Given $g\in \VA[\Gamma]$ and $X\subset S$, the conjugate $g\,\VA_X[\Gamma]\, g^{-1}$ is called a \textit{parabolic subgroup} of $\VA[\Gamma]$.
\end{defn}
\noindent Similarly to Coxeter groups and Artin groups, the following result holds.
\begin{thm}\cite[Thm 1.1]{GaGaPa26}
    Let $X \subseteq S$. The natural homomorphism $\VA[\Gamma_X] \longrightarrow \VA_X [\Gamma]$ that maps $\sigma_x$ to $\sigma_x$ and $\tau_x$ to $\tau_x$ for all $x\in X$ is an isomorphism.
\end{thm}
\noindent
Thus, we can identify $\VA[\Gamma_X]$ with $\VA_X[\Gamma]$, where $\Gamma_X$ is the full subgraph of $\Gamma$ spanned by the vertices in $X$.

\begin{nt}
    When it is clear from the context, we usually omit the Coxeter graph $\Gamma$, writing for instance $\VA$ and $\VA_X$ instead of $\VA[\Gamma]$ and $\VA_X[\Gamma]$. When $X\subset S$ is a singleton, $X=\{s\}$, we write $G_s$ instead of $G_{\{s\}}$, for $G\in \{\W,\A,\VA\}$.
\end{nt}

\noindent     Recall that $\Sf:=\{X\subset S\mid \W_X[\Gamma]\,\text{ is finite}\}$ is a poset for the inclusion, and that $K_{\Gamma}=|(\Sf)'|$.

\begin{defn}
We define the following simple complex of groups over $\Sf$ as follows: 
\begin{itemize}
    \item For each $X\in \Sf$, the associated local group is the parabolic subgroup $\VA_X[\Gamma]$.
    \item The local maps are the inclusions of standard parabolic subgroups \cite[Thm 1.1]{GaGaPa26}.
\end{itemize}
\end{defn}
\noindent By construction, we have that $\pi_1(\VA(\Sf))=\VA[\Gamma]$. As before, the injectivity of the maps $\VA_X\xhookrightarrow{}\VA$ guarantees that this complex of groups is strictly developable. 

\begin{defn}
The \textit{virtual Deligne complex}, denoted by $\vD_{\Gamma}$, is the universal cover of the simple complex of groups $\VA(\Sf)$.
\end{defn}
\noindent The virtual Deligne complex is a simply connected simplicial complex on which $\VA[\Gamma]$ acts cellularly and cocompactly. The poset from which it arises is 
    \[
   \VA\Sf:= \{g\,\VA_X\mid g \in \VA, \;X\in \Sf\},
    \]
    ordered by inclusion. Thus, $\vD_{\Gamma}=|(\VA\Sf)'|$, and the action of $\VA[\Gamma]$ on $\vD_{\Gamma}$ is the action on $\VA\Sf$ by left multiplication. A strict fundamental domain for the action is $K_\Gamma$, identified with the sub-poset consisting of the $\VA_X$ for $X\in \Sf$. The stabilizer of the vertex $g\VA_X$ of $\vD_{\Gamma}$ is the parabolic subgroup $g\VA_X \,g^{-1}$, which is infinite. \medskip \\
\noindent As in the Coxeter and Artin case, we use the Moussong metric on $K_\Gamma$ to define a  Moussong metric on $\vD_{\Gamma}$. The following question now naturally arises.
\begin{quest}
    When is the virtual Deligne complex $\vD_{\Gamma}$ CAT(0) for the Moussong metric?
\end{quest}
\noindent  As in the Artin group case, we expect the non-positive curvature to hold in general (with respect to the Moussong metric) for all virtual Artin groups (see Conjecture \ref{conj-vD-CAT0} in the introduction).\medskip
\\
\noindent We now give a general result on the virtual Deligne complex which will be needed in the following sections. Recall that if a group $G$ acts on a polyhedral or simplicial complex, the action is said \textit{without inversion} if for a cell $\sigma$, the setwise stabilizer in $G$ equals  pointwise stabilizer. In other words, no directed cell is mapped to itself with a reversed orientation.
\begin{lem}\label{prop-vD-complete-action-no-inversion}
    For any Coxeter graph $\Gamma$, the virtual Deligne complex $\vD_{\Gamma}$ is a complete metric space and $\VA[\Gamma]$ acts on it without inversion.
\end{lem}
\begin{proof}
The virtual Deligne complex $\vD_{\Gamma}$ equipped with the Moussong metric is a piece-wise Euclidean metric space. Moreover, by definition of the metric, $\vD_{\Gamma}$ only contains a finite number of isometry classes of cells. The fact that $\vD_{\Gamma}$ is complete follows now from \cite[Theorem~7.13]{BriHaefl}. \medskip\\
The action of $\VA[\Gamma]$ on $\vD_{\Gamma}$ is the action by left multiplication on the cosets in $\VA\Sf$. Such an action preserves the vertex types, i.e. the cardinality of the corresponding standard parabolic subgroup. Since vertices in simplex of $\vD_\Gamma$ are cosets of parabolic subgroups of pairwise different cardinality, it follows that if a simplex of $\vD_{\Gamma}$ is fixed by an element of $\VA[\Gamma]$, then all its vertices are fixed pointwise. Thus, the action is without inversion.
\end{proof}

\subsection{\texorpdfstring{Non-positive curvature of $\vD_{\Gamma}$ for the 2-dimensional case}{Non-positive curvature of vD for the 2-dimensional case}}\label{subs-CAT0-2dim}

 \noindent Recall that we say that $\Gamma$ has dimension 2 if the maximal cardinality of a spherical type subset of $S$ is 2. The main result of this subsection is the following.

\begin{thm}\label{thm-2dim-vD-CAT(0)}
   Le $\Gamma$ be a Coxeter graph of dimension 2. Then the virtual Deligne complex $\vD_{\Gamma}$ associated with $\VA[\Gamma]$ is CAT(0). 
\end{thm}

\noindent We prove this result using Gromov's link condition, namely by showing that every loop in the link of a vertex has length greater or equal to $2\pi$. The same result holds for the classical Deligne complex $\D_{\Gamma}$ \cite{CharDav95}. A loop around a vertex in the Deligne (or, analogously, in the virtual Deligne) complex, corresponds to a specific element in the group. 
\begin{nt}
    We denote the free group on $n$ generators by $\mathbb{F}_n$. If $x_1,\ldots,x_n$ are the explicit generators of $\mathbb{F}_n$, we write $\mathbb{F}_n=\mathbb{F}(x_1,\ldots,x_n)$. 
\end{nt}
\begin{nt}\label{notation-freeproducts}
    Let $G$ be a group, and let $G_1,G_2$ be subgroups of $G$. For any element $\mu$ in the free product $G_1\frpp G_2$, we denote by $\ev_G(\mu)$ the evaluation of $\mu$ in the group $G$.  We say that an element $\mu\in G_1\frpp G_2$ \textit{represents the identity in $G$} if $\ev_G(\mu)=\id_G$. If $\mu=a_1b_1\ldots a_{\ell}b_{\ell}$, we say that $\mu$ is \textit{written in normal form} with respect to the decomposition $G_1\frpp G_2$ if $a_i\in G_1, \;b_{i}\in G_2$ for $i=1,\ldots,\ell$, and $a_i\neq \id\neq b_i$ for all $i=1,\ldots,\ell$, except for $a_1$ and $b_{\ell}$, which may be trivial. The factors $a_i, b_i$ are called \textit{syllables} of $\mu$. The \textit{syllabic length} of a word $\mu\in G_1\frpp G_2$ written in normal form $\mu=a_1b
    _1\cdots a_{\ell} b_{\ell}$ is $||\cdot||_{G_1\ast G_2}$, i.e. the number of non-trivial syllables in $\mu$. This naturally generalises to $n$ factors for $G_1,\ldots,G_n<G$ and an element $\mu\in G_1\frpp \cdots \frpp G_n$.
\end{nt}
\begin{rmk}\label{rmk_cyclically_reduced}
    If $\mu=a_1b_1\ldots a_{\ell} b_{\ell}$ is an element written in normal form in $G_1\frpp G_2$, then any cyclic permutation $\mu'$ of $\mu$ represents a conjugate of $\ev_G(\mu)$ in $G$. In this work we will often consider elements in free products of groups such that $\ev_G(\mu)=\id_G$, and this equality will always be understood up to cyclic permutation of the syllables of $\mu$.
\end{rmk}

\noindent When $\Gamma$ is the dihedral Coxeter graph with vertices $s$ and $t$, and a single edge labelled by the integer $m:=m_{s,t}\in \mathbb{N}_{\geq 2}$, we write $\Gamma=\Gamma_m$, and we also denote the associated dihedral Coxeter group, Artin group and virtual Artin group respectively by $\W_m, \A_m$ and $\VA_m$.\medskip\\
\noindent The key ingredient in \cite{CharDav95} to show that $\D_{\Gamma}$ is CAT(0) in the 2-dimensional case is the following theorem, proved by Appel and Schupp.

\begin{thm}\cite[Lemma 6]{AppelSchupp83}\label{AppelSchuppLemma}
Let $\Gamma_m$ be a dihedral Coxeter graph with vertex set $S=\{s,t\}$ and $m\geq 2$. Let $G_m$ be either $\W_m$ or $\A_m$, and let $G_{r}$ be the associated standard parabolic subgroup on $\{r\}\subset \{s,t\}$. Let $\mu$ be a non-trivial element in $G_s\frpp G_t$ written in normal form and representing the identity in $G_m$. Then the syllabic length of $\mu$ satisfies $\|\mu\|_{G_s\ast G_t}\geq 2m$. 
\end{thm}

\noindent The proof of Theorem \ref{thm-2dim-vD-CAT(0)} deeply relies on an analogue of Appel--Schupp's result for virtual Artin groups. Indeed, controlling the length of a closed curve in the link of a vertex in $\vD_{\Gamma}$, equals controlling the syllabic length of a certain element written in normal form in a free product of subgroups.  \medskip\\
\noindent
Let $\Gamma$ be the dihedral Coxeter graph with vertices $S=\{s,t\}$ and label $m_{s,t}=:m$. Let $\omega$ be an element in the free product
$ \VA_s \frpp \VA_t$, written in normal form. 
\begin{align*}
    \VA_{s}\frpp \VA_{t} &\xrightarrow[]{\qquad} \VA_{\{s,t\}}=\VA_m\\
   \omega= g_1\cdots g_{\ell} &\longmapsto \ev_{\VA_m}({g_1\cdots g_{\ell}})=\ev_{\VA_m}({\omega}).
\end{align*}

\noindent In this context, for $\omega$ to be written in normal form it means that if $g_i\in \VA_{r_i}$, then $g_{i+1}\in \VA_{r_{i+1}}$ with $r_i\in \{s,t\}$, $r_i\neq r_{i+1}$ for all $i\in \mathbb{Z}_{/\ell}$. The following theorem, which we will refer to as ``virtual Appel--Schupp", is the key step for proving Theorem \ref{thm-2dim-vD-CAT(0)}.

\begin{thm}[``Virtual Appel--Schupp Theorem'']\label{virtualAppelSchupp}
Let $\Gamma_m$ be a dihedral Coxeter graph with vertex set $S=\{s,t\}$ and $m\geq 2$. Let $\VA_m$ be the associated virtual Artin group, and let $\omega$ be a non-trivial element in $\VA_s\frpp \VA_t$ written in normal form, representing the identity in $\VA_m$. Then the syllabic length of $\omega$ satisfies $\|\omega\|\geq 2m$. 
\end{thm}
  \noindent Theorem \ref{virtualAppelSchupp} is the key technical result from this work, whose proof is postponed to Sections~\ref{sectionDihedral}-\ref{sec-virtualAppelSchupp}.

\begin{proof}[Proof of Theorem \ref{thm-2dim-vD-CAT(0)}]  Observe that $\vD_{\Gamma}$ is simply connected by construction. To show that $\vD_{\Gamma}$ is CAT(0) when $\Gamma$ has dimension 2, we use Gromov's link criterion. We therefore study the length of an embedded loop in the link of a vertex of $\vD_{\Gamma}$.\medskip\\
\noindent The vertices of $\vD_{\Gamma}$ are of the form $v=g\VA_X$, where $g\in \VA[\Gamma]$ and $X\in \Sf$ can have cardinality 0, 1 or 2.  Up to the action of the group, it is enough to consider the case in which $g=\id$. See Figure~\ref{fig:ex-triangle} for a picture of a small portion of $\vD_\Gamma$ with the angles coming from the Moussong metric. 
\begin{figure}[h!]
\centering\includegraphics[width=5cm]{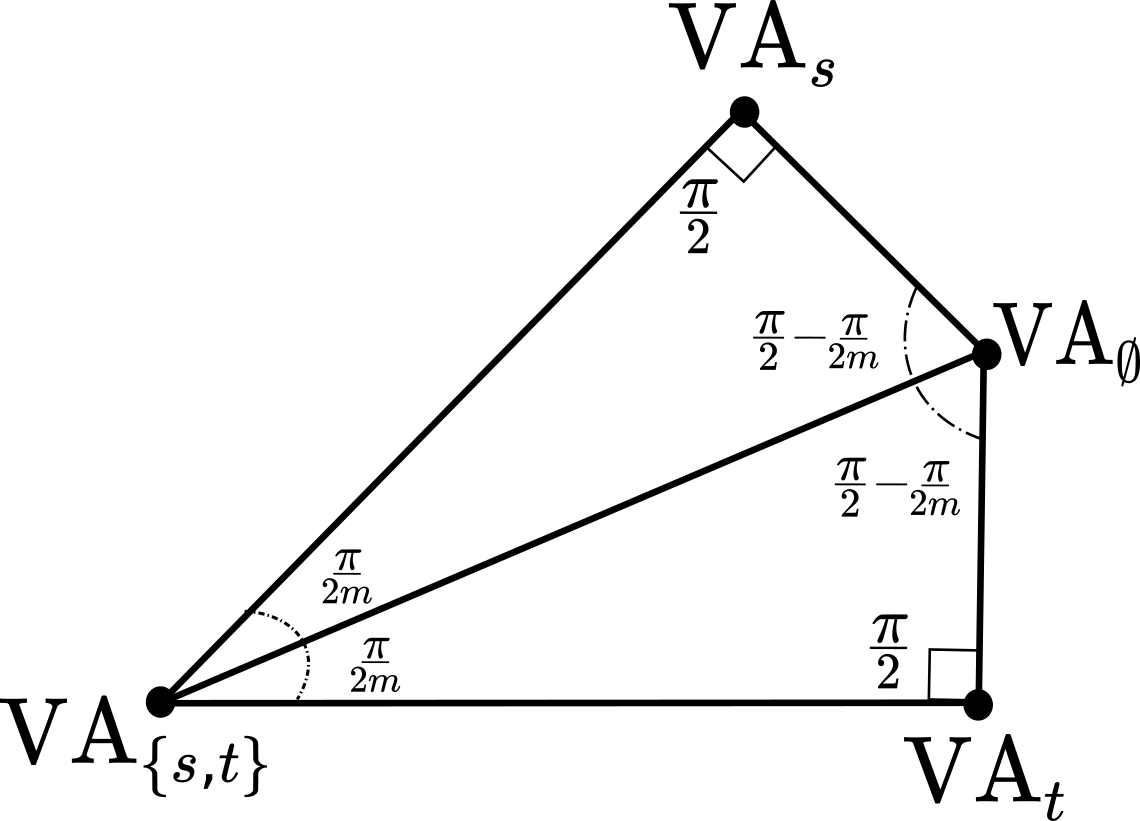}
    \caption{The simplices spanned by $\VA_{\varnothing},\VA_s,\VA_t$ and $\VA_{\{s,t\}}$ in $\vD_{\Gamma}$.}\label{fig:ex-triangle}
\end{figure}
\noindent If $X=\varnothing$, the vertex $v=\VA_X$ is $v=\VA_{\varnothing}=\{\id\}$. By construction of the complex of groups and its metric, the link of $\VA_\varnothing$ in $\vD_\Gamma$ is isometric to the link of $\A_\varnothing$ in $\D_\Gamma$, which is CAT(1) by \cite{CharDav95}.
\medskip
\\
\noindent  If $|X|=1$, then $X=\{s\}$ for some $s\in S$ and $v=\VA_s$. There are two types of vertices in the link of $\VA_s$: those of the form $s^i\VA_\varnothing$ for $i\in \mathbb Z$, and those of the form $\VA_{\{s,t\}}$ for $t\in S$ such that $\{s,t\}\in \Sf$. One sees that the link is a complete bipartite graph on these two sets. Since all edges have length $\pi/2$ by definition of the Moussong metric, and closed embedded loops in a bipartite simplicial graph contain at least four edges, we get that the minimal length of a loop in the link of $v$ is $2\pi$.\medskip
\\
\noindent Suppose now that $|X|=2$, so that $X=\{s,t\}$ and that $m_{s,t}\neq \infty$. Consider the vertex $v=\VA_{\{s,t\}}$ in $\vD_{\Gamma}$. A simple closed curve $\gamma$ in the link of $v$ alternates betweens vertices of type $0$ of the form $g\VA_\varnothing$, and vertices of type $1$ of the form $g\VA_r$ for $g\in \VA_{\{s,t\}}$ and $r\in \{s, t\}$. Let us denote cyclically $v_0, w_0, v_1, w_1, \ldots, v_\ell, w_\ell$ the vertices of $\gamma$, with the $v_i$ being the vertices of type $0$ in $\gamma$,  the $w_i$ the vertices of type $1$, and with $v_0=v_\ell$ and $w_0 = w_\ell$. For each $i$, let $r_i \in \{s, t\}$ such that $w_i$ corresponds to a coset $h_i\VA_{r_i}$ with $h_i\in \VA_{\{s,t\}}$. Note that, since $\gamma$ is embedded, we always have $r_{i+1}\neq r_i$.\medskip\\
\begin{figure}[h!]
    \centering
\includegraphics[width=7.5cm]{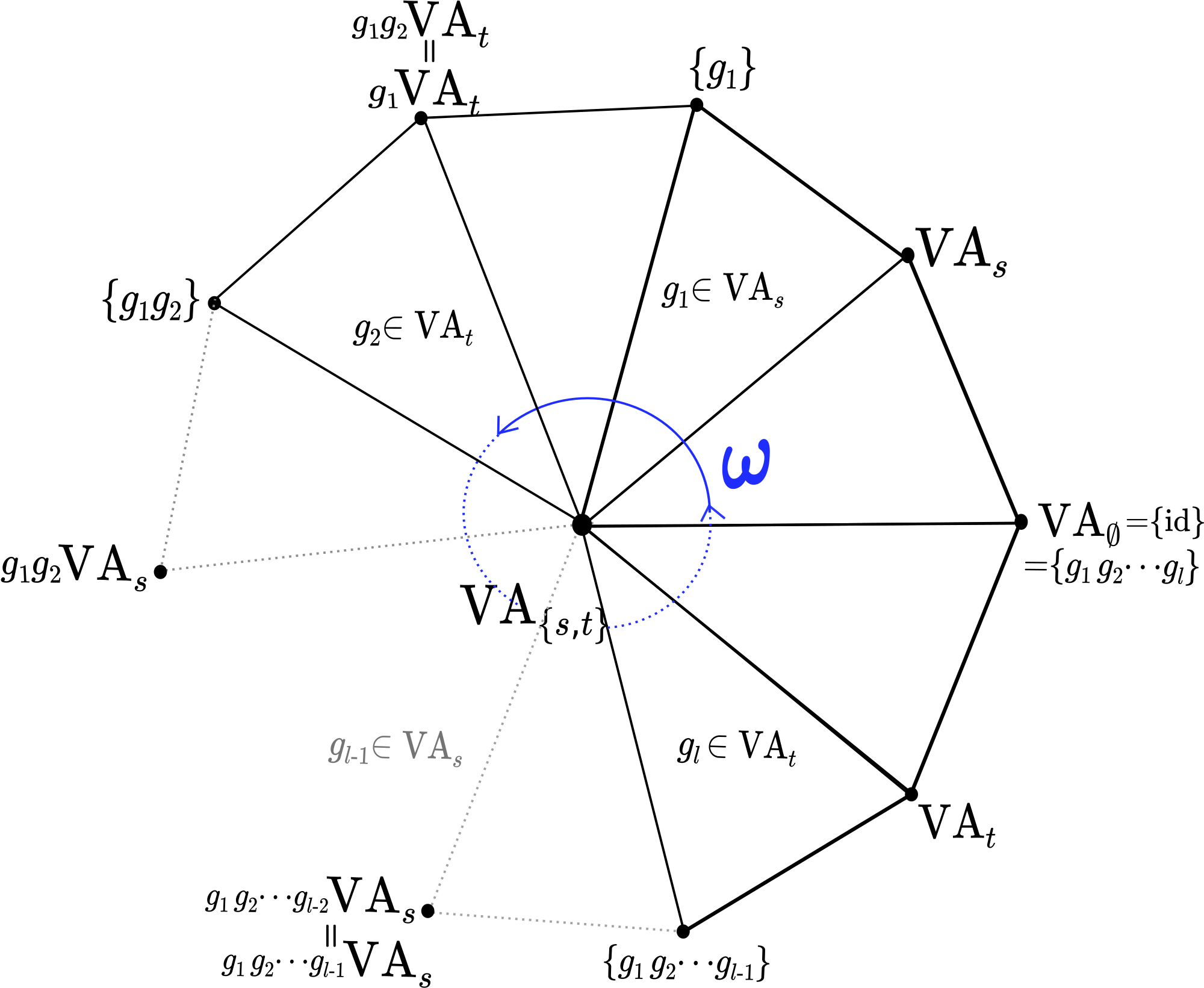}
    \caption{The loop of triangles around a vertex represented by the element $\omega=g_1\cdots g_{\ell}$ written in normal form.}
    \label{fig:vD-dihedral-loop}
\end{figure}

\noindent By construction of the complex of groups, for every $i$ we can write $h_{i+1} = h_i g_i$, with $g_i\in \VA_{r_i}$. Thus, the equality $v_0 = v_\ell$ between vertices of type $0$ implies that the element $g_1\cdots g_\ell$ is trivial in $\VA_{\{s,t\}}$. Thus, thought as an element of $\VA_s \frpp\VA_t$, the element $g_1\cdots g_\ell$ is in normal form, and in the kernel of the evaluation map 
 $$\ev_{\VA_{\{s,t\}}}: \VA_{s}\frpp \VA_{t} \longrightarrow\VA_{\{s,t\}}.$$

\noindent By Theorem \ref{virtualAppelSchupp}, $\ell\geq 2m$. Hence, when $|X|=2$, any non-trivial closed curve in $\mathrm{Link}(v,\vD_{\Gamma})$ has length $\ell\cdot \frac{\pi}{m}\geq 2m\cdot \frac{\pi}{m}=2\pi$. \\
Since any non-trivial closed curve in $\mathrm{Link}(v,\vD_{\Gamma})$ has length at least $2\pi$, by Gromov's link criterion we can conclude that $\vD_{\Gamma}$ is CAT(0).
\end{proof}

\subsection{From 2-dimensional to locally reducible virtual Artin groups}\label{subs-from2tolocallyred}
 
\noindent In \cite{Charney2000}, Charney showed that the Moussong metric is CAT(0) for the class of
\textit{locally reducible} Artin groups, defined as follows: 
\begin{defn}\label{def-locred}
    Given a Coxeter graph $\Gamma$ with vertices $V(\Gamma)=S$, the Artin group $\A[\Gamma]$ is \textit{locally reducible} if its spherical-type standard parabolic subgroups decompose as direct products of rank 1 and rank 2 standard parabolic subgroups.
\end{defn}
 \noindent   Charney shows that, in the locally reducible case, the link of any vertex in $v$ in $\D_{\Gamma}$ is CAT(1) by decomposing it as a certain \textit{orthogonal join} of CAT(1) complexes, some of which are links of vertices in 2-dimensional Artin groups. The proof in the virtual Artin case extends in a completely analogous way: one can show that the links of vertices decomposes in a similar way and use the fact that links of vertices are CAT(1) in the 2-dimensional case (see Theorem~\ref{thm-2dim-vD-CAT(0)}) to conclude:

\begin{cor}\label{cor-CAT0-locred}\label{thm-from2-tolocred}
Let $\Gamma$ be a locally reducible Coxeter graph. Then the Moussong metric on the virtual Deligne complex $\vD_{\Gamma}$ is CAT(0).
\end{cor}

\noindent For the reader's convenience, we recall Charney's proof and adapt it to the virtual Artin case. We emphasise however that the rest of this subsection is a direct transfer of her proof to the virtual Artin case. We use here the notations introduced in Section~\ref{subs-moussongmetric}.\medskip\\
\noindent When studying links of vertices in $\vD_{\Gamma}$, we may assume without loss of generality that the vertex lies in the fundamental domain $K_{\Gamma}=K_{\Gamma}^{\VA}$. Now let $\Gamma$ be a locally reducible Coxeter graph, and let $\vD_{\Gamma}$ be the associated virtual Deligne complex. 
\noindent Recall that the link of a piece-wise Euclidean complex is a piece-wise spherical complex. Given two spherical simplices $\sigma_1,\sigma_2$ of dimensions $\kappa_1$, $\kappa_2$, the \textit{orthogonal join} $\sigma_1 \frpp \sigma_2$ is defined as follows. Embed $\mathbb{S}^{\kappa_1}$ and $\mathbb{S}^{\kappa_2}$ in $\mathbb{S}^{\kappa}$ with $\kappa=\kappa_1+\kappa_2+1$ so that every point of in $\mathbb{S}^{\kappa_1}$ is at distance $\pi/2$ from every point in $\mathbb{S}^{\kappa_2}$. Then $\sigma_1\frpp \sigma_2$ is the $\kappa$-simplex spanned by $\sigma_1\subset \mathbb{S}^{\kappa_1}$ and $\sigma_2\subset \mathbb{S}^{\kappa_2}$. If $L_1$ and $L_2$ are piece-wise spherical complexes, their \textit{orthogonal join} $L_1\frpp L_2$ is the piece-wise spherical complex whose simplices are the orthogonal joins $\sigma_1\frpp \sigma_2$ of (possibly empty) simplices $\sigma_1\subset L_1$ and $\sigma_2\subset L_2$. \\
\\
\noindent Consider now $\vD_{\Gamma}$ with the Moussong metric described in Subsection \ref{subs-davis-and-deligne-complex}. If $\Gamma$ is of spherical type with $V(\Gamma)=S$, then $\VA\Sf$ has a unique maximal element $\VA_S=\VA[\Gamma]$. Hence $K_{\Gamma}^{\VA}=cube(S)$ and $\vD_{\Gamma}$ is a cone with cone point at the vertex $\VA_S=\VA$. Let $\mathcal{B}_S$ denote the link of the cone point. Then $\mathcal{B}_S$ is a simplicial complex of dimension $|S|-1$ which has a piece-wise spherical structure given by identifying each highest-dimensional simplex with the link of the origin in the cone $T_S$.  \medskip\\
\noindent Suppose now that $\Gamma$ is of infinite type. If $Y\subseteq X\in \Sf$, let  $K_{\Gamma_{\geq Y}}^{\VA}$ be the subcomplex of $K_{\Gamma}^{\VA}$ spanned by the vertices $\VA_X$ such that $Y\subseteq X\in \Sf$. If $Y\subseteq X\in \Sf$, denote by $cube(Y,X)$ the face of $cube(X)$ spanned by $\VA_X$ and $\VA_Y$, which lies in the subcomplex $K_{\Gamma_{\geq Y}}^{\VA}$ of $K_{\Gamma}^{\VA}$. \\

\noindent The proof of Corollary \ref{cor-CAT0-locred} relies on 
the following link decomposition. Its proof is essentially identical to that \cite[Lemma 2.2]{Charney99}, and we only include here for ease of reference:

\begin{lem}\label{lem-decomp-link}
    If $x$ lies in the interior of $cube(Y,X)$, then $\mathrm{Link}(x,\vD_{\Gamma})$ is isometric to the orthogonal join $\mathrm{Link}(x, K_{\Gamma_{\geq Y}}^{\VA})\ast \mathcal{B}_Y$.
\end{lem}

\begin{proof}
Observe first that $cube(Y)$ and $cube(Y,X)$ intersect orthogonally at the vertex $\VA_Y$ of $K_{\Gamma}^{\VA}$, by construction of the Moussong metric (see  Figure \ref{fig:MoussongMetric}). If $x$ is a point in the interior of $cube(Y,X)$, then a sufficiently small neighbourhood of $x$ in $cube(X)$ can be identified with a neighbourhood of $(\VA_Y,x)$ in the orthogonal product $cube(Y)\times cube(Y,X)$ (see Figure \ref{fig:linkdecomp}).
\begin{figure}[h!]
\centering
\begin{minipage}{4cm}
  \centering
\includegraphics[width=4cm]{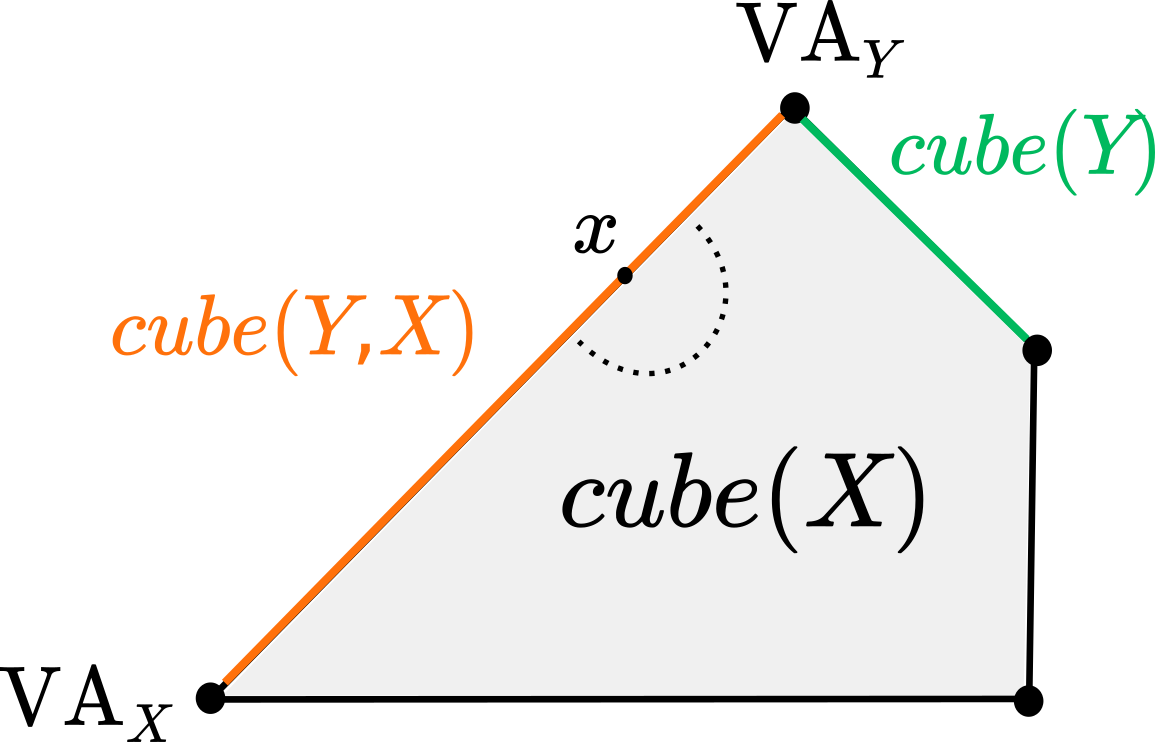}
\end{minipage}
\qquad\qquad
\begin{minipage}{4cm}
  \centering
\includegraphics[width=3.5cm]{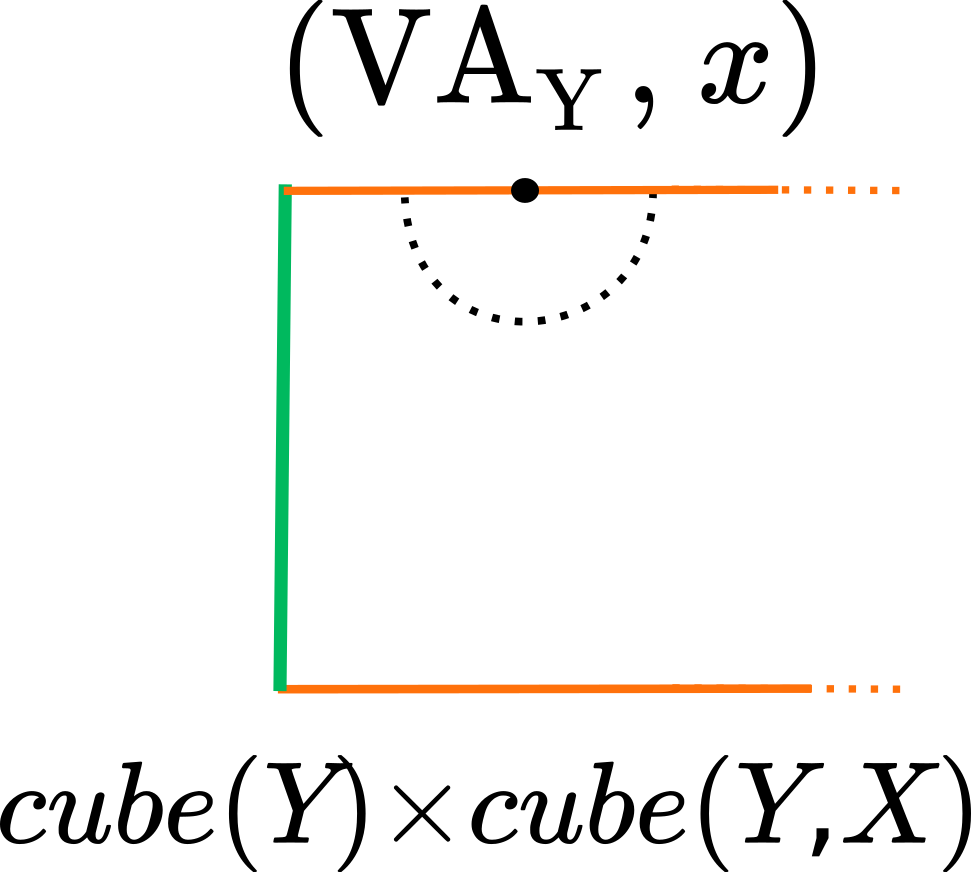}
\end{minipage}
  \label{fig:linkdecomp}
  \caption{The neighbourhood of $x\in cube(Y,X)$. }
\end{figure}Taking the union over all $X\in \Sf$ such that $Y\subseteq X$, we see that a neighbourhood of $x$ in $K_{\Gamma}^{\VA}$ can be identified with a neighbourhood of $(\VA_Y,x)$ in the orthogonal product $cube(Y)\times K_{\Gamma_{\geq Y}}^{\VA}$. Thus
\[
\mathrm{Link}(x,K_{\Gamma}^{\VA})=\mathrm{Link}(\VA_Y,cube(Y))\frpp \mathrm{Link}(x,K_{\Gamma_{\geq Y}}^{\VA}).
\]
\noindent Taking the orbit of $K_{\Gamma}^{\VA}$ under the $\VA_{Y}$-action, we obtain a neighbourhood of $x$ in $\vD_{\Gamma}$. The action of $\VA_Y$ fixes $K_{\Gamma_{\geq Y}}^{\VA}$ and hence 
\[
\mathrm{Link}(x,\vD_{\Gamma})=\mathrm{Link}(x,K_{\Gamma_{\geq Y}}^{\VA})\frpp \mathcal{B}_Y.
\]
\end{proof}

\noindent To show Corollary \ref{cor-CAT0-locred}, we rephrase Charney's proof \cite[Theorem 3.2]{Charney2000} for the classical Deligne complex and we use the decomposition of the link of a vertex carried out in Lemma \ref{lem-decomp-link}.
\begin{proof}[Proof of Corollary \ref{cor-CAT0-locred}]
Since $\vD_{\Gamma}$ is simply connected by construction, we only need to check that the link of any vertex is CAT(1). Pick a vertex $v=g\VA_X$ of $\vD_{\Gamma}$: without loss of generality, we can assume that $g=\id$ and thus that $v=\VA_X\in K_{\Gamma}^{\VA}$ with $X\in \Sf$. By Lemma \ref{lem-decomp-link}, 
\[
\mathrm{Link}(v,\vD_{\Gamma})=\mathrm{Link}(v,K_{\Gamma_{\geq X}}^{\VA})\frpp \mathcal{B}_X.
\]
\noindent The fact that $\mathrm{Link}(G_X,K_{\Gamma_{\geq X}}^{G})\subset K_{\Gamma}^{G}$ is CAT(1) was shown for Coxeter groups by Moussong in \cite{Moussong} and for Artin groups by Charney and Davis in \cite[Lemma 4.4.1]{CharDav95}. The proof does not depend on $G$ but only on the intrinsic metric $d_M$ on $K_{\Gamma}$, so it follows  that $\mathrm{Link}(\VA_X,K_{\Gamma_{\geq X}}^{\VA})$ is CAT(1).\medskip\\
\noindent It remains to show that $\mathcal{B}_X$ is CAT(1). Since $\Gamma$ is locally reducible, for each $X\in \Sf$ we can write 
\[
\VA[\Gamma_X]=\VA[\Gamma_{X_1}]\times \cdots \times \VA[\Gamma_{X_k}]
\]
\noindent with $|X_i|=1$ or $2$ for all $i=1,\ldots,k$. It is easily seen that $\mathcal{B}_X$ decomposes as the orthogonal join 
\[
\mathcal{B}_{X}=\mathcal{B}_{X_1}\frpp \cdots \frpp \mathcal{B}_{X_k}.
\]
\noindent Recall that $\mathcal{B}_{X_i}$ is the link of the vertex $\VA_{X_i}$ in $\vD_{X_i}$. If $|X_i|=1$, then $\mathcal{B}_{X_i}$ is discrete and clearly CAT(1). If $|X_i|=2$, then by Theorem \ref{thm-2dim-vD-CAT(0)}, $\mathcal{B}_{X_i}$ is CAT(1). Since a orthogonal join of CAT(1) piece-wise spherical complexes is CAT(1) by  \cite[Theorem A.10]{CharneyDavis93}, it follows that $\mathrm{Link}(v,\vD_{\Gamma})$ is CAT(1), hence $\vD_{\Gamma}$ is CAT(0).
\end{proof}

\section{Dihedral virtual Artin groups}\label{sectionDihedral}

\noindent The goal of this and the following two sections is to develop the tools to show Theorem \ref{virtualAppelSchupp}, which we recall below:
\begin{reptheorem}{virtualAppelSchupp}[Virtual Appel--Schupp Theorem]
    Let $\Gamma_m$ be a dihedral Coxeter graph with vertex set $S=\{s,t\}$ and $m\geq 2$. Let $\VA_m$ be the associated virtual Artin group, and let $\omega$ be a non-trivial element in $\VA_s\frpp \VA_t$ written in normal form, such that $\ev_{\VA_m}(\omega)=\id_{\VA_m}$. Then the syllabic length of $\omega$ satisfies $\|\omega\|\geq 2m$. 
\end{reptheorem}

\noindent An outline of what is done in Sections \ref{sectionDihedral}, \ref{sec-poly-of-groups-ArtinGammaHat} and 
\ref{sec-virtualAppelSchupp} is summarized here below.\begin{itemize}
\item In this section, we study the equation $\ev_{\VA_m}({\omega})=\id_{\VA_m}$ in the dihedral virtual Artin group $\VA_m$. In Subsection \ref{subs-fromVAtoAgammaHat}, we decompose $\VA_m$ into $\A[\hGammam]\rtimes \W_m$, and we obtain that the desired equality in $\VA_m$ holds only if a certain equality  holds in the Artin group $\A[\hGammam]$. In Subsection \ref{subs-description-gammaHat} we describe the Coxeter graph $\hGammam$ starting from a study of the dihedral root system $\Phi_m$ of $\W_m$, and we investigate the action of the Coxeter group $\W_m$ on $\Phi_m$, and thus on $\A[\hGammam]$.
\item In Section \ref{sec-poly-of-groups-ArtinGammaHat} we assume $m\geq 3$ and we realize $\A[\hGammam]$ as the fundamental group of a new strictly developable polygon of groups. This gives an action of $\A[\hGammam]$ on a CAT(0) polygonal complex $X$, whose key feature is that the aforementioned equality  in the Artin group $\A[\hGammam]$ corresponds to a loop of polygons in $X$. If such a loop contains no back-tracking, we use disc diagram arguments to bound below the length of such a loop (see Proposition \ref{prop_eq_inKVA}).
\item Finally, in Section \ref{sec-virtualAppelSchupp} we finish the proof of Theorem \ref{virtualAppelSchupp}, by studying the case where the aforementioned loop of polygons contains back-tracking. This requires a finer study of the action of $\W[\hGammam]$ on $\hGammam$. 
\end{itemize}

\subsection{\texorpdfstring{From an equation in $\VA_m$ to an equation in $\A[\hGammam]$}{From an equation in VA}}\label{subs-fromVAtoAgammaHat}

\noindent Take a non-trivial element $\omega=g_1\cdots g_{\ell}$ in the free product $\VA_s\frpp \VA_t$, written in normal form and such that $\ev_{\VA_m}(\omega)=\id_{}$ in the dihedral virtual Artin group $\VA_m$.
\begin{align*}
    \VA_{s}\frpp \VA_{t} &\xrightarrow[]{\qquad} \VA_m\\
    \omega=g_1\,g_2\cdots g_{\ell} &\longmapsto \ev_{\VA_m}({\omega})=\id_{\VA_m}.
\end{align*}
\noindent We will assume that $\omega$ is cyclically reduced. In particular, $g_i\neq \id$ for all $i=1,\ldots,{\ell}$, $g_i$ and $g_{i+1}$ never belong to the same free factor for all $i=1,\ldots,{\ell}-1$, and $g_{\ell}$, $g_{1}$ do not belong to the same free factor. 
\\
\noindent Consider now a syllable $g_i$ in $\omega$. We know that $g_i\neq \id$ and that $g_i\in \VA_{r_i}$ for ${r_i}\in \{s,t\}$. Write now $g_i=(k_i,w_i)$ with $k_i\in \KVA[\Gamma_{r_i}]=\A[\widehat{\Gamma_{r_i}}]$ and $w_i\in \W[\Gamma_{r_i}]=\{\idW,r_i\}$. Therefore, by $\ev_{\VA_m}(\omega)=\id$, we obtain
\[
g_1\,g_2\cdots g_{\ell}=(k_1,w_1)\,(k_2,w_2)\cdots (k_{\ell},w_{\ell})=\id,
\]
\noindent where $g_1\cdots g_{\ell}$ is meant to be a product of elements in the dihedral virtual Artin group $\VA_m$. By the action described in Equation \ref{actionWonKVA} and Notations \ref{notationactionWonKVA}-\ref{nt-g=(k,w)}, we write:
\[
(k_1\,k_2^{w_1},w_1w_2)\cdots (k_{\ell},w_{\ell})= (k_1\,k_2^{w_1}\,k_3^{w_1 w_2}\cdots k_{\ell}^{w_1 w_2\cdots w_{{\ell}-1}},\,w_1\cdots w_{\ell})=\id=(\idK,\idW).\]
\noindent Set now $u_1:=\idW$ and $u_i:=w_{1}\cdots w_{i-1}$ for $2\leq i \leq {\ell}$. We obtain:
\begin{equation}\label{eq1}
(\underbrace{k_1^{u_1}\,k_2^{u_2}\cdots k_{\ell}^{u_{\ell}}}_{\in\, \KVA[\Gamma_m]=\A[\hGammam]}\,,\,\underbrace{w_1\cdots w_{\ell}}_{\in \, \W_m})=(\idK,\idW).
\end{equation}
\noindent Equation \ref{eq1} holds in $\VA_m$ if and only if 
\begin{equation}\label{eq2}
 \begin{cases}
k_1^{u_1}\,k_2^{u_2}\cdots k_{\ell}^{u_{\ell}}=\idK \quad&\text{in $\A[\hGammam]$};\\
    w_1\cdots w_{\ell}=\idW \quad &\text{in $\W_m$.} 
\end{cases}   
\end{equation}
\begin{defn}
    Given a non-trivial element $\omega=g_1\cdots g_{\ell}$ written in normal form in $\VA_s\frpp \VA_t$, we call the system in (\ref{eq2}) the \textit{system of equations associated with $\omega$.}
\end{defn}
\noindent To show Theorem \ref{virtualAppelSchupp}, we consider the equation 
\begin{equation}\label{equationInKVA}
k_1^{u_1}\,k_2^{u_2}\cdots k_{\ell}^{u_{\ell}}=a_1\cdots a_{\ell}=\idK \quad\text{in $\A[\hGammam]$},
\end{equation}
\noindent with $a_i:=k_i^{u_i}$ for all $i\in \{1,\ldots, \ell\}$. To study this equation in $\A[\hGammam]$, we need first to describe the kernel $\KVA[\Gamma]=\A[\hGamma]$ when $\Gamma$ is the dihedral graph $\Gamma=\Gamma_m$.

\subsection{\texorpdfstring{Description of $\hGamma$ for $\Gamma$ dihedral}{Description of hGamma for Gamma dihedral.}}\label{subs-description-gammaHat}\noindent Recall that the group $\KVA[\Gamma]$ is isomorphic to the Artin group $\A[\hGamma]$, where the vertices of $\hGamma$ are the roots in $\Phi[\Gamma]$ of $\W[\Gamma]$. According to Definition \ref{defGammaHat}, we place an edge in $\hGamma$ between the roots $\beta$ and $\gamma$ if there exist $w\in \W[\Gamma]$ and $a,b\in V(\Gamma)=S$ such that $\beta=w(\alpha_a)$ and $\gamma=w(\alpha_b)$; such an edge is labelled by $\hmbg:= m_{a,b}$.\\
\\
\noindent Since we consider the dihedral case, the group $\W[\Gamma]$ is the dihedral Coxeter group $\W_m$ of order $2m$. For the reader's convenience, we give a brief overview of root systems in this case (see also \cite{Bourbaki, Hump}). We denote by $\Phi_m$ the dihedral root system of $\W_m$, which has only two simple roots, $\alpha_s$ and $\alpha_t$. In Figure \ref{fig:oddrootsystem}, we illustrate the root system $\Phi_m$ of the dihedral Coxeter group of order $2m$ for $m=5$.\\
\\
\noindent Recall that $\W_m$ is the group of isometries of a regular $2m$-gon centred at the origin of the Euclidean plane, consisting of $m$ rotations by angles that are multiples of $2\pi/m$, and $m$ reflections with respect to the symmetry axes. If $m$ is even, a symmetry axis of the $2m$-polygon is either the line joining two opposite vertices or the line joining the midpoints of opposite edges. If $m$ is odd, a symmetry axis is a line joining a vertex of the polygon to the midpoint of its opposite edge. The dihedral group is generated by two reflections whose associated orthogonal hyperplanes form an angle of $\pi/m$. Indeed, a rotation of $2\pi/m$ is the composition of two such reflections. \\
\begin{figure}[ht]
\centering
\begin{minipage}{7cm}
  \centering
\includegraphics[width=4.5cm]{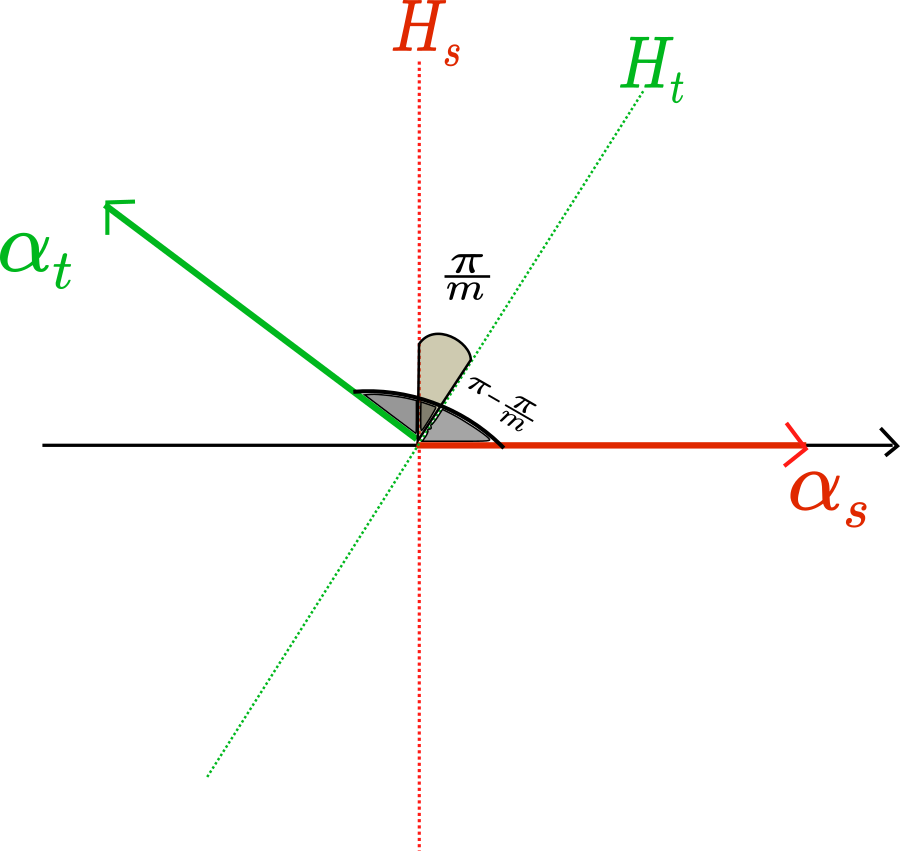}
  \caption{The simple roots $\alpha_s$ and $\alpha_t$ with their respective orthogonal hyperplanes.}
  \label{fig:rootsandhyp}
\end{minipage}
\qquad
\begin{minipage}{7cm}
  \centering
\includegraphics[width=5cm]{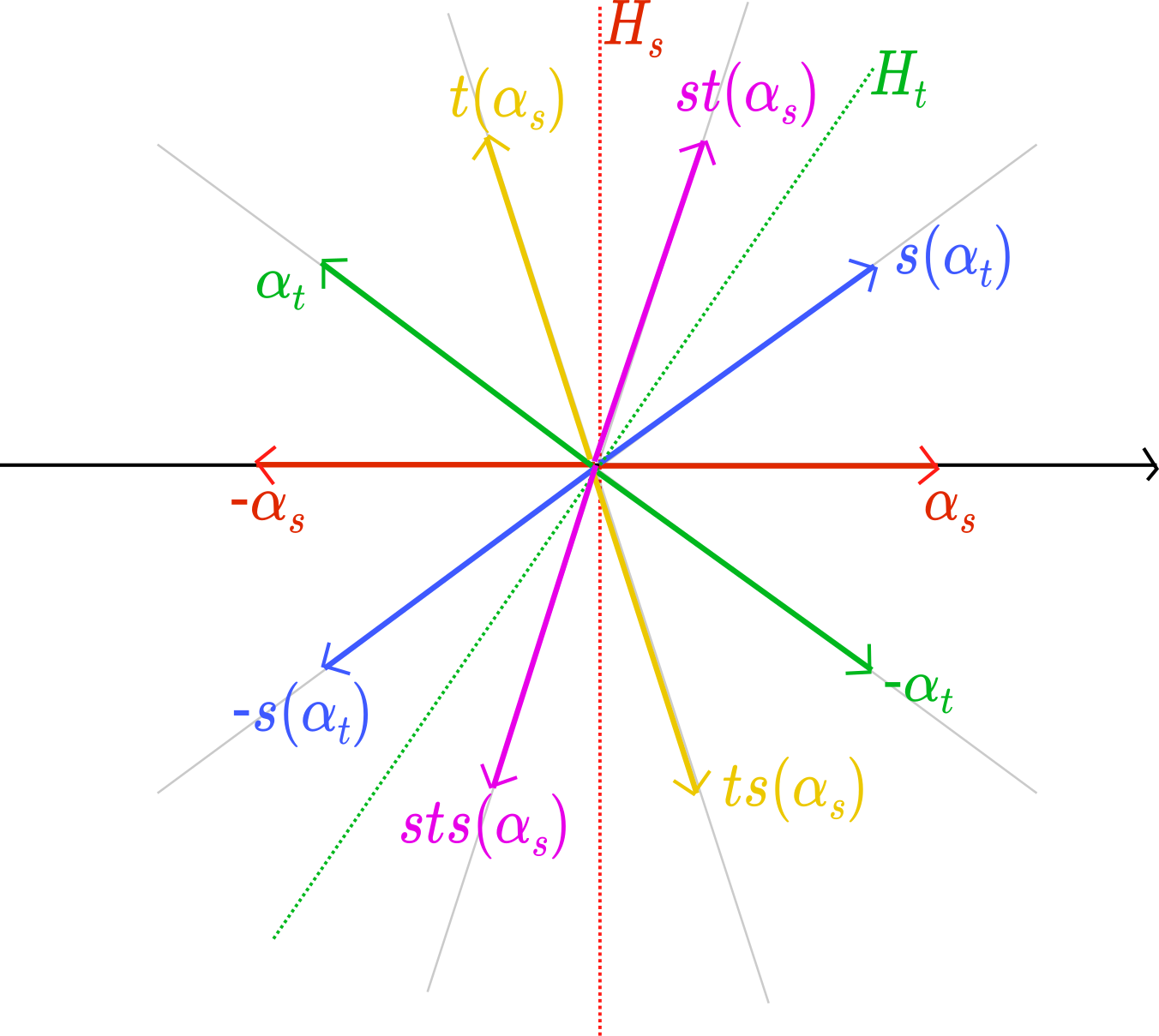}
  \caption{The root system $\Phi[\Gamma]$ of the dihedral group for $m=5$.}
  \label{fig:oddrootsystem}
\end{minipage}
\end{figure}
\\
\noindent The dihedral root system $\Phi_m$ is constructed as follows. Take $V=\mathbb R^2$ and fix the unit root vector $\alpha_s$ to be the vector $(1,0)$. Its orthogonal hyperplane $H_s$ is the axis $x=0$. Now draw the line $H_t$, which will be the hyperplane orthogonal to the simple root $\alpha_t$, such that the angle between the lines $H_t$ and $H_s$ is $\pi/m$ (counter-clockwise). Draw now the unit root vector $\alpha_t$ orthogonal to $H_t$ in the upper half plane. Then, the angle between $\alpha_s$ and $\alpha_t$ is $\pi-\pi/m$. All the other non-simple roots are obtained acting by reflections with respect to $H_s$ and $H_t$ on $\alpha_s$ and $\alpha_t$ (see Figures \ref{fig:rootsandhyp} and \ref{fig:oddrootsystem}). In other words, all the $2m$ root vectors are obtained as rotations of the vector $(1,0)=\alpha_s$ by integer multiples of $\pi/m$.\\
\begin{rmk} The standard scalar product $\langle \alpha_s,\alpha_t \rangle$ of $\mathbb R^2$ is $\cos{(\pi-\pi/m)}=-\cos(\frac{\pi}{m})$, therefore we can identify it with the symmetric bilinear form defined in Section \ref{subs-rootsystems}.
\end{rmk}\medskip
\begin{nt}
    When two vertices $a,b$ of a Coxeter graph are joined by an edge labelled by $m\in \mathbb N_{\geq 2}$, we say that $a$ and $b$ are $m$-\textit{adjacent}. If there is a path between $a$ and $b$ (not necessarily of length 1) in which all the edges are labelled by $m$, we say that $a$ and $b$ are $m$-\textit{connected}. \medskip
\end{nt}
\noindent We now start our study of the graph $\hGamma$ when $\Gamma$ is the dihedral graph $\Gamma_m$ with vertices $S=\{s,t\}$ and $m_{s,t}=m\neq \infty$. We will denote such a graph by $\hGammam$. The vertices of $\hGammam$ are $\{{\beta}\mid \beta\in \Phi_m\}$. By the general construction of $\hGamma$, if $\beta,\gamma\in \Phi_m$ are such that $\hmbg=m$, then there exists $w\in \W
_m$ such that $\{\beta,\gamma\}=w\{\alpha_s,\alpha_t\}$. Since $\W_m$ acts by isometries on $V=\mathbb R^2$, we have that if $\hmbg=m$, then 
\[\langle \beta,\gamma \rangle=\langle w(\alpha_s),w(\alpha_t)\rangle=\langle \alpha_s,\alpha_t\rangle=\pi-\frac{\pi}{m}. \]
\noindent Hence, the two roots $\beta, \gamma$ are $m$-adjacent only if they form an obtuse angle measuring $\pi-\frac{\pi}{m}$. We will now see that this necessary condition is also sufficient.\\
\\
\noindent Indeed, suppose that $\beta,\gamma\in \Phi_m$ form an angle of $\pi-\pi/m$. Assume without loss of generality that $\beta$ is the first root that we encounter moving counter-clockwise from $\alpha_s$. Let $\theta=\frac{k\pi}{m}$ be the angle between $\alpha_s$ and $\beta$, with $0<k\leq m$. We have the following two possibilities.
\begin{enumerate}
    \item[(1.)] If $k$ is even, then $\theta=l\frac{2\pi}{m}$ for some $l\in \mathbb N$. Since $st\in \W_m$ is the counter-clockwise rotation of $\frac{2\pi}{m}$, then $w=(st)^l$ is such that $w(\alpha_s)=\beta$ and $w(\alpha_t)=\gamma$.
    \item[(2.)] If $k$ is odd, write $k=2l+1$ for $l\in \mathbb N$. If $m$ is odd, let $w$ be the reflection\\ $w=\Prod_R(s,t;l)\;s\;\Prod_L(t,s;l)$. Then $w(\alpha_s)=\beta$ and $w(\alpha_t)=\gamma$. If $m$ is even, let $w$ be the reflection $\Prod_R(t,s;l)\;t\;\Prod_L(s,t;l)$. Then $w(\alpha_t)=\beta$ and $w(\alpha_s)=\gamma$.
\end{enumerate}

\noindent Thanks to this discussion, we can show the following lemma.

\begin{lem}\label{adjacencyinGammaHat}
Any $\beta\in \Phi_m$ is $m$-adjacent in $\hGammam$ to exactly two other roots $\gamma_1,\gamma_2$.
\end{lem}
\begin{proof}
By what is written above, $\beta$ is $m$-adjacent  to another root $\gamma$ in $\hGammam$ if and only if they form an angle of $\pi-\pi/m$. The roots in $\Phi_m$ are separated in pairs by angles of multiples of $\pi/m$. Therefore, there are exactly two roots $\gamma_1,\gamma_2\in \Phi_m$ such that the angle between $\beta$ and $\gamma_i$ is $\pi-\pi/m$ for $i=1,2$, and the result follows.
\end{proof}

\noindent The following is a classical result in the theory of Coxeter groups, and it can be found in \cite{deodh82}.
\begin{lem}\label{lemma_roots}
    Let $\Gamma$ be a Coxeter graph with vertex set $V(\Gamma)=S$, and let $s,t\in S$ be two elements such that $m_{s,t}\neq \infty$. Then 
    \[
    \Prod_R(s,t;m_{s,t}-1)(\alpha_s)=\begin{cases}
        \alpha_t \quad \text{if $m_{s,t}$ is odd,}\\
        \alpha_s \quad \text{if $m_{s,t}$ even.}
    \end{cases}
    \]
\end{lem}

\begin{rmk}\label{orbitsofroots}
    The previous lemma says in particular that, if $m$ is odd, then all the roots in in the dihedral root system $\Phi_m$ belong to the same orbit under the action of $\W_m$. If $m$ is even, then the two simple roots $\alpha_s$ and $\alpha_t$ generate the two disjoint orbits of roots under the action of the dihedral Coxeter group.
\end{rmk}

\noindent A straightforward consequence of Lemma \ref{lemma_roots} is the following result. 

\begin{cor}\label{cor_roots}
    Let $\Gamma_m$ be a dihedral Coxeter graph with $V(\Gamma)=\{s,t\}$, and let $\alpha_s$ and $\alpha_t$ be the associated simple roots in $\Phi_m$. Then, for $1\leq k \leq m$:
    \begin{enumerate}
        \item[(1.)] If $m$ is even, $\Prod_R(s,t;m-k)(\alpha_s)=\Prod_R(s,t;k-1)(\alpha_s)$ and $\Prod_R(t,s;m-k)(\alpha_s)=\Prod_R(t,s;k+1)(\alpha_s)$.
        \item[(2.)] If $m$ is odd, $\Prod_R(s,t;m-k)(\alpha_s)=\Prod_R(t,s;k-1)(\alpha_t)$ and $\Prod_R(t,s;m-k)(\alpha_s)=\Prod_R(s,t;k+ 1)(\alpha_t)$.
    \end{enumerate}
\end{cor}

\noindent Another consequence of Lemma \ref{lemma_roots} which will be used in the next sections is the following.

\begin{prop}\label{prop_roots}
    Let $\Gamma_m$ be a dihedral Coxeter graph with $V(\Gamma)=\{s,t\}$, and let $\alpha_s$ and $\alpha_t$ be the associated simple roots in $\Phi_m$. Let $w$ be an element in $\W_m$. Then:
    \begin{enumerate}
        \item [(1.)] If $w(\alpha_s)=\alpha_s$, then either $m$ is even and $w=\Prod_R(s,t;m-1)$, or $w=\idW$.
        \item [(2.)] If $w(\alpha_s)=-\alpha_s$, then either $m$ is even and $w=\Prod_R(s,t;m)$, or $w=s$.
        \item [(3.)] If $w(\alpha_s)=\alpha_t$, then $m$ is odd and $w=\Prod_R(s,t;m-1)$.
        \item [(4.)] If $w(\alpha_s)=-\alpha_t$, then $m$ is odd and $w=\Prod_R(s,t;m)$.
    \end{enumerate}
\end{prop}
\begin{proof}
    Recall that, in the dihedral Coxeter group $\W_m$, the elements of odd length are reflections with respect to some root $\beta \in \Phi_m$, while elements of even length $l$ are counter-clockwise rotations of angle $2l\pi/m$. Reflections fix the hyperplane orthogonal to the associated root $\beta$, while rotations have no fixed points. \medskip \\
    \noindent (1.) If $w(\alpha_s)=\alpha_s$, then either $w=\idW$, or $w=r_{\beta}$ is a reflection and $\alpha_s$ is orthogonal to $\beta$. Since the angle between $\beta$ and $\alpha_s$ must be a multiple of $\pi/m$, it is equal to $\pi/2$ only if $m$ is even. By Lemma \ref{lemma_roots}, the reflection $r_{\beta}$ is $\Prod_R(t,s;m-1)$, which can be written as $\Prod_R(t,s;\frac{m-2}{2})\,t\,\Prod_L(s,t;\frac{m-2}{2})$. In this case the (positive) root $\beta$ is $\Prod_R(t,s;\frac{m-2}{2})(\alpha_t)$. \medskip \\
    \noindent (2.) If $w(\alpha_s)=-\alpha_s$, then $w$ is either a counter-clockwise rotation of $\pi$, or a reflection fixing the hyperplane orthogonal to $\alpha_s$. In the latter case, we clearly have $w=s$. In the former case, we have that $m$ must be even and $w$ is the longest element $w=\Prod_R(s,t;m-1)$.\medskip\\
    \noindent (3.) If $w(\alpha_s)=\alpha_t$, then $\alpha_s$ and $\alpha_t$ are in the same orbit. By Remark \ref{orbitsofroots} we have that $m$ is odd and that $w$ is a counter-clockwise rotation of $\pi-\pi/m$, namely, $w=\Prod_R(s,t;m-1)$. \medskip\\
    \noindent (4.) If $w(\alpha_s)=-\alpha_t=t(\alpha_t)$, then $\alpha_s$ and $\alpha_t$ are in the same orbit. Again, this implies that $m$ is odd and that $tw(\alpha_s)=\alpha_t$. By part (3.) we get that $tw=\Prod_R(s,t;m-1)$, hence $w=\Prod_R(s,t;m)$.
\end{proof}

\noindent We are now able to completely characterise $\hGammam$.

\begin{prop}\label{propdihedralgammahat}
    Let $\Gamma_m$ the dihedral Coxeter graph with vertices $S=\{s,t\}$.
    \begin{enumerate}
        \item[(1.)] If $m$ is even, then $\hGammam$ is a $2m$-gon with edges labelled by $m$.
        \item[(2.)] If $m$ is odd, then $\hGammam$ is the disjoint union of two $m$-gons with edges labelled by $m$.
    \end{enumerate}
\end{prop}
\begin{proof}
First observe that, since the only label in $\Gamma_m$ is $m$, the only label (other than $\hmbg=\infty$) that can appear in $\hGammam$ is $m$. \\\\
  \noindent(1.) If $m$ is even, we show in $\hGammam$ all the roots of $\Phi_m$ appear in a single cycle of length $2m$, whose edges are all labelled by $m$. By Lemma \ref{adjacencyinGammaHat}, each root of $\Phi_m$ is $m$-adjacent to exactly two other roots. We begin with $\alpha_s$. Clearly, there is an $m$-labelled edge between $\alpha_s$ and $\alpha_t$, corresponding to the image of the base of the root system under the action of $w=\idW$. The root $\alpha_s$ is $m$-adjacent to exactly one further root. By Lemma \ref{lemma_roots}, we have $\Prod_R(s,t;m-1)(\alpha_s)=\alpha_s$, and therefore $\alpha_s$ is $m$-adjacent to $\Prod_R(s,t;m-1)(\alpha_t)$. By part (1.) of Corollary \ref{cor_roots}, this equals $st(\alpha_t)$. Hence, in $\hGammam$ there is $m$-labelled edge between $\alpha_s$ and $st(\alpha_t)$.\medskip\\
  \noindent On the other hand, $st(\alpha_t)$ is $m$-adjacent to $st(\alpha_s)$. Again by Corollary \ref{cor_roots}, we have $st(\alpha_s)=\Prod_R(s,t;m-3)$. If $m=4$, his already yields an $m$-labelled path from $\alpha_s$ to $\Prod_R(s,t;m-3)(\alpha_s)=-\alpha_s$. By the same reasoning, we obtain an $m$-labelled path between $\alpha_t$ (which is $m$-adjacent to $\alpha_s$), to $-\alpha_t$ (which is $m$-adjacent to $-\alpha_s$). Hence, all roots appear in a single $2m$-cycle whose edges are labelled by $m$.

\begin{figure}[h]
\centering
\begin{minipage}{6cm}
  \centering
  \includegraphics[width=5cm]{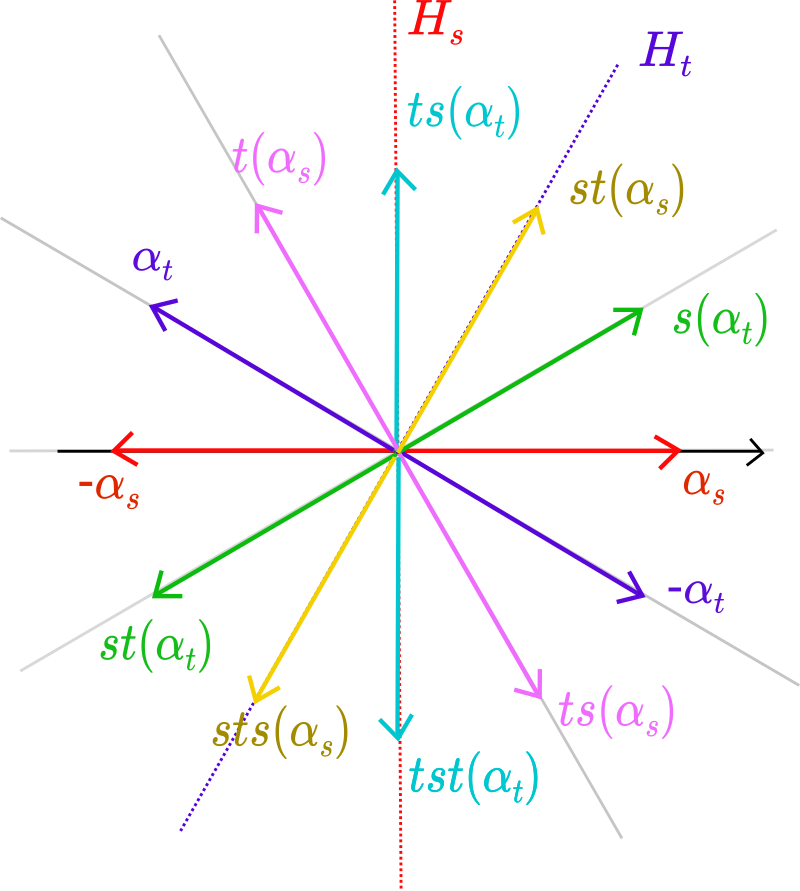}
  \caption{The root system $\Phi_m$ of the even dihedral group for $m=6$.}
  \label{fig:evenRootSystem}
\end{minipage}
\qquad \qquad
\begin{minipage}{6cm}
  \centering
\includegraphics[width=6cm]{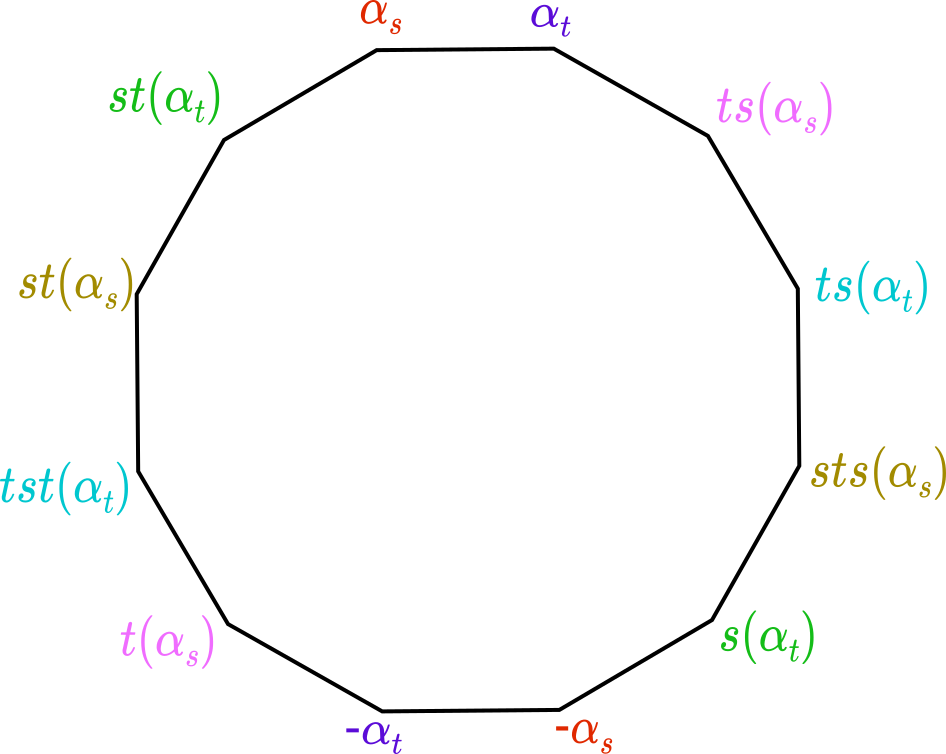}
  \caption{The graph $\hGammam$ for $m=6$. All edges are labelled by 6.}
  \label{fig:evenGammaHat}
\end{minipage}
\end{figure}
  
  \noindent If $m\geq 6$, we continue inductively using Corollary \ref{cor_roots}. At each step, we extend the $m$-labelled path from $\alpha_s$ to $(st)^k(\alpha_s)$, until we reach $-\alpha_s = (st)^{m/2}(\alpha_s)$, and similarly from $\alpha_t$ to $-\alpha_t$. In this way, we obtain a $2m$-cycle containing all roots, with all edges labelled by $m$. Figure \ref{fig:evenGammaHat} illustrates the graph $\hGammam$ in the case $m=6$.\\\\
  \noindent (2.) Let $m$ be odd. By Remark \ref{orbitsofroots}, all the roots of $\Phi_m$ belong to the same orbit under the action of $\W_m$. We show that there exists an $m$-cycle in $\hGammam$ containing the roots $\alpha_s$ and $\alpha_t$. By Lemma \ref{lemma_roots}, we have $
  \Prod_R(s,t;m-1)(\alpha_s)=\alpha_t$. Hence, $\alpha_s$ is $m$-adjacent to $\alpha_t=\Prod_R(s,t;m-1)(\alpha_s)$, which in turn is $m$-adjacent to $\Prod_R(s,t;m-1)(\alpha_t)$. By part (1.) of Corollary \ref{cor_roots}, the latter equals $\Prod_R(t,s;2)(\alpha_s)=ts(\alpha_s)$, which is $m$-adjacent to $ts(\alpha_t)$.
So far we found an $m$-labelled path  with vertices $\alpha_s,\alpha_t,ts(\alpha_s),ts(\alpha_t)$. If $m=3$, then $ts(\alpha_t)=\alpha_s$ and we obtain a cycle of length 3. If $m\geq 5$, we continue by writing $ts(\alpha_t)=\Prod_R(s,t;m-3)(\alpha_s)$ and proceeding inductively. After $m$ steps, we reach $\Prod_R(t,s;m-1)(\alpha_t)=\alpha_s$, thus closing an $m$ cycle.
\begin{figure}[h]
\centering
\includegraphics[width=8cm]{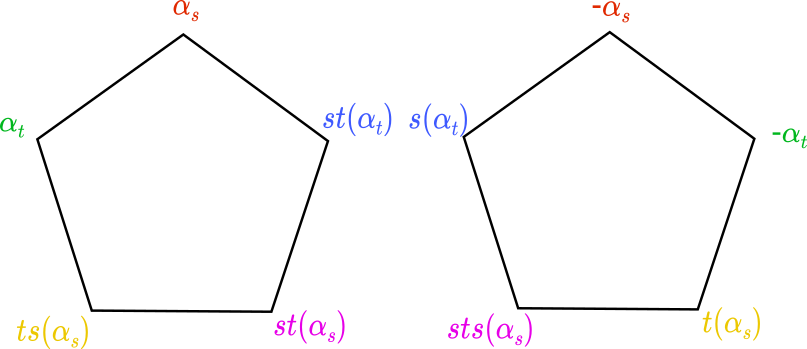}
  \caption{The Coxeter graph $\hGammam$ for $m=5$.}
  \label{fig:odd_Gammahat}
\end{figure}
\noindent By the same argument, we obtain another cycle of length $m$ containing the vertices $-\alpha_s$ and $-\alpha_t$. Observe for any $\beta\in \Phi_m$ with $m$ odd, the two roots $\beta,-\beta$ can never be $m$-connected. Indeed, if there were an $m$-labelled path in $\hGammam$ joining $\beta$ and $-\beta$, then the angle $\pi$ between $\beta$ and $-\beta$ would have to be an integer multiple of $\pi\cdot\frac{m-1}{m}$, which is impossible. Figure \ref{fig:odd_Gammahat} illustrates the two pentagons appearing in $\hGammam$ for $m=5$, whose associated root system is the one pictured in Figure \ref{fig:oddrootsystem}.
\end{proof}

\section{\texorpdfstring{The polygon of groups associated with $\A[\hGammam]$}{The polygon of groups associated with A[]}} \label{sec-poly-of-groups-ArtinGammaHat}

\noindent  Recall that, in order to show the virtual Appel--Schupp Theorem \ref{virtualAppelSchupp}, in Subsection \ref{subs-fromVAtoAgammaHat} we have translated an equation $\ev_{\VA_m}(\omega)=\ev_{\VA_m}(g_1\cdots g_{\ell})=\id$ in $\VA_m$ into an equation in the Artin kernel $\A[\hGammam]$, namely 
\[
\qquad a_1\cdots a_{\ell}=\idK \quad \text{in \,\, $\A[\hGammam]$, \qquad\qquad\qquad\qquad\qquad(\ref{equationInKVA})}
\]
\noindent where $a_i=k_i^{u_i}$ for all $i=1,\ldots,\ell$.\medskip\\
\noindent  To study Equation \ref{equationInKVA}, and more generally equations in $\A[\hGammam]$, in this section we construct a CAT(0) polygonal complex $X$ on which $\A[\hGammam]$ acts. A key feature of this complex is that such equations can naturally be interpreted as certain  loops of polygons in $X$.\medskip\\
\noindent To this aim, we proceed as follows. We assume $m\geq 3$ and view the group $\A[\hGammam]$ as the fundamental group of a certain strictly developable simple polygon of groups $G(\mathcal{Q})$. Then we consider the (subdivision of the) universal cover of such a complex, and we show that, for a specific choice of the metric, it is CAT(0). Thus, we consider elements $\xi=x_1\cdots x_r$ which are \textit{cyclically reduced} a suitable sense (see Definition~\ref{def:cyclically_reduced_GQ}), and if they  they represent the identity in $\A[\hGammam]$, we identify them with non-trivial and non-backtracking loops in such a CAT(0) complex. Using disc diagram arguments, we show that any cyclically reduced element $\xi$ in $G(\mathcal{Q})$  representing the identity in $\A[\hGammam]$ has length at least $2m$.

\subsection{Construction of the polygon of groups} \label{subs-construction-polygon-of-groups}
\noindent
In this subsection, we construct a simple complex of groups $G(\mathcal{Q})$ such that $\A[\hGammam]\cong \pi_1(G(\mathcal{Q}))$. The underlying complex is a regular $m$-gon $P$ (with $m\geq 3$), and $\mathcal{Q}$ denotes the poset of faces of $P$, naturally ordered by inclusion.\medskip \\
\noindent After proving that $G(\mathcal{Q})$ is strictly developable, we consider its development $X$, which is a polyhedral complex with strict fundamental domain equal to the polygon $P$. This yields a face-preserving action of $\A[\hGammam]$ on $X$. We then observe that $P$ (and hence the entire complex $X$) admits a square subdivision $X'$ with a choice of the metric that endows it with the structure of a CAT(0) square complex.
\begin{defn}\label{defn-complexofgroups-GQ} Let $m\geq 3$ be an integer, and let $P$ be the regular $m$-gon whose edges are labelled by the positive roots in $\Phi^+_m$. Two edges $f_{\beta},f_{\gamma}$ are adjacent in $P$ if and only if $\beta$ is $m$-adjacent to $\gamma$ or to $-\gamma$ in $\hGammam$. Let $\mathcal{Q}$ be the opposite of the poset of faces of $P$, that is, the poset of faces of $P$ ordered by reverse-inclusion. We define the simple complex of groups $G(\mathcal{Q})$ over $\mathcal{Q}$ with the following data:
\begin{itemize}
    \item For each vertex ${v}_{\beta,\gamma}$ between two adjacent edges $f_{\beta}$ and $f_{\gamma}$, let the local group at $v_{\beta,\gamma}$ be $G_{v_{\beta,\gamma}}=\langle  \delta_{\beta},\delta_{\gamma},\delta_{-\beta},\delta_{-\gamma}\rangle \cong \A_m \frpp \A_m$.
    \item For each $\beta\in \Phi^+_m$, let the local group at the edge $f_{\beta}$ be $G_{f_{\beta}}=\langle\delta_{\beta},\delta_{-\beta}\rangle=\mathbb{F}(\delta_{\beta},\delta_{-\beta})$.
    \item Let the local group $G_P$ at the entire polygon $P$ be the trivial group.
    \item The local morphisms are the inclusions of standard parabolic subgroups of $\A[\hGammam]$.
\end{itemize}
\end{defn}
\begin{lem}\label{A[hGammam]andcomplexesofGroups}
Let $\Gamma_m$ be the dihedral Coxeter graph with vertices $S=\{s,t\}$ and integer label $m\geq 3$. Then, $\A[\hGammam]$ is isomorphic to $\pi_1(G(\mathcal{Q}))$, where $\mathcal{Q}$ is the simple complex of groups introduced in Definition \ref{defn-complexofgroups-GQ}. Moreover, $G(\mathcal{Q})$ satisfies the hypotheses of Theorem~\ref{thm-strictly-developable-scog}, and in particular is strictly developable.  
\end{lem}
\begin{proof}
   The structure of the Coxeter graph $\hGammam$ is described in Proposition \ref{propdihedralgammahat}. As usual, denote by $\{\delta_{\beta}\,\mid \beta\in \Phi_m\}$ the generating set of $\A[\hGammam]$. The fact that
    \[
    \A[\hGammam]\cong \pi_1(G(\mathcal{Q}))
    \]
    \noindent easily follows from the definition of $\pi_1(G(\mathcal{Q}))$. It remains to show that $G(\mathcal{Q})$ is a strictly developable complex of groups, namely, that all the morphisms $i_q: G_q\longrightarrow \pi_1(G(\mathcal{Q}))=\A[\hGammam]$ are injective. Note that $|\mathcal{Q}'|$ is simply connected since the face $P$ is a maximal element of $\mathcal{Q}$, hence $|\mathcal{Q}'|$ is a cone. Since the local groups of $G(\mathcal{Q})$ are standard parabolic subgroups of $\A[\hGammam]$, they inject into $\A[\hGammam]$ by Van der Lek's result \cite{van1983homotopy}, and hence the result is proven. 
\end{proof}
\begin{ex}
    Let $m=6$, so that $\Phi_6$ is the root system illustrated in Figure \ref{fig:evenRootSystem}. 
    In Figure \ref{fig:even_complexofgroups} we illustrate the hexagon $P$ and the strictly developable simple complex of groups $G(\mathcal{Q})$. The edge-colours of $P$ in Figure \ref{fig:even_complexofgroups} correspond to the lines containing $\beta$ and $-\beta$ in Figure \ref{fig:evenRootSystem}.

\begin{figure}[h]
\centering
\begin{minipage}{6cm}
  \centering
\includegraphics[width=6cm]{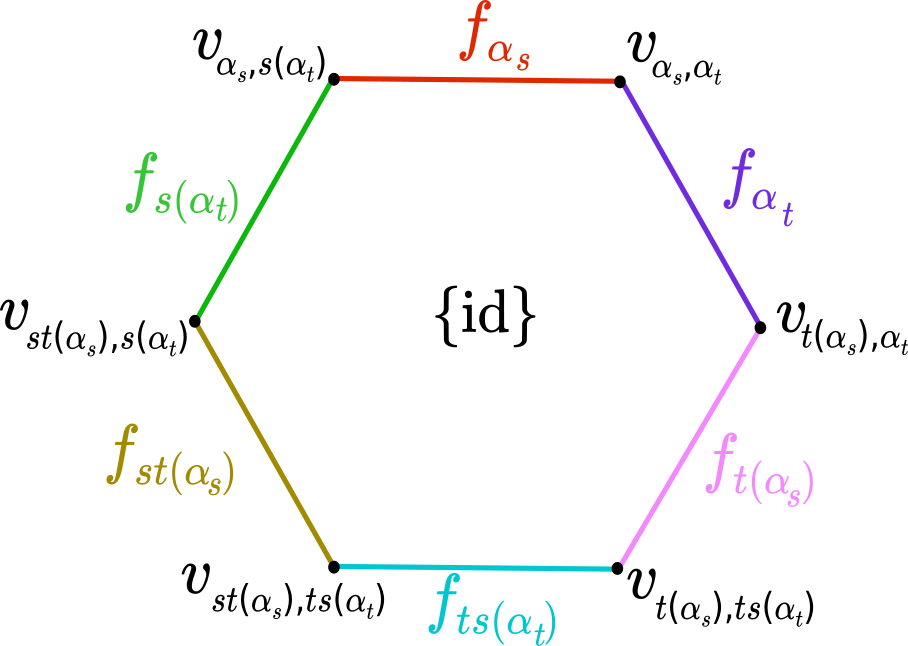}
\end{minipage}
\qquad\qquad\quad
\begin{minipage}{6.5cm}
  \centering
\includegraphics[width=6.5cm]{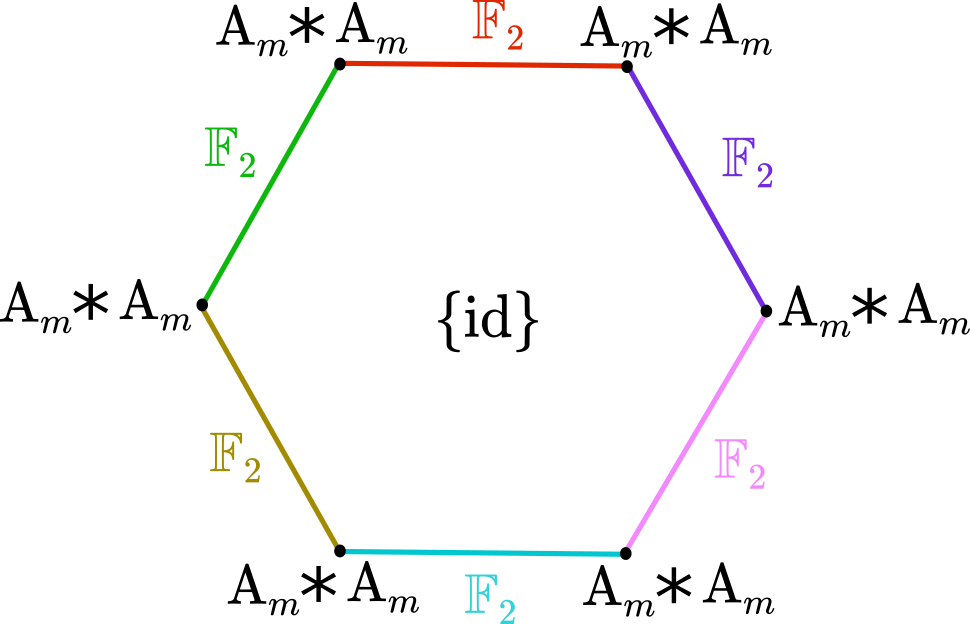}
\end{minipage}
\caption{The polygon $P$ (left) and the complex of groups $G(\mathcal{Q})$ for $m=6$ (right). For better readability, the local morphisms are omitted.}
\label{fig:even_complexofgroups}
\end{figure}
    
\end{ex}

\noindent Proposition \ref{A[hGammam]andcomplexesofGroups} implies that $\A[\hGammam]$ acts face-preserving on a polygonal complex $X$ built out of $m$-gons, and with strict fundamental domain $P$.
\noindent We now consider 
a square subdivision $X'$ of $X$ obtained by adding a vertex at the centre of each polygon and each edge of $X$, and by adding edges between the centre of a polygon and the midpoint of each of its each edges.\medskip\\
\noindent The study of the complex $X$ (and, in particular, of its square subdivision $X'$) is motivated by the fact that an equation like the one we have in \ref{equationInKVA} can naturally be interpreted as a certain non back-tracking loop of polygons in $X$, which we now explain. 

\begin{defn}\label{def:cyclically_reduced_GQ} A \textit{word in $G(\mathcal{Q})$} is a word of the form $\xi=x_1\cdots x_r$ with each syllable $x_i$ belonging to an edge-group of $G(\mathcal{Q})$.  We say that the word $\xi$ is \textit{cyclically reduced} in  $G(\mathcal{Q})$ if it is cyclically reduced as a word in the free product of the edge-groups of $G(\mathcal{Q})$, that is, for every $i\in \mathbb{Z}_{/r}$, the syllables $x_i$ and $x_{i+1}$ never belong to the same edge-group associated to the fundamental domain $P$, and $x_i\neq \id$ for all $i=1,\ldots,r$.
\end{defn}
\noindent Similarly to what is done for free products, we denote by $\ev_{\A[\hGammam]}(\xi)$ the evaluation of the element $\xi$ in the fundamental group of $G(\mathcal{Q})$, which is isomorphic to $\A[\hGammam]$ by Lemma \ref{A[hGammam]andcomplexesofGroups}. The following lemma motivates the study of this polygonal complex.
\begin{lem}\label{lem-cycred-nonbackloops}  Let  $\xi=x_1\cdots x_r$ be a cyclically reduced word in $G(\mathcal{Q})$ such that $\ev_{\A[{\hGammam}]}(\xi)=\idK$. For each $0 \leq i \leq r$, let $P_i\coloneqq x_1\cdots x_iP$. The sequence $$P_0=P, P_1,\ldots, P_r = P$$defines a loop of polygons in $X$, i.e. a sequence of polygons such that any two consecutive polygons intersect along an edge. For each $i$, let $u_i$ be the vertex of $X'$ that is the centre of $P_i$ and let $u_i'$ be the vertex of $X'$ that is the centre of the edge $P_i \cap P_{i+1}$. Then the path 
$$u_0, u_0', u_1, u_1', \ldots, u_r=u_0$$ defines a non-backtracking loop in $X'$. 
\end{lem}
 \begin{proof}
 Since $\xi$ is reduced, each syllable $x_i$ is non-trivial. In particular, $P_i \neq P_{i+1}$, hence $u_i \neq u_{i+1}$. Moreover, since $\xi$ is reduced,  
$x_i$ and $x_{i+1}$ never belong to the same edge-group of $P$. Thus, the edges $P_{i-1}\cap P_i$ and $P_i \cap P_{i+1}$ are distinct, hence $u_i \neq u_{i+1}'$. Thus, the path $u_0, u_0', u_1, u_1', \ldots, u_r$ has no backtracking.
\end{proof}

\noindent In order to study the loops coming from the above lemma, we now introduce a piecewise-Euclidean metric on $X'$ (and thus on $X$), which we prove to be CAT(0).

\begin{defn}[Metric on $X'$]\label{def-verts-metric-X'}
    A vertex $v$ of the square subdivision $X'$ of $X$ is said to be 
    \begin{itemize}
        \item of type $(a)$ if it is the centre of a polygon of $X$;
        \item  of type $(b)$ if it is the midpoint of an edge of $X$;
        \item of type $(c)$ if is it a vertex of $X$ prior to subdivision.
    \end{itemize}
    \noindent A square of $X'$ has two opposite vertices of type $(b)$ and two vertices of type $(a)$ and $(c)$, respectively.    When $m\geq4$  let $\theta\coloneqq \pi/2$ (see Figure \ref{fig:subdivisionofX}),  and when $m=3$ let $\theta\coloneqq \pi/3$ (see Figure~\ref{fig:subdivisionofX-triangle}).The metric that we choose on $X'$ is the piecewise-Euclidean metric where each square is identified with the parallelogram of the Euclidean plane satisfying the following (see Figure~\ref{fig:metric-X'}):
    \begin{itemize}
        \item The angle at each vertex of type $(b)$ is $\pi/2$.
        \item The angle at the vertex of type $(c)$ is $\theta$.
        \item The angle at the vertex of type $(a)$ is $\pi-\theta$.
        \item Edges between vertices of type (b) and (c) have length $1$. 
    \end{itemize}
\end{defn}

\begin{figure}[h!]
\centering
\includegraphics[width=8.5cm]{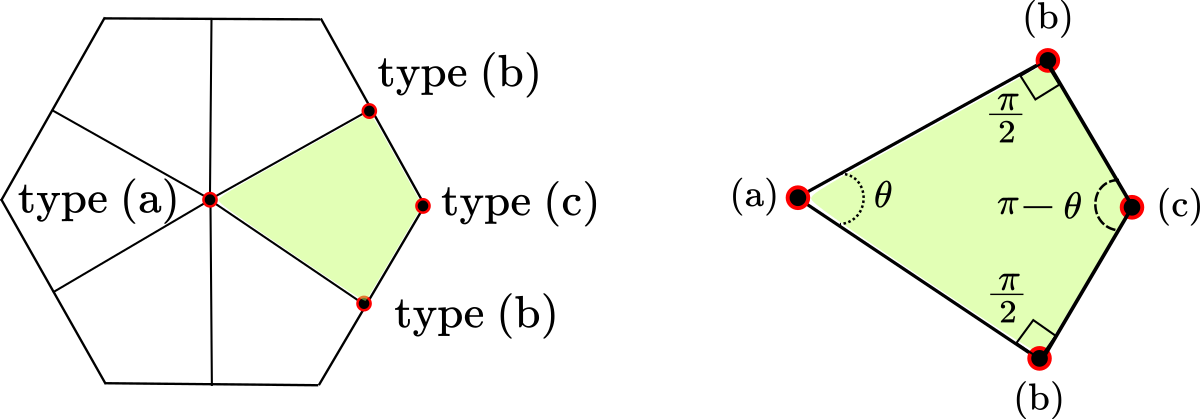}
    \caption{The metric on the squares of $X'$.}
\label{fig:metric-X'}
\end{figure}

\begin{figure}[h!]
\centering
\begin{minipage}{7cm}
    \centering \includegraphics[width=7cm]{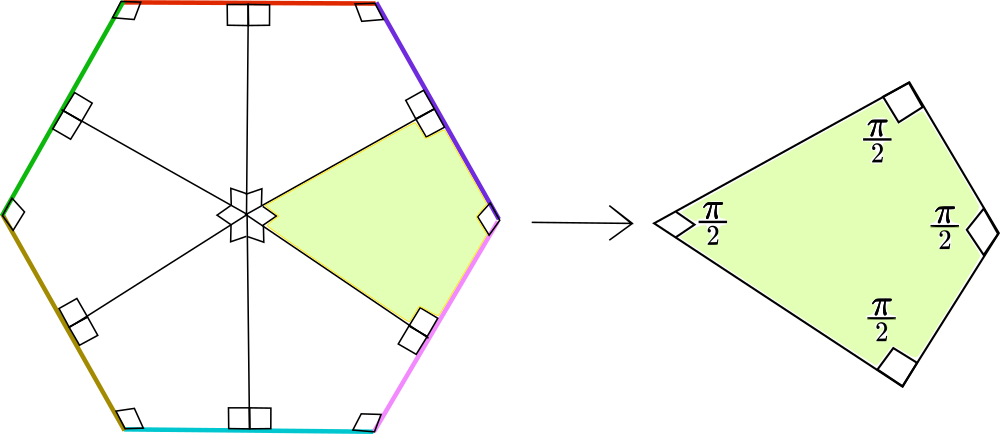}
    \caption{Square subdivision of the polygon $P$ for $m\geq 4$.}
\label{fig:subdivisionofX}
\end{minipage}
\qquad
\begin{minipage}{7cm}
        \centering \includegraphics[width=7cm]{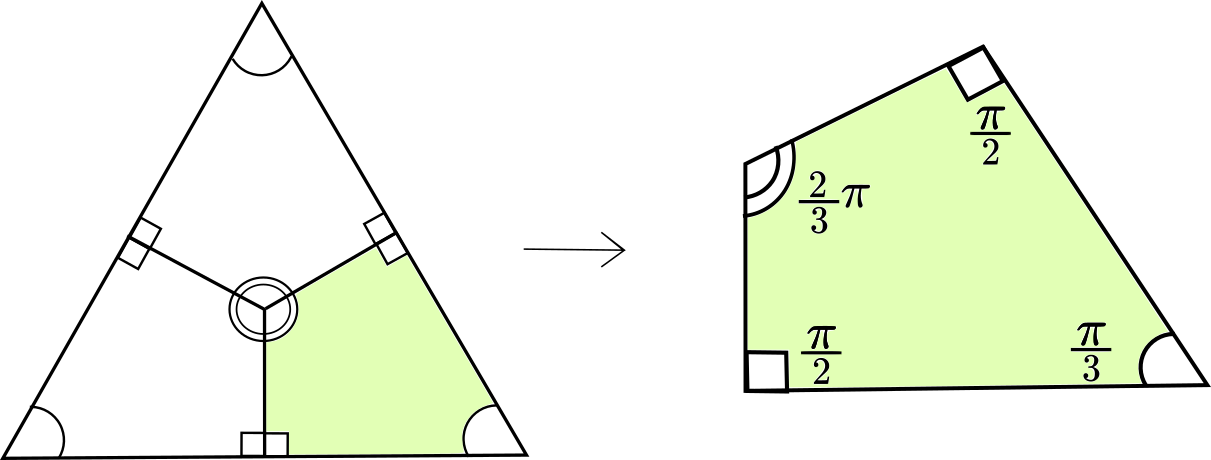}
    \caption{Square subdivision of the polygon $P$ for $m=3$.}
\label{fig:subdivisionofX-triangle}
\end{minipage}
\end{figure}

\begin{prop}\label{prop-X'-CAT(0)}
For the choice of metric done in Definition \ref{def-verts-metric-X'}, the polygonal complex $X'$ is CAT(0).
\end{prop}

\noindent A key result to show that $X'$ is non-positively curved is the following lemma, which is a sort of generalization of Appel--Schupp Lemma \ref{AppelSchuppLemma}. Recall that we denote the free group on $n$ generators by $\mathbb{F}_n$. If $x_1,\ldots,x_n$ are the explicit generators of $\mathbb{F}_n$, we write $\mathbb{F}_n=\mathbb{F}(x_1,\ldots,x_n)$. 

\begin{lem}\label{lem-Extended-AppelSchupp}
\noindent Let $m \geq 2$, and let $\Gamma=\Gamma^1_m\sqcup\Gamma^2_m$ be the Coxeter graph given by two disjoint edges with label $m$ and vertex set $V(\Gamma_m^{(i)})=\{\beta_i,\gamma_i,\}$ for  $i=1,2$. Consider the free subgroups $H_{\beta} \coloneqq \A_{\{\beta_1, \beta_2\}}=\mathbb{F}(\beta_1,\beta_2)$ and $H_{\gamma}\coloneqq \A_{\{\gamma_1, \gamma_2\}}=\mathbb{F}(\gamma_1,\gamma_2)$ of $\A[\Gamma]$. Let $$g=b_1c_1\cdots b_kc_k$$ be an element of the free product $H_{\beta}\frpp H_{\gamma}$, written in normal form. If $\ev_{\A[\Gamma]}(g) = \id_{\A[\Gamma]}$, then the syllabic length of $g$ with respect to the free product $H_{\beta}\frpp H_{\gamma}$  satisfies $\|g\|_{H_{\beta}\ast H_{\gamma}}\geq 2m$. 
\end{lem}

\begin{proof}
    Let us write each $b_i$ in normal form in the free group $\mathbb{F}( \beta_1, \beta_2)$ and each $c_i$ in normal form in $\mathbb{F}( \gamma_1, \gamma_2)$. Then the concatenation $b_1c_1\cdots b_kc_k$ yields an element in normal form in $\mathbb{F}(\beta_1,\beta_2,\gamma_1,\gamma_2)$. By grouping together on the one hand all the syllables of the form $\beta_1^\ast$ and $\gamma_1^\ast$, and on the other hand all the syllables of the form $\beta_2^\ast$ and $\gamma_2^\ast$, we can rewrite this element $g$ as a concatenation of the form 
    $g = \eta_1^{(1)}\eta_1^{(2)}\cdots \eta_\ell^{(1)}\eta_\ell^{(2)}$, which is in normal form for the decomposition of $ \mathbb{F}(\beta_1,\beta_2,\gamma_1,\gamma_2)$ as the free product $\mathbb{F}( \beta_1, \gamma_1)\frpp \mathbb{F}(\beta_2,\gamma_2)$ (i.e. each $\eta_i^{(1)} \in \mathbb{F}( \beta_1,\gamma_1)$ non-trivial except maybe $\eta_1^{(1)}$, and each $\eta_i^{(2)} \in \mathbb{F}( \beta_1,\gamma_1)$ non-trivial except maybe $\eta_\ell^{(2)}$). Under the evaluation map $\ev_{\A[\Gamma]}: H_{\beta} \frpp H_{\gamma} \rightarrow \A[\Gamma]$, we have that
    $$\id_{\A[\Gamma]} = \ev_{\A[\Gamma]}(g) = \mathrm{ev}_{\A[\Gamma]}(\eta_1^{(1)})\;\ev_{\A[\Gamma]}(\eta_1^{(2)})\cdots \;\ev_{\A[\Gamma]}(\eta_\ell^{(1)})\;\mathrm{ev}_{\A[\Gamma]}(\eta_\ell^{(2)}).$$
    If none of the syllables $\mathrm{ev}_{\A[\Gamma]}(\eta_j^{(i)})$ is trivial, then this a normal form in the free product $\A[\Gamma] = \A_{\{\beta_1,\gamma_1\}}\frpp \A_{\{\beta_2, \gamma_2\}}$, hence $\mathrm{ev}_{\A[\Gamma]}(g)$ is non-trivial, which is a contradiction. Thus, one of the syllables $\mathrm{ev}_{\A[\Gamma]}(\eta_j^{(i)})$ is the trivial element of $\A_{\{\beta_1,\gamma_1\}}$ or $\A_{\{\beta_2, \gamma_2\}}$, for some $j=1,\ldots,\ell$. This means that the word $\eta_j^{(i)}$, which is in normal form in the free product $\mathbb{F}(\beta_i, \gamma_i) = \langle \beta_i \rangle \frpp \langle \gamma_i\rangle$, satisfies $$\mathrm{ev}_{\A_{\{\beta_i, \gamma_i\}}}(\eta_j^{(i)}) = \id_{\A_{\{\beta_i, \gamma_i\}}}$$
    By the standard Appel--Schupp Lemma \ref{AppelSchuppLemma} in the dihedral Artin group $\A_{\{\beta_i, \gamma_i\}}$, we get that $$||\eta_j^{(i)}||_{\langle \beta_i\rangle\ast \langle \gamma_i\rangle} \geq 2m.$$
    Since syllables of the form $\beta_i^\ast$ and $\gamma_i^\ast$ are in different free factors on the free product $H_{\beta}\frpp H_{\gamma}=\A_{\{\beta_1, \beta_2\}}\ast \A_{\{\gamma_1, \gamma_2\}}$, it follows that $$||g||_{H_{\beta}\ast H_{\gamma}} \geq ||\eta_j^{(i)}||_{\langle \beta_i\rangle\ast \langle \gamma_i\rangle} \geq 2m,$$
    as we wanted to show.
\end{proof}

\begin{proof}[Proof of Proposition \ref{prop-X'-CAT(0)}]
    The complex $X'$ is simply connected as (a subdivision of) the universal cover of the developable complex of groups $G(\mathcal{Q})$. As usual, we use Gromov's link criterion to show that $X'$ is locally CAT(0). Let $v$ be a vertex of $X'$. If $v$ is a vertex of type $(a)$, then it exactly belongs to $m$ squares, and a non trivial closed curve in $\mathrm{Link}(v)$ measures at least $m\cdot(\pi-\theta)\geq 2\pi$ for all $m\geq 3$. If $v$ is of type $(b)$, then it is incident to at least $4$ squares, thus a non trivial closed curve in $\mathrm{Link}(v)$ measures at least $4\cdot\pi/2=2\pi$. If $v$ is of type $(c)$, then it is a vertex of $X$ which corresponds to a standard parabolic subgroup of $\A[\hGammam]$ of the form $\langle\delta_{\beta},\delta_{-\beta},\delta_{\gamma},\delta_{-\gamma}\rangle$, and the angle at its corner is $\theta\in \{\pi/2,\pi/3\}$. As in the proof of Theorem \ref{thm-2dim-vD-CAT(0)}, an embedded loop in the link of $v$ corresponds to a cyclically reduced element of the form $b_1c_1\cdots b_kc_k$, where the syllables alternatively belong to two edge-groups $\langle\delta_{\beta},\delta_{-\beta}\rangle$ and $\langle \delta_{\gamma},\delta_{-\gamma}\rangle$.    It thus follow from Lemma \ref{lem-Extended-AppelSchupp} that the girth of $\mathrm{Link}(v)$ is at least $2m$. Therefore the length of a non trivial loop in $\mathrm{Link}(v)$ is at least $2m\cdot\theta\geq 6\;\pi/3=2\pi$ for all $m\geq 3$. Hence $X'$ is locally CAT(0), which concludes the proof.
\end{proof}

\medskip
\noindent From now on, when we refer to the complex $X'$, we will always mean this square complex equipped with the above choice of angles. With this notion of subdivision, note that $\A[\hGammam]=\pi_1(G(\mathcal{Q}))$ acts face preserving on the square complex $X'$. 

\subsection{\texorpdfstring{Disc diagrams and cyclically reduced elements}{Disc diagrams and cyclically reduced elements }}\label{subs-discdiagrams_inAgammahat}

\noindent In this subsection, we 
use disc diagram arguments to establish a lower bound on the number of syllables in a cyclically reduced element $\xi=x_1\cdots x_r$ in $G(\mathcal{Q})$ that represents the identity element of $\A[\hGammam]$.
\noindent We recall below terminology and results about disc diagrams and refer the reader to 
\cite[Chapter V]{LynSchu77} and \cite[Section 4]{McCammondWise02} for further details. \medskip
\\
\noindent A \textit{disc diagram} $D$ over a polygonal complex $X$ is a finite contractible planar CW-complex of dimension 2, together with a cellular map $D\longrightarrow X$ which restricts to a homeomorphisms on every closed 2-cell. We say that a disc diagram $D$ is \textit{reduced} if no two distinct 2-cells of $D$ that share an edge are mapped to the same polygon of $X$. We say that a disc diagram $D$ is \textit{non-degenerate} if its boundary is isomorphic to a circle, and \textit{degenerate} otherwise. \medskip\\
\noindent Given a locally finite 2-dimensional CW-complex and $x$ a point in its 1-skeleton, the \textit{link of $x$ in $X$}, denoted by $\mathrm{Link}(x)$, is the set of points that are at distance $\epsilon$ from $x$. This set has a natural graph structure where the vertices are the the points of $\mathrm{Link}(x)$ in the 1-skeleton of $X$, and the arcs are the points of $\mathrm{Link}(x)$ in the interior of the 2-cells of $X$. If $v\in X$ is a 0-cell and we regard the 2-cells of $X$ as polygons, then the edges of $\mathrm{Link}(v)$
correspond to the corners of these polygons attached to $v$. We will refer to a
particular edge in $\mathrm{Link}(v)$ as a \textit{corner of $R$ at $v$} if this edge comes from the
polygon $R\longrightarrow X$.\medskip\\
\noindent A vertex $v$ of a disc diagram $D\longrightarrow X$ over a polygonal complex $X$ of dimension 2 is said \textit{internal}, and we write $v\in \overset{\circ}{D}$, if its link is homeomorphic to a circle, while it is said \textit{boundary vertex}, and we write $v\in \partial D$, otherwise.\\
\\
\noindent Thanks to the Lyndon-van Kampen theorem, we know that to every null-homotopic and non-backtracking closed curve $\xi: \mathbb{S}^1\longrightarrow X$ in a polygonal complex $X$, one can associate a disc diagram $D\longrightarrow X$ whose restriction to the boundary is the given loop $\xi$. We say that \textit{$D$ fills $\xi$}.

\begin{defn}\label{defncurvature}
     Let $D$ be a planar, contractible 2-complex and suppose that every corner $c$ of a 2-cell of $D$ has been assigned a real positive number $\theta_c$, which will be the \textit{angle} at such a corner. The \textit{combinatorial curvature} of a vertex $v\in D$, denoted by $\kappa_{D}(v)$, is \[\kappa_D(v)=2\pi-\pi\,\cdot \chi(\mathrm{Link}(v))\;\;-\sum_{c\,\text{corner at }v}\theta_c\,; \]
    \noindent where $\chi(\mathrm{Link}(v))$ is the Euler characteristic of $\mathrm{Link}(v)$, i.e. number of vertices of $\mathrm{Link}(v)$ minus number of edges in $\mathrm{Link}(v)$. The \textit{curvature} of a closed 2-cell $f$ of $D$, denoted by $\kappa_D(f)$, is defined to be 
    \[
    \kappa_D(f)=2\pi-|\partial f|\cdot\pi\;\,-\sum_{c\;\text{corner of } f}\theta_c\;;
    \]
\noindent where $|\partial f|$ is the number of edges of $D$ in the boundary of $f$.
\end{defn}

\noindent The following result, known as the combinatorial Gauss-Bonnet theorem, will be our key tool to study loops in $X'$:

\begin{thm}\cite[Theorem 4.6]{McCammondWise02}\label{comb_GAusBonnet}
    Let $D$ be a planar contractible polygonal complex with angles. Then:
    \[
    \sum_{v\in D}\kappa_D(v)+\sum_{f\in D}\kappa_D(f)=2\pi.
    \]
\end{thm}

\begin{prop}\label{prop_eq_inKVA}
Let $\Gamma_m$ be the dihedral Coxeter graph with integer label $m\geq 3$, and let $\xi=x_1\cdots x_r$ be a non-empty, cyclically reduced element in the complex of groups $G(\mathcal{Q})$. If $\ev_{\A[\hGammam]}(\xi)=\idK$ in $\A[\hGammam]$, then $r\geq 2m$.   
\end{prop}

\begin{proof}[Proof of Proposition \ref{prop_eq_inKVA}]
\begin{figure}[h!]
    \centering \includegraphics[width=5cm]{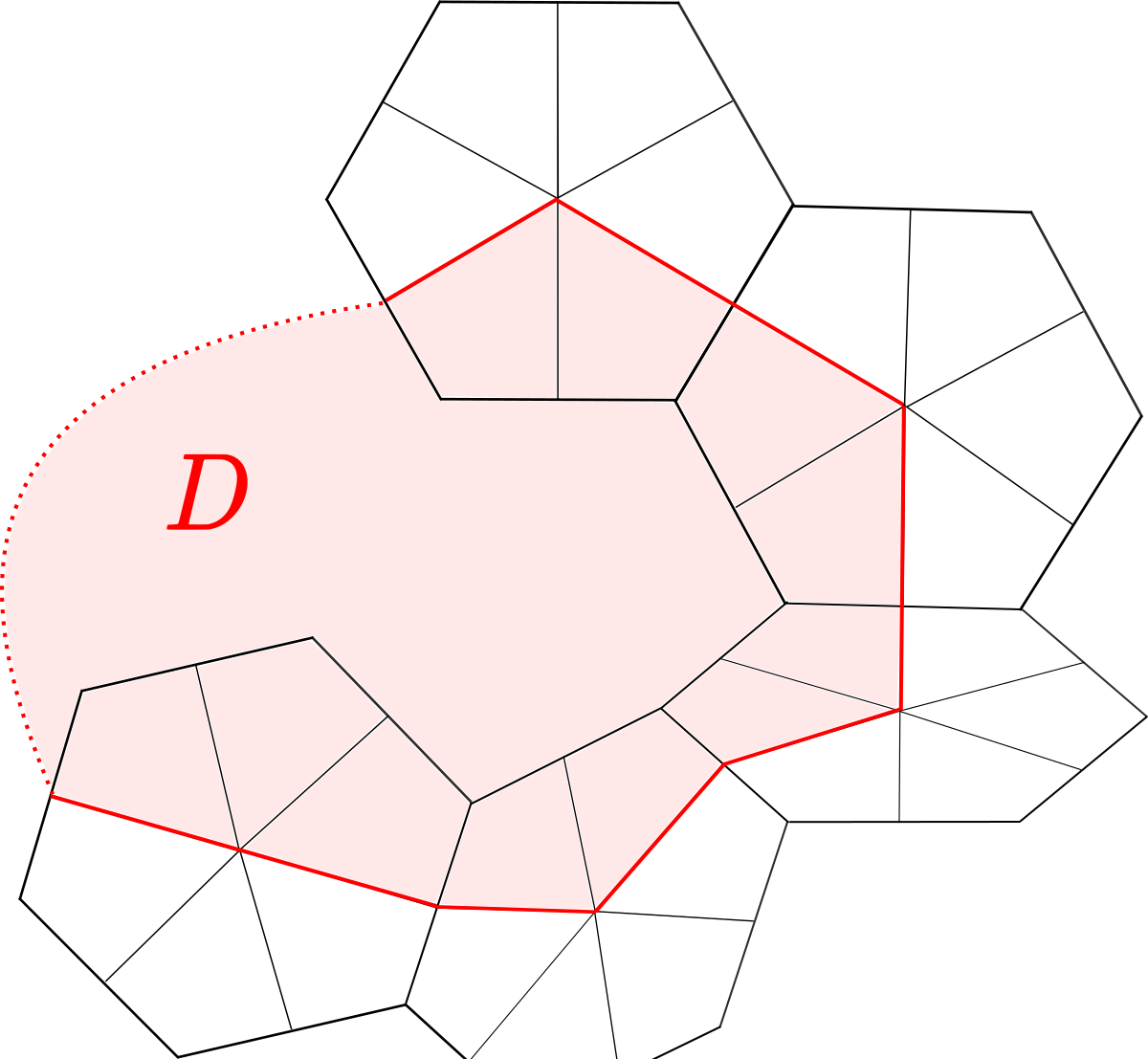}
    \caption{The disc diagram $D$ whose boundary is the cyclically reduced element $\xi=x_1\cdots x_r$ representing the identity in $\A[\hGammam]$.}
\label{fig:discdiagram}
\end{figure} 
    Let $X'$ be the square complex associated with $G(\mathcal{Q})$. Since $\xi=x_1 \cdots x_r$ represents the identity in $\A[\hGammam]$, it defines a null-homotopic loop in $X'$ without back-tracking by Lemma~\ref{lem-cycred-nonbackloops}, which we fill with a minimal diagram $D$. Note that it is enough to consider the case where $r$ is minimal, and in particular we can assume that the associated loop of $X'$ is embedded. We assign to each corner the angle coming from the CAT(0) metric on $X'$ (see Figure~\ref{fig:metric-X'}). Note in particular that for this choice of angles, each square has curvature zero, and the combinatorial Gauss-Bonnet theorem (Theorem \ref{comb_GAusBonnet}) boils to the following equality:
\[
2\pi=\sum_{v\in D}\kappa_D(v).
\] 
\noindent We now study the curvature of the vertices in $D$ (see Figure \ref{fig:discdiagram}), considering separately the case of boundary vertices and interior vertices. Observe first that vertices of type $(c)$, by construction, cannot belong to the boundary of $D$.

\medskip

\noindent \textbf{Claim 1.} Let $v$ be a boundary vertex of $D$. If $v$ is a of type $(a)$, then  $\kappa_D(v) \leq \theta$. If $v$ is of type $(b)$, then $\kappa_D(v)\leq 0$.

\medskip

\noindent Suppose that $v$ is of type $(a)$. Since the angle at every corner containing $v$ is $\pi-\theta$ by construction, it follows that $\kappa_D(v) \leq \theta$, as we wanted.\\
\noindent Suppose now that $v$ is of type $(b)$. Since the angle at every corner containing $v$ is $\pi/2$ by construction, in order to show that $\kappa_D(v) \leq 0$ it is enough to show that $v$ is not contained in a single square of $D$. This follows from the fact that $\partial D$ only contains vertices of type $(a)$ or (b), while two consecutive edges of a square of $X'$ necessarily contain a vertex of type $(c)$. This proves the claim.

\medskip

\noindent \textbf{Claim 2.} Let $v$ be a vertex in the interior of $D$. Then $\kappa_D(v) \leq 0$. Moreover, if $v$ is a vertex of type $(c)$, then $\kappa_D(v)\leq 2\pi - 2m\theta$. 

\medskip

\noindent Since $D$ is reduced, the map $D \rightarrow X'$ induces an immersion at the level of links of vertices. In particular, since $X'$ is CAT(0) by Proposition~\ref{prop-X'-CAT(0)}, it follows that $D$ itself is CAT(0). Thus, the girth of every interior vertex of $D$ is at least $2\pi$, hence $\kappa_D(v) \leq 0$. 

\noindent Suppose now that $v$ is a vertex of type $(c)$. Again, the map $f:D \rightarrow X'$ induces an immersion at the level of links of vertices. In particular, the girth of $\mathrm{Link}_D(v)$ is bounded below by the girth of $\mathrm{Link}_{X'}(f(v))$. Now since every corner of a vertex of type $(c)$ is $\theta$ by construction, it follows from Lemma~\ref{lem-Extended-AppelSchupp} that the girth of $\mathrm{Link}_D(v)$ is at least $2m\theta$, hence $\kappa_D(v) \leq 2\pi - 2m\theta$. This proves the claim.

\medskip

\begin{figure}[h!]
\centering
\begin{minipage}{7cm}
    \centering \includegraphics[width=4cm]{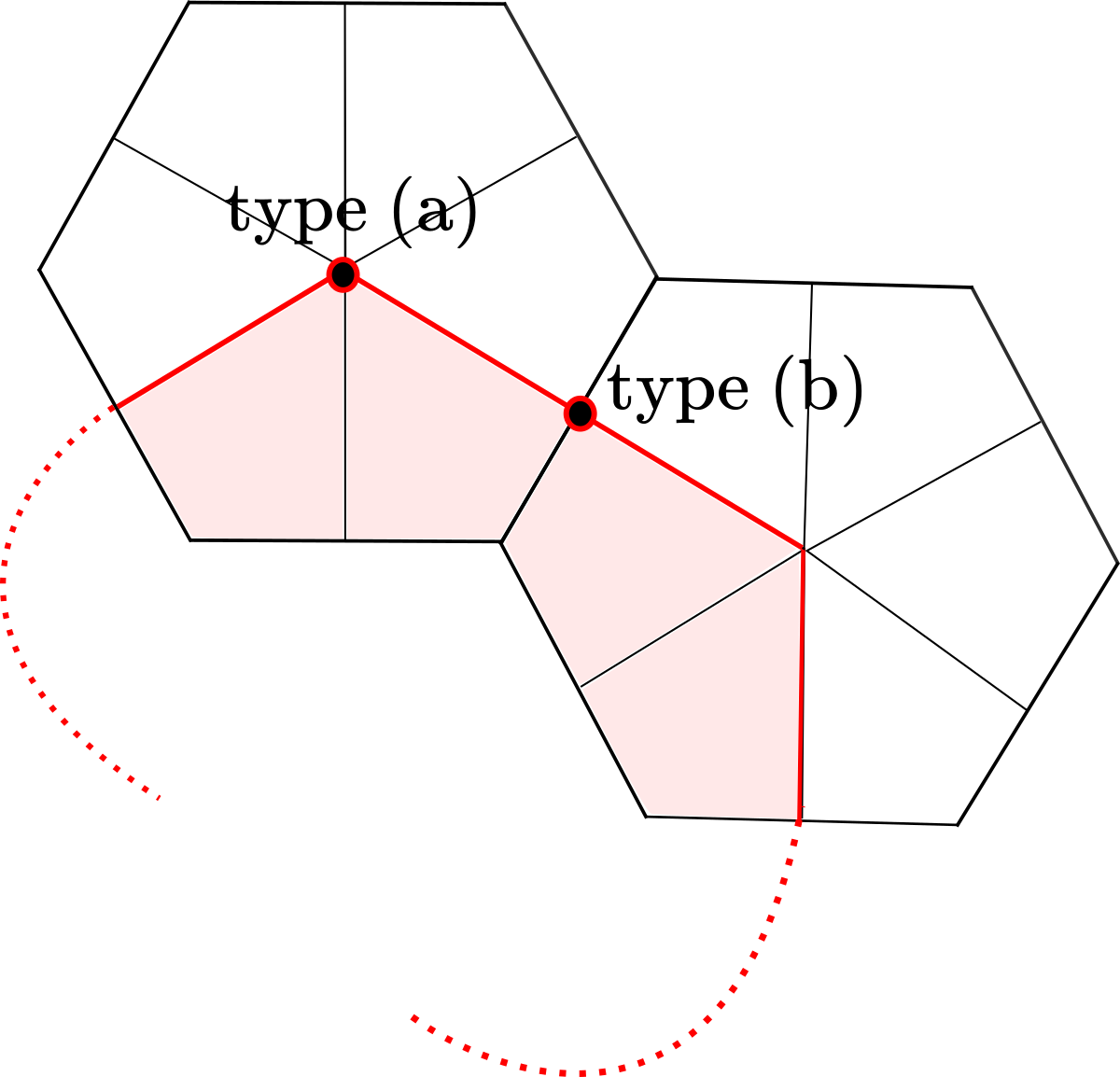}
    \caption{Boundary vertices in $\partial D$ of type $(a)$ and $(b)$.}
\label{fig:boudary-vertices}
\end{minipage}
\qquad
\begin{minipage}{7cm}
    \centering \includegraphics[width=4cm]{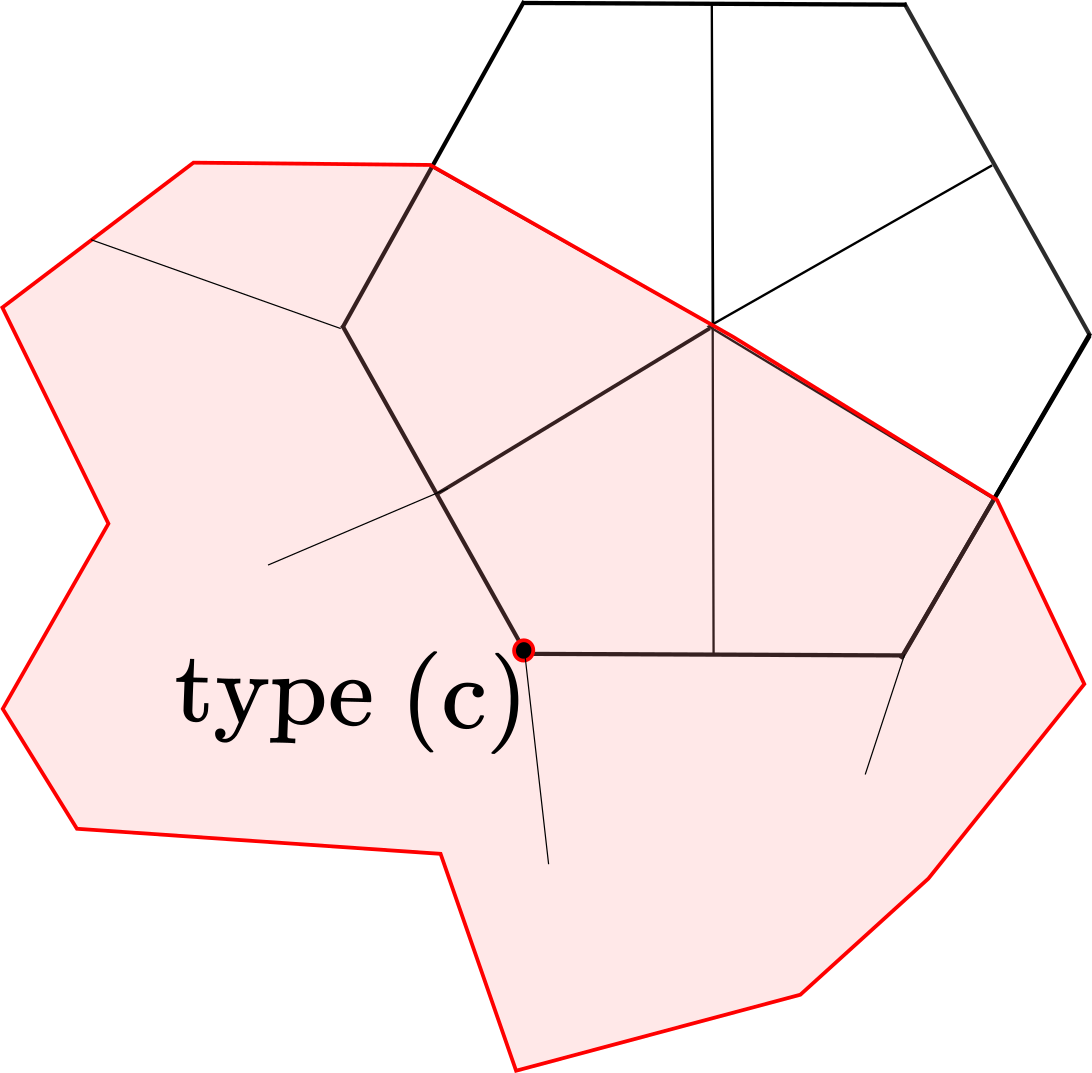}
    \caption{Interior vertices in $\overset{\circ}{D}$ of type $(c)$.}
\label{fig:interior_vertices}
\end{minipage}
\end{figure}

\noindent We now apply the combinatorial Gauss-Bonnet theorem (Theorem \ref{comb_GAusBonnet}):
\[
2\pi=\sum_{v\in D}\kappa_D(v).
\]
\noindent Since the loop of $X'$ associated to $\xi$ is embedded by assumption, it follows that $\partial D$ is homeomorphic to a circle. In particular, $D$ contains at least one square, hence vertices of type $(a)$, $(b)$,$(c)$. It thus follows from the above claims that

\begin{align}
2\pi=\underbrace{\sum_{v\,\text{of type $(a)$}}\kappa_D(v)}_{\leq \,r\,\theta}\;+\underbrace{\sum_{v\,\text{of type $(b)$}}\kappa_D(v)}_{\leq 0}\;+\underbrace{\sum_{v\,\text{of type $(c)$}}\kappa_D(v)}_{\leq 2\pi-2m \theta}\
\end{align}
Therefore $2\pi \leq 2\pi\,+(r-2m)\theta$, hence $ r\geq 2m$. 
\end{proof}

\section{\texorpdfstring{Proof of the virtual Appel--Schupp Theorem (Theorem \ref{virtualAppelSchupp})}{Proof of virtual Appel--Schupp (Theorem \ref{virtualAppelSchupp})}}\label{sec-virtualAppelSchupp}

In this section, we prove the following result.
\begin{reptheorem}{virtualAppelSchupp} Let $\Gamma_m$ be a dihedral Coxeter graph with vertex set $S=\{s,t\}$ and $m\geq 2$. Let $\VA_m$ be the associated virtual Artin group, and let $\omega$ be an element in $\VA_s\frpp \VA_t$ written in normal form such that $\ev_{\VA_m}(\omega)=\id_{\VA_m}$. Then the syllabic length of $\omega$ satisfies $\|\omega\|\geq 2m$. 
\end{reptheorem}
\noindent The fact that the minimal length of a  non-empty element in normal form in $\VA_s \frpp \VA_t$ that represents the identity in the quotient $\VA_m$ is at least $2m$ follows from the results established in Sections \ref{sectionDihedral} and \ref{sec-poly-of-groups-ArtinGammaHat}. Indeed, we have seen that $g_1\cdots g_{\ell}=\id\in \VA_m$ if and only if the following system of equations is satisfied:
\begin{equation*}
 \begin{cases}
k_1^{u_1}\,k_2^{u_2}\cdots k_{\ell}^{u_{\ell}}=\idK \quad&\text{in $\A[\hGammam]$};\\
    w_1\cdots w_{\ell}=\idW \quad &\text{in $\W_m$;} 
\end{cases}   \qquad\qquad\qquad (\text{\ref{eq2}})
\end{equation*}
\noindent where $g_i=(k_i,w_i)$, $u_1=\idW$ and $u_i=w_1\cdots w_{i-1}$ for all $i=1,\ldots,\ell-1$. \\

\noindent Unless in the proof of Theorem \ref{virtualAppelSchupp}, we will assume in this section that $m\geq 3$. We now focus on the first equation of (\ref{eq2}), aiming to interpret it as a certain null-homotopic, non-backtracking loop of polygons in $X'$, and then use Proposition \ref{prop_eq_inKVA} to bound its minimal length.

\subsection{\texorpdfstring{The equation in $\A[\hGammam]$}{The equation in the Artin group}} \label{subs-equation-in-A}

\noindent We start with the following observation:

\begin{lem}\label{reducedwordcasek_inontrivial} Let $\omega$ be an element in $\VA_s\frpp \VA_t$ satisfying equation (\ref{eq2}). Then the product $k_1^{u_1}\,k_2^{u_2}\cdots k_{\ell}^{u_{\ell}}$ defines a word in $G(\mathcal{Q})$, where $G(\mathcal{Q})$ is the complex of groups constructed in Definition~\ref{defn-complexofgroups-GQ}.
\end{lem}

\begin{proof}
    For all $i=1,\ldots,\ell$, let $a_i \coloneqq k_i^{u_i}$. The element $k_i$ belongs to $\A[\widehat{\Gamma_{r_i}}]$, which is the free group generated by $\delta_{\alpha_{r_i}}$ and $\delta_{-\alpha_{r_i}}$. 
    Denote by $\beta_i$ the root $\beta_i=u_i(\alpha_{r_i})$. Then $a_i=k_i^{u_i}=(u_i\cdot k_i)\;u_i^{-1}$ belongs to the free group $\mathbb{F}( \delta_{\beta_i},\delta_{-\beta_i})$. By construction of the polygon of groups $G(\mathcal{Q})$, we have that for every (positive) root $\beta\in \Phi_m^+$, the standard parabolic subgroup $\A[\hGamma_{\{\delta_{\beta},\delta_{-\beta}\}}]=\mathbb{F}(\delta_{\beta}, \delta_{-\beta})$ corresponds to an edge-group of $P$.
    \end{proof}

\noindent
However, the product $k_1^{u_1}\cdots k_{\ell}^{u_\ell}$ might not represent a cyclically reduced element in $G(\mathcal{Q})$, in which case we can put it in cyclically reduced form in the usual way, by removing trivial syllables and merging adjacent syllables in the same factor. In Definition~\ref{def:reduction_algorithm} below, we give a slightly different, but equivalent, reduction algorithm which will be more convenient in subsequent proofs.

\begin{nt}
    In this section, we will be working with words $\xi=x_1\cdots x_{r}$ in $G(\mathcal{Q})$ and will be constructing new words representing conjugated elements in $\A[\hGammam]$. Thus, we always think of such words up to conjugation, and in particular we will always think of the syllables $x_i$ as being indexed by $\mathbb{Z}_{/r}$. We will often represent such words in a loop, as in Figure~\ref{fig:intervals}.
\end{nt}

\begin{defn}[Fusible pairs]  Given a word $\xi=x_1\cdots x_{r}$ in $G(\mathcal{Q})$, we say that a pair $(x_{i_1}, x_{i_2})$ of syllables of $\xi$ is a \textit{fusible pair} if $x_{i_1},x_{i_2}$ are non-trivial and belong to the same edge group of $G(\mathcal{Q})$, and if $x_i = \idK$ for all $i \in \{i_1+1, \ldots, i_2-1\}$ (with $i \in \mathbb{Z}_
{/r}$ considered cyclically). We call  $I_{i_1,i_2}=\{i_1+1,\ldots,i_{2}\}$ the \textit{fusion interval} of the fusible pair $(x_{i_1}, x_{i_2})$, and the integer $z_{i_1,i_2}:=|I_{i_1,i_2}|$ the \textit{fusion gap}. 
\end{defn}

\begin{rmk}\label{rmk:disjoint_fusion_intervals}
Note that if $(x_{i_1}, x_{i_2})$ and $(x_{i_3}, x_{i_4})$ are distinct fusible pairs of $\xi$, then the fusion intervals  $I_{i_1,i_2}$ and $I_{i_3,i_4}$ are disjoint, see Figure~\ref{fig:intervals}.   
\end{rmk}

\begin{figure}[h!]
    \centering
\includegraphics[width=13cm]{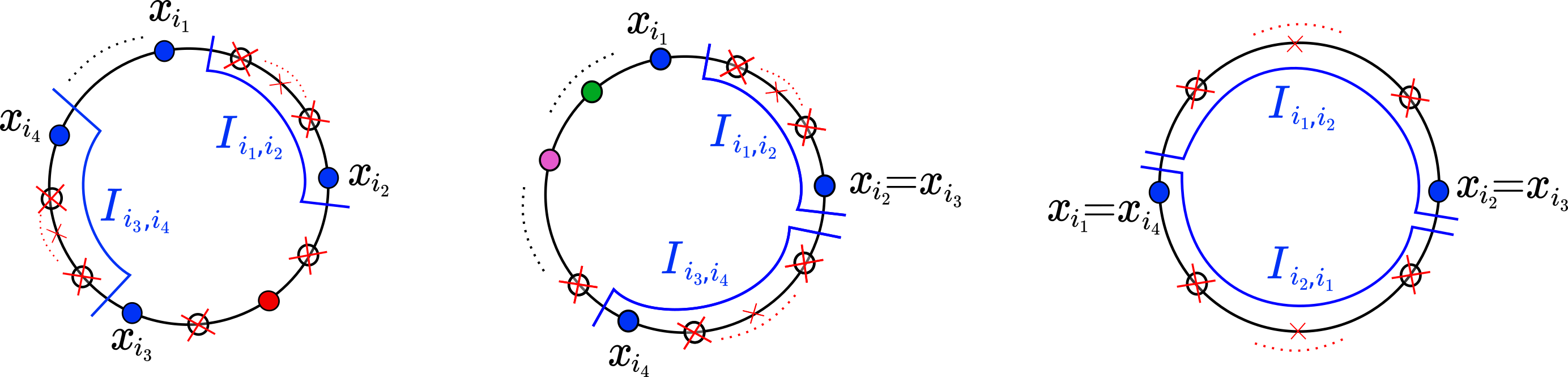}
    \caption{We represent a word $\xi = x_1\cdots x_r$ cyclically. In each image, two distinct fusible pairs and the corresponding disjoint fusion intervals.}
    \label{fig:intervals}
\end{figure}

\begin{defn}\label{def:reduction_algorithm}
  Given a word $\xi=x_1\cdots x_{r}$ in $G(\mathcal{Q})$, we construct a finite sequence of words $\xi^{(0)} = \xi, \xi^{(1)}, \ldots$ all representing elements in the same conjugacy class of $\A[\hGammam]$ by applying the following algorithm:
\begin{itemize}
\item \textbf{Fusion:} If $\xi^{(j)}$ contains a fusible pair, pick such a pair $(x_{i_1}, x_{i_2})$ and define
$$\xi^{(j+1)} = y_1\cdots y_r$$
with $y_i = x_i$ if $i \notin \{i_1, i_2\}$, $y_{i_1}= x_{i_1}x_{i_2}$, and $y_{i_2} = \idK$.
\item \textbf{Cancellation:} if $\xi^{(j)}$ no longer contains any fusible pair, define $\xi^{\CR}$ to be the word obtained from $\xi^{(j)}$ by removing all trivial syllables, and stop the algorithm. 
\end{itemize}
\end{defn}
\noindent  Observe that words $\xi^{(j)}$ have the same number of syllables by construction, but the number of non-trivial syllables in $\xi^{(j)}$ decreases with $j$. We have the following:
\begin{lem}
   Given a word $\xi=x_1\cdots x_{r}$ in $G(\mathcal{Q})$, the algorithm above terminates in a cyclically reduced product $\xi^{\CR}$ in $G(\mathcal{Q})$ such that $\xi$ and $\xi^{\CR}$ represent conjugated elements of $\pi_1(G(\mathcal{Q}))=\A[\hGammam]$.
\end{lem}
\noindent We now apply this reduction to the first equation of (\ref{eq2}).
\begin{defn}[The words $\nu, \nu'$ and $\nu^{\CR}$]\label{def:nu-dash-doubledash}
   Let $\omega$ be an element in $\VA_s\frpp \VA_t$ satisfying equation (\ref{eq2}). For each $i$, let $a_i \coloneqq k_i^{u_i}$. Let $\nu$ be the product $$\nu=a_1\cdots a_\ell$$ 
\noindent in $G(\mathcal{Q})$. Let $\nu'$ be the word obtained from $\nu$ by removing all trivial syllables $a_i=k_i^{u_i}=\idK^{u_i}=\idK$ with $k_i=\idK$, 
\begin{equation}
\nu'=b_1\cdots \,b_{\ell'},
\end{equation}
\noindent in $G(\mathcal{Q})$ where, for each $j=1,\ldots,\ell'$, the syllable $b_j$ is of the form $b_j:=a_{i_j}=k_{i_j}^{u_{i_j}}$ for some indices $i_1<\cdots<i_{\ell'}$ (up to cyclic permutation) and non-trivial elements $k_{i_j}$. Then let $\nu^{\CR}=c_1\cdots c_{\ell''}$ be the cyclically reduced word in $G(\mathcal{Q})$ obtained from $\nu$ (or, equivalently, from $\nu'$) by applying the algorithm from Definition~\ref{def:reduction_algorithm}.\medskip
\end{defn} 

\noindent By construction, $\nu^{\CR}$ is cyclically reduced, the words $\nu, \nu', \nu^{\CR}$ represent conjugated elements of $\A[\hGammam]$, and in particular $\ev_{\A[\hGammam]}(\nu)=\ev_{\A[\hGammam]}(\nu')=\ev_{\A[\hGammam]}(\nu^{\CR})=\idK$. Moreover, $\ell\geq \ell'\geq \ell''$. Observe that $\nu'$ or $\nu^{\CR}$ might be empty. If $\nu$ is already cyclically reduced, namely, if all the syllables $a_i$ are non-trivial and adjacent syllables never belong to the same edge-group of $G(\mathcal{Q})$, then $\nu=\nu'=\nu^{\CR}$.
In what follows we distinguish the cases $\nu^{\CR}$ non-empty and $\nu^{\CR}$ empty to find that, in both cases, the original length $\ell$ must be greater or equal to $2m$. In Subsection \ref{nu''-empty}, we combine all these results to complete the proof of Theorem \ref{virtualAppelSchupp}. The former case is straightforward and is considered in the next result.

\begin{prop}\label{prop-subs-equation-in-A}
    Let $\nu$, $\nu'$ and $\nu^{\CR}$ be as in Definition \ref{def:nu-dash-doubledash}. If $\nu^{\CR}$ is non-empty, then $\ell\geq2m$.
\end{prop}

\begin{proof}
We know that $\nu^{\CR}$ is a cyclically reduced, non-empty product in $G(\mathcal{Q})$ such that $\ev_{\A{\hGammam}}(\nu^{\CR})=\idK$. Then by Lemma \ref{lem-cycred-nonbackloops}, $\nu^{\CR}$ defines a non-backtracking loop in the polygon of groups $X'$ constructed in Section \ref{sec-poly-of-groups-ArtinGammaHat}. By Proposition \ref{prop_eq_inKVA}, we obtain that $\ell''\geq 2m$, and so that $\ell\geq 2m$.
\end{proof}

\subsection{\texorpdfstring{Case $\nu^{\CR}$ empty}{Case nu'' empty}} \label{nu''-empty}

\noindent The aim of this subsection is showing the following.

\begin{prop}\label{prop-nu''-empty}
Let $\nu$, $\nu'$ and $\nu^{\CR}$ be as in Definition \ref{def:nu-dash-doubledash}. If $\nu^{\CR}$ is empty, then $\ell\geq2m$.
\end{prop}

\noindent To show this result, we first treat separately some sub-cases. The case in which $\nu'=\nu^{\CR}$ is empty is treated below.

\begin{lem}\label{lem-nu'-empty}
    If $\nu'$ is empty, then $\ell \geq 2m$.
\end{lem}

\begin{proof}
    If $\nu'$ is empty, this means in particular that in $\nu$ all the syllables are trivial, and so $k_i=\idK$ for all $i$. Since the element $(k_i, w_i)$ is non-trivial by assumption, it follows that $w_i \neq \idW$ for all $i$. In particular, it follows that the element $w_1\cdots w_\ell$ is cyclically reduced and represents the identity element of $\W_m$. It now follows from Theorem~\ref{AppelSchuppLemma} that $\ell \geq 2m$.
\end{proof}

\noindent In the rest of this subsection, we suppose that $\nu'$ is non-empty and $\nu^{\CR}$ is empty. Then $\nu$ contains a fusible pair $(a_p, a_q)$. 
In the following result, we analyse this situation in detail.

\begin{lem}\label{lemma_z_j}
Let $(a_p, a_q)$ be a fusible pair of $\nu$. Then the fusion gap $z_{p,q}=z$ is either $z\geq 2m$ or $z=m$. 
\end{lem}
\begin{proof}
  \noindent By definition, $a_p$ and $a_q$ belong to the same edge-group of $G(\mathcal{Q})$. Thus $a_p=k_{p}^{u_{p}}$ and $a_q=k_{q}^{u_{q}}$ with $k_{p}\neq \idK \neq k_{q}$. Recall that $k_{p}\in \A[\widehat{\Gamma_{r_{p}}}]$ with $r_{p}\in \{s,t\}$ and $r_{p}\neq r_{p+1}$. There are two possibilities:
\begin{itemize}
    \item [\textbf{(a)}] $k_{p}$ and $k_{q}$ both belong to $\A[\widehat{\Gamma}_{r_{p}}]$, i.e. $z=q-p$ is even;
    \item [\textbf{(b)}] $k_{p}\in \A[\widehat{\Gamma}_{r_{p}}]$ and $k_{q}\in \A[\widehat{\Gamma}_{r_{p}+1}]$, i.e. $z=q-p$ is odd.
\end{itemize}
\noindent Without loss of generality, assume $r_{p}=s$ and $r_{p+1}=t$.\\
\textbf{Case (a):} We have
\begin{align*}
  a_p=k_{p}^{u_{p}}\in \langle \delta_{u_{p}(\alpha_{s})}, \delta_{-u_{p}(\alpha_{s}))}\rangle,  \qquad \qquad  a_q=k_{{q}}^{u_{q}}\in \langle \delta_{u_{q}(\alpha_{s})}, \delta_{-u_{q}(\alpha_{s})}\rangle.
\end{align*}
\noindent Since the two syllables lie in the same edge-group, we must have 
\[
u_{p}(\alpha_s)=\pm u_{q}(\alpha_s).
\]
\noindent Write $u_{q}=u_{p+z}=w_1\cdots w_{p-1}\,w_{p}\cdots w_{p+z-1}$ as $u_{q}=u_{p}\,v_{p}$, where $v_{p}:=w_{p}w_{p+1}\cdots w_{p+z-1}$. Hence 
\begin{equation}\label{eqk_i=id}
u_{p}(\alpha_s)=\pm \,u_{p}v_{p}(\alpha_s) \qquad \Longrightarrow \quad\pm\,\alpha_s=\,v_{p} (\alpha_s). 
\end{equation}
\noindent Recall that, whenever $k_i=\idK$, the corresponding Coxeter component $w_{i}\neq \idW$, thus $v_{p}$ is written in normal form in $\W_s\frpp \W_t$, of syllabic length $z$ or $z-1$, depending on whether $w_{p}\neq \idW$ or not. If $v_{p}=\idW$ (i.e. if $\ev_{\W_m}(v_p)=\idW$), then by Appel--Schupp Lemma \ref{AppelSchuppLemma} we obtain $z\geq z-1\geq 2m$, and we are done.\medskip\\
\noindent Assume therefore that $\ev_{\W_m}(v_{p})=\ev_{\W_m}(w_{p}w_{p+1}\cdots w_{p+z-1})\neq \idW$ in $\W_m$. From Proposition \ref{prop_roots}, Equality \ref{eqk_i=id} with positive sign holds only if either $v_{p}=\idW$ (excluded), or $v_{p}=w_{p}\cdots w_{q-1}=\Prod_R(s,t;m-1)$ and $m$ is even.
A length computation gives
\[
m-1=\|v_{p}\|_{\W_m}=\begin{cases}
    z-1 &\text{ if $w_{p}=\idW$,}\\
    z &\text{ if $w_{p}\neq\idW$,}
\end{cases} \qquad\Longrightarrow\qquad z=\begin{cases}
    m &\text{ if $w_{p}=\idW$,}\\
    m-1 &\text{ if $w_{p}\neq\idW$.}
\end{cases} 
\]\noindent However, case (a) requires $z$ to be even, thus, since $m$ is even, the only admissible case is $z=m$.\medskip\\
From Proposition \ref{prop_roots}, Equality \ref{eqk_i=id} with negative sign holds only if either $v_{p}=s$, or if $m$ is even and $v_{p}=w_{p}\cdots w_{q-1}=\Prod_R(s,t;m)$. The case $v_{p}=s$ can happen either when $z=1$, $q=p+1$ and $w_{p}=s$, or when $z=2$, $q=p+2$, $w_{p}=\idW$ and $v_{p}=w_{p+1}=s$. In the first scenario, we have that $z=1$ is odd, which contradicts the assumption that $k_{p}$ and $k_{q}$ both belong to $\A[\widehat{\Gamma_s}]$. In the second scenario, it must be $w_{p+1}=s$. But for two consecutive indices, $k_i,k_{i+1}$ never belong to the same standard parabolic subgroup. Since we assumed $k_p\in \VA_s$, then $w_{p+1}\in \W_t$, and this yields a contradiction. Hence, Equation \ref{eqk_i=id} holds with negative sign only if $v_{p}=w_{p}\cdots w_{i_{1}+z-1}=\Prod_R(s,t;m)$ and $m$ is even, which means
\[
m=\|v_{p}\|_{\W_m}=\begin{cases}
    z-1 &\text{ if $w_{p}=\idW$,}\\
    z &\text{ if $w_{p}\neq\idW$,}
\end{cases} \qquad\Longrightarrow\qquad z=\begin{cases}
    m+1 &\text{ if $w_{p}=\idW$,}\\
    m &\text{ if $w_{p}\neq\idW$.}
\end{cases} 
\]\noindent Again, since both $m$ and $z$ are even,  one concludes $z=m$. \\
\\
\noindent \textbf{Case (b):} Compute now
\begin{align*}
   a_p=k_{p}^{u_{p}}\in \langle \delta_{u_{p}(\alpha_{s})}, \delta_{-u_{p}(\alpha_{s}))}\rangle,  \qquad \qquad  a_q=k_{{q}}^{u_{q}}\in \langle \delta_{u_{q}(\alpha_{t})}, \delta_{-u_{q}(\alpha_{t})}\rangle.
\end{align*}
\noindent The condition that they lie in the same edge-group becomes $u_{p}(\alpha_s)=\pm \,u_{p+z}(\alpha_t)$, that is, 
\begin{equation}\label{eq_k_i2}
u_{p}(\alpha_s)=\pm \,u_{p}v_{p}(\alpha_t) \qquad \Longrightarrow \quad\pm\,\alpha_s=v_{p} (\alpha_t).
\end{equation}
\noindent From part (3.) of Proposition \ref{prop_roots}, Equation \ref{eq_k_i2} holds with positive sign if and only if $v_{p}=w_{p}\cdots w_{p+z-1}=\Prod_R(t,s;m-1)$ and $m$ is odd. Similarly, Equation \ref{eq_k_i2} holds with negative sign only if $v_{p}=w_{p}\cdots w_{p+z-1}=\Prod_R(t,s;m)$ and $m$ is odd.
Again using the length computation, and observing that in case (b) the integer $z$ is odd, one deduces that the only possible case is $m$ odd and $z=m$.
\end{proof}

\begin{rmk}
    The previous lemma also implies, in particular, that if $k_i\neq \idK\neq k_{i+1}$ for some $i\in \mathbb{Z}_{/\ell}$, then $a_i$ and $a_{i+1}$ do not belong to the same edge-group of $G(\mathcal{Q})$.
\end{rmk}

\noindent The next lemma guaranties the existence of at least two fusible pairs in $\nu$. 
\begin{lem}\label{lem-justonepair}
     Suppose that $\nu^{\CR}$ is empty and that $\nu$ has at least one non-trivial syllable. Then $\nu$ contains at least two distinct fusible pairs.
\end{lem}
\begin{proof}
    We reason by contradiction and assume that there is a single fusible pair $(a_{i_1},a_{i_2})$ in $\nu$. (There must be at least one such pair since $\nu$ has at least one non-trivial syllable and $\nu^{\CR}$ is empty by assumption.)\medskip

   \noindent Since $(a_{i_2},a_{i_1})$ is not a fusible pair by assumption, we have that there exists at least a third non-trivial syllable $a_j$ for some $j \notin \{i_1, i_1+1, \ldots, i_2\}$ (red dot in Figure~\ref{fig:fusion-process}). Since $(a_{i_1},a_{i_2})$ is the unique fusible pair of $\nu$ and $\nu'$ is non-trivial, to get the cyclically reduced word $\nu^{\CR}$ we must first fuse the pair $(a_{i_1},a_{i_2})$, thus obtaining a new word $\nu^{(1)}$. Since the resulting word $\nu^{(1)}$ contains a non-trivial syllable by the above, there must exist a fusible pair $(a_{i_3}, a_{i_4})$ of $\nu^{(1)}$. Note that this pair cannot be a fusible pair of $\nu$ by assumption, so in particular we must have $I_{i_1,i_2}\subsetneq I_{i_3, i_4}$. Moreover, $(a_{i_4}, a_{i_3})$ is not a fusible pair of $\nu^{(1)}$, otherwise it would be a fusible pair of $\nu$ already, contradicting our assumption. Thus, there exists another non-trivial syllable $a_{j}$ for some $j \notin \{i_3, i_3+1,\ldots, i_4\}$. 
   
\begin{figure}[h!]
        \centering
\includegraphics[width=9cm]{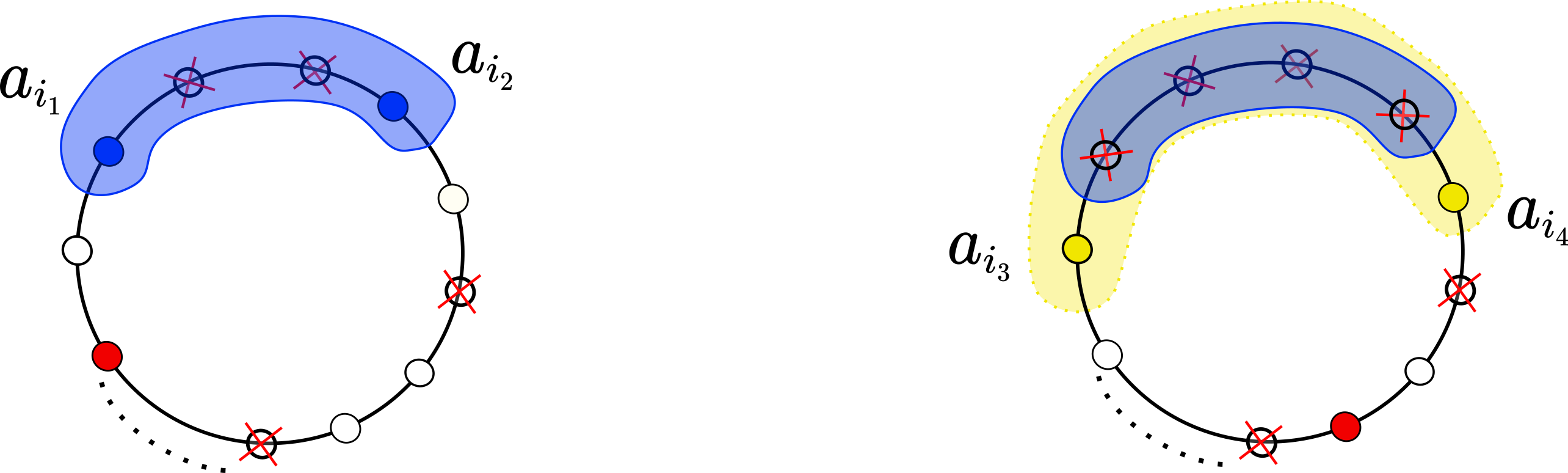}
    \caption{On the left: The word $\nu$, the fusible pair $(a_{i_1}, a_{i_2})$ and a non-trivial syllable (in red). On the right: The word $\nu^{(1)}$, the fusible pair $(a_{i_3}, a_{i_4})$ and a non-trivial syllable (in red).}
\label{fig:fusion-process}
    \end{figure}
   \noindent In this way, we can always repeat the process and thus produce a sequence $\nu^{(j)}$, for $j\geq 1$, of words all representing the identity element, where at each step we fuse a pair $(a_{i_{2j+1}}, a_{i_{2j+2}})$ of the word $\nu^{(i)}$ such that $I_{i_{2j-1}, i_{2j}}\subsetneq I_{i_{2j+1}, i_{2j+2}}$. Since we get in this way a strictly increasing sequence of subsets of $\mathbb{Z}_\ell$, we get a contradiction. 
\end{proof}

\begin{proof}[Proof of Proposition \ref{prop-nu''-empty}]
\noindent If $\nu^{\CR}$ is empty, then either $\nu'$ is empty, in which case the result follows from Lemma~\ref{lem-nu'-empty}, or by Lemma~\ref{lem-justonepair} there is at least two distinct fusible pairs $(a_{i_1}, a_{i_2})$, and $(a_{i_3},a_{i_4})$ in $\nu$. In the latter case, the fusion intervals $I_{i_{1}, i_2}$ and $I_{i_3, i_4}$ are disjoint by Remark~\ref{rmk:disjoint_fusion_intervals}. Moreover, each contains at least $m$ elements by Lemma~\ref{lemma_z_j}, hence $\ell\geq 2m$.
\end{proof}

\noindent We are finally ready to give the proof of Theorem \ref{virtualAppelSchupp}.

\begin{proof}[Proof of Theorem \ref{virtualAppelSchupp}]
Let $\omega=g_1\cdots g_{\ell}$ in $\VA_s \frpp \VA_t$ be a non-trivial element written in normal form such that $\ev_{\VA_m}(\omega)=\id$ in the quotient $\VA[\Gamma_m]$. We consider the associated system of equations (\ref{eq2}). \medskip\\
\noindent In the case $m\geq 3$, the result follows from Propositions~\ref{prop-subs-equation-in-A} and~\ref{prop-nu''-empty}. Suppose now that $m=2$, and we want to show that $\ell\geq 2m=4$. Note that it is enough obtain the desired bound when $\ell$ is even. The only case to rule out is thus $\ell =2$. But an equality of the form $h_sh_t=\id$ in $\VA_m$, with $h_s\in \VA_s$ and $h_t\in \VA_t$ non-trivial, would imply that $\VA_s$ and $\VA_t$ have a non-trivial intersection in $\VA_m$, contradicting~\cite[Theorem~1.2]{GaGaPa26}. 
\end{proof}

\section{\texorpdfstring{Finite subgroups of $\VA[\Gamma]$}{Finite subgroups of }}\label{sec-fin-subg}

\noindent In this section we show that $\VA[\Gamma]$ satisfies the torsion conjecture when $\Gamma$ is locally reducible. From now on, with abuse of notation, if a group $H<\VA[\Gamma]$ is such that $H=\iota_{\W}(H')$ for $H'<\W[\Gamma]$, we simply write $H<\W[\Gamma]$. \medskip

\begin{thm}\label{thm-finitesubgroups}
    Let $\Gamma$ be a locally reducible Coxeter graph, and let $H\leq \VA[\Gamma]$ be a finite subgroup. Then, up to conjugacy, $H<\W[\Gamma]$. Namely, $\VA[\Gamma]$ satisfies the torsion conjecture.
\end{thm}

\begin{rmk}
    By a result of Tits \cite[Ex 2d) p.137]{Bourbaki}, a finite subgroup of $\W[\Gamma]$ is contained in a finite parabolic subgroup of $\W[\Gamma]$. Thus, we get that in the locally reducible case, a finite subgroup of $\VA[\Gamma]$ is isomorphic to a direct product of dihedral Coxeter groups and cyclic groups of order~$2$.
\end{rmk}

\noindent To prove Theorem \ref{thm-finitesubgroups}, we reduce to the analogous statement in dimension 2, and more specifically to the dihedral case. As before, we denote by $\Gamma_m$ the dihedral Coxeter graph with integer label $m\geq 2$. 
\begin{lem}\label{reduction-to-dihedral-lem}
If the torsion conjecture for virtual Artin groups $\VA[\Gamma]$ holds for all $\Gamma=\Gamma_m$ dihedral, then it holds for all locally reducible virtual Artin groups.
\end{lem}
\begin{proof}
 Let $\Gamma$ be locally reducible. By Corollary \ref{cor-CAT0-locred}, the virtual Deligne complex 
$\vD_{\Gamma}$ is CAT(0), and $\VA[\Gamma]$ acts on it simplicially and cocompactly. Let 
$H<\VA[\Gamma]$ be a finite subgroup. By the Bruhat–Tits fixed point theorem (see \cite[Corollary II.2.8]{BriHaefl}), every finite group acting by isometries on a complete CAT(0) space fixes a point. Thus $H$ fixes a point $x\in \vD_{\Gamma}$. Recall that $\vD_{\Gamma}$ is complete thanks to Lemma \ref{prop-vD-complete-action-no-inversion}. Let 
$C_x$ be the unique cell whose interior contains $x$. Then $H$ stabilises $C_x$ setwise. Again by Lemma \ref{prop-vD-complete-action-no-inversion}, the action of $\VA[\Gamma]$ on $\vD_{\Gamma}$ is without inversion, thus $H$ must fix a vertex $\eta\VA_X$ of $\vD_{\Gamma}$, so that $H\leq \Stab(\eta\,\VA_X)$. The stabiliser of $\eta\,\VA_X$ is the parabolic subgroup $\eta\,\VA_X\,\eta^{-1}$, where $X\in \Sf$. Since $\Gamma$ is locally reducible, the spherical-type subgroup $\VA_{X}$ decomposes as a direct product
\[
\VA_X=\VA_{X_1}\times \cdots \times\VA_{X_k},
\]
where $X_i\subseteq S$ and 
$|X_i|\leq 2$ for all $i=1,\ldots ,k$. Up to conjugacy, we may therefore assume $H\leq \VA_{X_1}\times \cdots \times \VA_{X_k}$.\medskip\\
\noindent
Each factor $\VA_{X_i}$ satisfies the torsion conjecture. Indeed, if $|X_i|=1$, $\VA_{X_i}=\mathbb{Z}\frpp \mathbb{Z}_{/2}$, and the finite subgroups of this free product lie, up to conjugacy, in the factor $\mathbb{Z}_{/2}$. If $|X_i|=2$ the group $\VA_{X_i}$ satisfies the conjecture by hypothesis. Then, up to conjugacy in the factor $\VA_{X_i}$, we can assume that the projection of $H$ onto the factor $\VA_{X_i}$ is contained in the corresponding Coxeter group 
$\W_{X_i}$. Since a finite subgroup of a direct product embeds into the product of its projections, it follows that $H$ is contained in $\W_{X_1}\times \cdots\times \W_{X_k}=\W_X$,
up to conjugacy. Therefore, up to conjugacy in $\VA[\Gamma]$, $H\leq \W[\Gamma]$. This completes the proof.   
\end{proof}

\noindent The rest of this section is devoted to proving the following.

\begin{thm}\label{thm-finitesubgroups-dihedral}
    Let $\Gamma_m$ be the dihedral Coxeter graph with vertex set $S=\{s,t\}$ and integer label $m\geq 2$. If $H\leq \VA_m$ is a finite subgroup, then, up to conjugacy, $H\leq \W_m$. Namely, $\VA_m$ satisfies the torsion conjecture.
\end{thm}

\noindent In 
Subsection \ref{subs-extending the action} we study the action of $\A[\hGammam]$ on its Deligne complex $\D_{\hGammam}$, and we extend this action to the entire virtual Artin group $\VA_m$. In Subsection \ref{subs-stabilisers-action} we analyse the corresponding stabilisers, and finally show Theorem \ref{thm-finitesubgroups-dihedral}.

\subsection{\texorpdfstring{Extending the action of $\A[\hGammam]$ on its Deligne complex to an action of $\VA_m$}{Extending the action of Ah on dh}}\label{subs-extending the action}
\noindent Consider the Coxeter graph $\hGammam$ such that $V(\hGammam)=\Phi_m$, together with its Deligne complex $\D_{\hGammam}$. 
    It follows from Proposition~\ref{propdihedralgammahat} that the graph $\hGammam$ is 2 dimensional. 
    This implies by \cite{CharDav95} that $\D_{\hGammam}$, equipped with the Moussong metric described in Subsection \ref{subs-davis-and-deligne-complex}, is CAT(0). We recall that a strict fundamental domain for the action of $\A[\hGammam]$ on  $\D_{\hGammam}$ is given by $K_{\hGammam}$, the geometric realisation of the poset \[\widehat{\mathcal{S}^f}=\{\mathcal{X}\subseteq \Phi_m\mid \W_{\mathcal{X}}[\hGammam]\text{   \,is finite\,}\}.\]
\noindent The maximal subsets of the root system appearing above are of the form $\mathcal{X}=\{w(\alpha_s),w(\alpha_t)\}$ for some $w\in \W_m$. 
 We recall that $\VA_m=\A[\hGammam]\rtimes \W_m$, where $\W_m$ acts on $\A[\hGammam]$ as described in Equation \ref{actionWonKVA}. We now extend the action of $\A[\hGammam]$ on $\D_{\hGammam}$ to an action of the entire group $\VA_m$.
\begin{defn}
    Let $g=(k,u)$ be an element in $\VA_m=\A[\hGammam]\rtimes \W_m$, with $k\in \A[\hGammam]$ and $u\in \W_m$. Let $v=a\,\A_{\cX}[\hGammam]$ be a vertex in $\hDm$. We define the action of $g$ on $v$ as \begin{equation}\label{actionofVAonD}
g\cdot v= (k,u)\cdot a\,\A_{\cX}[\hGammam]:=k\,a^u\, \A_{u(\cX)}[\hGammam],
\end{equation}
\noindent where, as in Notation \ref{notationactionWonKVA}, $a^u$ denotes the element of $\A[\hGammam]$ such that $u\cdot a=a^u\,u$, and, if $\cX=\{w(\alpha_s),w(\alpha_t)\}$, then $u(\cX)=\{uw(\alpha_s),uw(\alpha_t)\}$.
\end{defn} \noindent In the following proposition, we show that this definition yields a well-defined action of $\VA_m$ on $\hDm$.

\begin{lem}\label{prop-actionofVAm-on-Deligne}
 Let $\VA_m$ act on the vertices of $\hDm$ like in Equation \ref{actionofVAonD}. Then this action extends to a face-preserving action of $\VA_m$ on $\hDm$.  Moreover, the action is without inversion.
\end{lem}
\begin{proof}
    Note that for $u=1$, we recover the action of $A[\hGammam]$ on $\vD_{\hGammam}$ by left multiplications. We first verify that the action of $\VA_m$ on the vertices of $\hDm$ satisfies the identity and compatibility axioms.\medskip\\
    \noindent Let $v=a\,\A_{\cX}[\hGammam]$ be a vertex of $\hDm$, where $a\in \A[\hGammam]$ and $\cX\in \widehat{\mathcal{S}^f}$. If $g=\id$ in $\VA_m$, then $(\idK,\idW)\cdot v=\id\,a^{\id}\,\A_{\cX}[\hGammam]=v$. Let now $g_1=(k_1,u_1)$, $g_2=(k_2,u_2)$ be two elements in $\VA_m$ such that $k_1,k_2\in \A[\hGammam]$ and $u_1, u_2\in \W_m$. We compute:
    \[
    g_1\cdot (g_2\cdot a\,\A_{\cX}[\hGammam])=g_1\cdot (k_2\,a^{u_2}\,\A_{u_2(\cX)}[\hGammam])=k_1\,k_2^{u_1}\,a^{u_1u_2}\,\A_{u_1u_2(\cX)}[\hGammam].
    \]
    \noindent On the other hand, 
    \[
    (g_1 g_2)\cdot a\;\A_{\cX}[\hGammam]= (k_1\,k_2^{u_1},u_1u_2)\cdot a\;\A_{\cX}[\hGammam] =k_1\,k_2^{u_1}\,a^{u_1u_2}\,\A_{u_1u_2(\cX)}[\hGammam].
    \]
    \noindent Thus, the action is well defined on the vertices of $\hDm$. Let us now show that this action preserves the simplicial structure of $\hDm$. Since $\hDm$ is the geometric realisation of the poset of left cosets of standard parabolic subgroups of $\A[\hGammam]$, it is enough to check that the action preserves adjacency of vertices. Two adjacent vertices of $\hDm$ are of the form $a\A_{\cX}[\hGammam]$, $a\A_{\cX'}[\hGammam]$ for some $a\in \A[\hGammam]$ and $\cX \subsetneq \cX'$ in $\widehat{\mathcal{S}^f}$, and one checks that an element $(k, u)$ of $\VA_m$ takes these vertices to the adjacent pair of vertices $ka^u\A_{u(\cX)}[\hGammam]$ and  $ka^u\A_{u(\cX')}[\hGammam]$.\medskip\\ 
\noindent  To see that the action is without inversions, observe that for every $u\in \W_m$, the subset $\cX\in \widehat{\mathcal{S}^f}$ is sent to $u(\cX)$, which also belongs to $\widehat{\mathcal{S}^f}$ and has the same cardinality as $\cX$. In particular, adjacent vertices are in different $\VA_m$-orbits, hence the action is without inversions. 
\end{proof}
\noindent 

\subsection{\texorpdfstring{Stabilisers of the action and proof of Theorem \ref{thm-finitesubgroups-dihedral}}{Stabilisers of the action}}\label{subs-stabilisers-action}
\noindent Once we have constructed the cocompact and face-preserving action of $\VA_m$ on $\hDm$, we study here its stabilisers.
\begin{lem}\label{lem-stabilisers}
Consider the action of $\VA_m$ on $\hDm$ described in Proposition \ref{prop-actionofVAm-on-Deligne}. Then the stabiliser of the vertex $a\,\A_{\cX}[\hGammam]$ is given by
\[
\Stab(a\,\A_{\cX}[\hGammam])=\begin{cases}
   a\,\W_m\, a^{-1}&\text{if $|\cX|=0$;}\\
   a\,(\A_{\cX}[\hGammam]\times\langle r\rangle)\,a^{-1} &\text{if $|\cX|=1$ and $m$ is even};\\
    a\,\A_{\cX}[\hGammam]\, a^{-1} &\text{if $|\cX|=1$ and $m$ is odd};\\
    a\,\A_{\cX}[\hGammam]\, a^{-1} &\text{if $|\cX|=2$};
\end{cases}
\]
\noindent where $r$ is the reflection $r=w\,\Prod_R(s,t;m-1)\,w^{-1}$ if $\cX=\{w(\alpha_s)\}$, and $r=w\,\Prod_R(t,s;m-1)\,w^{-1}$ if $\cX=\{w(\alpha_t)\}$.
\end{lem}

\begin{proof} Up to conjugation, we can assume that $a=\idK$ and compute the stabiliser of a vertex of the form $\A_{\cX}[\hGammam]$. Let $g=(k,u)$ be in $\VA_m$ and let $\cX\subseteq\{w(\alpha_s),w(\alpha_t)\}$ for some $w\in \W_m$. A direct computation shows that
\[  g\cdot \,\A_{\cX}[\hGammam]=k\,\A_{u(\cX)}[\hGammam] =\,\A_{\cX}[\hGammam]\qquad\Longrightarrow\qquad
  \begin{cases}
      k\in \A_{\cX}[\hGammam]\,;\\
      u(\cX)=\cX .
  \end{cases}
    \]
    \noindent If $\cX=\varnothing$, then $\A_{\cX}[\hGammam] = \{\idK\}$ and $u$ can be any element in $\W_m$, hence the stabiliser of $\A_{\varnothing}[\hGammam]$ is $\W_m$.\medskip\\
    \noindent Suppose now $|\cX|=1$, say $\cX=\{w(\alpha_s)\}$. Then $g=(k,u)$ stabilises $\A_{\cX}[\hGammam]$ if and only if $u(w(\alpha_s))=w(\alpha_s)$ and $k\in \A_{\cX}[\hGammam]$. By Proposition~\ref{prop_roots}, this  holds either when $m$ is even and $u$ is the reflection $r=w\,\Prod_R(s,t;m-1)\,w^{-1}$, or when $u=\idW$. In the first case, since $r$ commutes with $\delta_{w(\alpha_s)}$, we have that the stabiliser of $\A_{w(\alpha_s)}[\hGammam]$ is $\A_{\cX}[\hGammam]\times \langle r\rangle$. In the second case, the stabiliser of $\A_{w(\alpha_s)}[\hGammam]$ is $\A_{\cX}[\hGammam]$.
    \medskip\\
    \noindent Finally, if $|\cX|=2$, then $\cX=\{w(\alpha_s),w(\alpha_t)\}$.  The equality $u(\cX)=\cX$ implies that $w^{-1}uw$ either fixes $\alpha_s$ and $\alpha_t$, or interchanges them. By Corollary \ref{cor_roots}, the only element sending $\alpha_s$ to $\alpha_t$ is $\Prod_R(s,t;m-1)$, while the only element sending $\alpha_t$ to $\alpha_s$ is $\Prod_{R}(t,s;m-1)$, for $m$ odd. These are distinct elements of $\W_m$, hence no element of $\W_m$ interchanges $\alpha_s$ and $\alpha_t$. By Proposition~\ref{prop_roots}, an element of $\W_m$ fixing $\alpha_s$ and $\alpha_t$ is in 
    $$\{\idW, \Prod_{R}(s,t;m-1)\} \cap \{\idW, \Prod_{R}(t,s;m-1)\} = \{\idW\}$$ hence the stabiliser of $\A_{\cX}[\hGammam]$ is $\A_{\cX}[\hGammam]$.
\end{proof}
\noindent Another step for proving Theorem \ref{thm-finitesubgroups-dihedral} is next lemma.
\begin{lem}\label{lem-finitesubs&stabilisers}
   Let $H$ be a finite subgroup of $\Stab(a\,\A_{\cX}[\hGammam])$. Then, up to conjugacy, $H\leq \W_m$.
\end{lem}
\begin{proof}
Recall that, for $\Gamma$ of spherical type, $\A[\hGamma]$ is torsion-free by \cite[Corollary 6.5]{BellParThiel}.\medskip\\
\noindent Thus, for $\cX=\varnothing$, the statement follows immediately by Lemma \ref{lem-stabilisers}. If $|\cX|=2$, or $|\cX|=1$ with $m$ odd, the stabiliser is $a\,\A_{\cX}[\hGammam]\,a^{-1}$, which is torsion-free by the above result. Hence any finite subgroup must be trivial, and therefore $H<\W_m$. If $|\cX|=1$ and $m$ is even, the stabilizer is conjugated to \[\A_{\cX}[\hGammam]\times \langle r \rangle\cong \mathbb{Z}\times \mathbb{Z}_{/2}.\] \noindent  As the only finite subgroup of $\mathbb{Z}\times \mathbb{Z}_{/2}$ is $\{0\} \times \mathbb{Z}_{/2}$, it follows that, up to conjugacy, we have $H\leq \langle r \rangle<\W_m$.
\end{proof}
\noindent We are now ready to prove that dihedral virtual Artin groups satisfy the torsion conjecture.

\begin{proof}[Proof of Theorem \ref{thm-finitesubgroups-dihedral}]
    By Proposition \ref{prop-actionofVAm-on-Deligne}, the dihedral virtual Artin group $\VA_m$ acts without inversion on the Deligne complex $\hDm$, which is CAT(0) since $\hGammam$ is 2 dimensional. It follows from the Bruhat–Tits' fixed point theorem (see \cite[Corollary II.2.8]{BriHaefl}) that any finite subgroup $H<\VA_m$  fixes a vertex $v$,
    hence $H\leq \Stab(v)$. By Lemma \ref{lem-finitesubs&stabilisers}, any finite subgroup of a vertex stabiliser is, up to conjugacy, contained in $\W_m$, as we wanted to show.
\end{proof}

\noindent A consequence of Theorem \ref{thm-finitesubgroups} is the following. Note that the Artin group $\A[\hGamma]$ was known to be torsion-free for $\Gamma$ of spherical or of affine type (see \cite[Section 6]{BellParThiel}).

\begin{cor}\label{thm-A[gammahat]-torfree-locred}
    If $\Gamma$ is locally reducible, then the Artin group $\A[\hGamma]$ is torsion-free and $\VA[\Gamma]$ is virtually torsion-free.
\end{cor}

\begin{proof} Coxeter groups are virtually torsion-free (see for instance \cite[Corollary~6.12.12]{Davis2008}), thus for $\W[\Gamma]$ locally reducible, there exists a finite index subgroup $\W_0$ that is torsion-free. Consider now the $\pi_K$-pre-image of $\W_0$ in $\VA[\Gamma]$, and call it $H:=\pi_K^{-1}(\W_0)$. Clearly, $\A[\hGamma]=\KVA[\Gamma]=\pi_K^{-1}(\id)$, so $\A[\hGamma]<H$. Suppose now that $H$ has torsion, and let $h \in H$ be a non-trivial element of finite-order. By Theorem \ref{thm-finitesubgroups}, there exists $g \in \VA[\Gamma]$ such that $h \in g\,\iota_\W(\W[\Gamma])\,g^{-1}$. Since $\iota_\W$ is a section of $\pi_K$, the projection $\pi_K: \VA[\Gamma]\rightarrow \W[\Gamma]$ is injective on $g\,\iota_\W(\W[\Gamma])\,g^{-1}$, in particular $\pi_K(h)$ is a non-trivial element of finite-order of $\W[\Gamma]$. But by construction of $H = \pi_K^{-1}(\W_0)$, we have that $\pi_K(h)\in \W_0$, which is torsion-free, a contradiction. Thus there cannot be torsion in $H$, and, consequently, in $\A[\hGamma]$. The subgroup $H=\pi_K^{-1}(\W_0)$ is a finite index torsion-free subgroup of $\VA[\Gamma]$, thus $\VA[\Gamma]$ is virtually torsion-free.
\end{proof}

\begin{rmk}
    Locally reducible Artin groups have a finite dimensional classifying space (the Salvetti complex, see \cite{CharDav95,Charney2000}), thus they are torsion-free. However, when $\Gamma$ is locally reducible, we do not know whether $\hGamma$ is locally reducible or not. In affirmative case, this would give another proof of the fact that $\A[\hGamma]$ is torsion-free.
\end{rmk}

\noindent The fact that $\VA[\Gamma]$ possesses a finite index torsion-free subgroup motivates the study of its cohomological dimension, carried out in the next section.

\section{Classifying space for proper actions}\label{sec-cohom+classifying}

\noindent In this section, we construct classifying space for proper actions of minimal dimension for locally reducible virtual Artin groups.

\begin{defn}
    A \textit{classifying space for proper actions} for a group $G$ is a contractible CW-complex with properly discontinuous action of $G$, such that for every finite subgroup $H$ of $G$, the fixed-point set $X^H$ is contractible. (In particular, $X$ itself is contractible.) Such a space always exists and is unique up to $G$-equivariant homotopy equivalence, and is denoted $\underline{\mathrm{E}}G$.
    
   \noindent We denote by $\underline{\mathrm{gd}}(G)$ the minimum dimension of a classifying space for proper actions for $G$. 
\end{defn}

\noindent To state our main result, we need the following definition:

\begin{defn}
    Let $\cl(\Gamma)$ denote the \textit{clique number} of $\Gamma$, that is, the number of vertices in a maximal clique of $\Gamma$. Let $\cl^f(\Gamma)$ denote the \textit{spherical clique number} of $\Gamma$, that is, the number of vertices in a maximal spherical clique of $\Gamma$.
\end{defn}

\noindent The goal of this section is to prove the following theorem:

\begin{thm}\label{thm:cocompact_EG_bar}
    Let $\VA[\Gamma]$ be a locally-reducible virtual Artin group. Then $\VA[\Gamma]$ admits a cocompact model of classifying space for proper actions of dimension $\cl^f(\Gamma)$.
\end{thm}

\noindent Since we proved that locally reducible virtual Artin groups are virtually torsion-free by Corollary~\ref{thm-A[gammahat]-torfree-locred}, the following is immediate:

\begin{cor}
    Let $\VA[\Gamma]$ be a locally-reducible virtual Artin group. Then $\VA[\Gamma]$ is of type VF, i.e. it contains a finite-index subgroup that admits a compact Eilenberg--MacLane space. 
\end{cor}

\noindent We also get the following:

\begin{cor}\label{cor_vcd-VA}
    Let $\VA[\Gamma]$ be a locally-reducible virtual Artin group. Then 
    $$ \underline{\mathrm{gd}}(\VA[\Gamma]) = \mathrm{vcd}(\VA[\Gamma]) = \mathrm{cd}(\A[\Gamma]) = \cl^f(\Gamma) $$
\end{cor}

\noindent Before giving the proof, we mention the following result, which is probably known to experts:

\begin{lem}\label{lem:free_abelian_subgroup_spherical_number}
    Let $\A[\Gamma]$ be an arbitrary  Artin group. Then $\A[\Gamma]$ contains a free abelian subgroup of rank $\cl^f(\Gamma)$.
\end{lem}

\begin{proof}
It is enough to prove the result when $\Gamma$ is spherical. Moreover, up to decomposing such a spherical-type $\A[\Gamma]$ into a direct product of its maximal irreducible standard parabolic subgroups, it is enough to prove the lemma when $\A[\Gamma]$ is irreducible.

\noindent Let $v_1, \ldots, v_n$ denote the vertices of $\Gamma$ (with $n=\cl^f(\Gamma)) $, define $P_i \coloneqq  \A_{v_1, \ldots, v_i}[\Gamma]$, and let $z_i$ denote a generator of the centre of the standard parabolic subgroup $P_i$. We claim that $z_1, \ldots, z_n$ generate a $\mathbb{Z}^n$-subgroup. Since $\A[\Gamma]$ is torsion-free by \cite{Deligne1972}, it is enough to prove that for each $i> 1$, $z_{i}\notin \langle z_1, \ldots, z_{i-1}\rangle$. Since $\langle z_1, \ldots, z_{i-1} \rangle \subseteq P_{i-1}$ by construction, it is enough to prove that $z_{i}$ does not belong to a strict parabolic subgroup of $P_i$. Let $P$ denote the parabolic closure of $z_i$, which exists by Theorem~1.1 of \cite{CGGMW19}. Since $z_i$ is central in $P_i$, $P$ is normal in $P_i$. Since the only non-trivial normal standard parabolic subgroup of $P_i$ is $P_i$ itself by Lemma~3.3 of \cite{CMMW26}, we get $P=P_i$, which concludes the proof. 
\end{proof}

\begin{proof}[Proof of Corollary~\ref{cor_vcd-VA}]
    The existence of a classifying space for proper actions of dimension $\cl^f(\Gamma)$ implies that $\underline{\mathrm{gd}}(\VA[\Gamma])$, $\mathrm{vcd}(\VA[\Gamma])$, and $ \mathrm{cd}(\A[\Gamma])$ are bounded above by $\cl^f(\Gamma)$. Since the Artin group $\A[\Gamma]< \VA[\Gamma]$ contains a free abelian subgroup of rank $\cl^f(\Gamma)$ by Lemma~\ref{lem:free_abelian_subgroup_spherical_number}, we also get that $\underline{\mathrm{gd}}(\VA[\Gamma])$, $\mathrm{vcd}(\VA[\Gamma])$, and $ \mathrm{cd}(\A[\Gamma])$ are bounded below by $\cl^f(\Gamma)$.
\end{proof}

\subsection{Reduction to parabolic subgroups}\label{subs-reduction-parabolic}
\noindent The rest of this section is dedicated to proving Theorem~\ref{thm:cocompact_EG_bar}. The key result we will be using is the following combination theorem for classifying space for proper actions: 

\begin{thm}[{\cite[Theorem~1 and Corollary~6.3]{Martin2015EG}}]\label{thm:combination_gdim}
Let $G$ be a group acting cocompactly on a contractible simplicial complex $X$. Suppose that:
\begin{itemize}
    \item for each simplex $\sigma$ of $X$, the stabiliser $\Stab(\sigma)$ admits a cocompact model of classifying space for proper actions.
    \item for each finite subgroup $H$ of $G$, the fixed-point set $X^H$ is contractible. 
\end{itemize} 
Then $G$ admits a cocompact model of classifying space for proper actions. Moreover, we have
$$\underline{\mathrm{gd}}(G) \leq \mathrm{max}_{\sigma} \big(\dim(\sigma) + \underline{\mathrm{gd}}{(\Stab(\sigma))} \big).$$
\end{thm}

\begin{rmk}\label{rem:fixed-point-contractible}
    In what follows, we will apply Theorem~\ref{thm:combination_gdim} to the case of group actions on certain complete CAT(0) spaces. We note that in that case, the second condition in the above theorem  is automatically satisfied: the fixed-point set $X^H$ is non-empty by the Bruhat--Tits fixed-point theorem~\cite[Corollary II.2.8]{BriHaefl}, and the uniqueness of geodesics in a CAT(0) space guaranties that $X^H$ is convex, hence contractible. 
\end{rmk}

\noindent In order to prove Theorem~\ref{thm:cocompact_EG_bar}, we first prove the following reduction result:

\begin{prop}\label{prop:reduction_Ebar}
    Let $\Gamma$ be a Coxeter graph and let $\VA[\Gamma]$ be the corresponding virtual Artin group. Suppose that the following holds:
    \begin{itemize}
        \item For every spherical clique $\Lambda$ of $\Gamma$, the standard parabolic subgroup $\VA[\Lambda]$ admits a cocompact model of classifying space for proper actions of dimension $|V(\Lambda)|=\cl^f(\Lambda)$.
        \item The virtual Deligne complex $\vD_\Gamma$ is CAT(0).
    \end{itemize}
    Then $\VA[\Gamma]$ admits a cocompact model of classifying space for proper actions of dimension~$\cl^f(\Gamma)$.
\end{prop}

\begin{proof}
    We consider the action of $\VA[\Gamma]$ on its virtual Deligne complex $\vD_\Gamma$. An $n$-simplex $\sigma$ of $\vD_{\Gamma}$ corresponds to a chain of inclusions $g\VA_{X_0} \subset \cdots \subset g\VA_{X_n}$, with each $X_i$ spanning a spherical clique. The stabiliser $\mathrm{Stab}(\sigma)$ of such a simplex is isomorphic to the standard parabolic subgroup $\VA_{X_0}$ by construction, and by hypothesis admits a cocompact model of classifying space of dimension $|X_0|$. Moreover,  we have 
    $$\dim(\sigma) + \underline{\mathrm{gd}}(\mathrm{Stab}(\sigma)) =  n + |X_0| \leq  |X_n| \leq \cl^f(\Gamma).$$
    Additionally, $\vD_\Gamma$ is a complete CAT(0) space by assumption. It thus follows from Remark~\ref{rem:fixed-point-contractible} that, for every finite subgroup $H$ of $\VA[\Gamma]$, the corresponding fixed-point set is contractible. Thus, it  follows from Theorem~\ref{thm:combination_gdim} that $\VA[\Gamma]$ admits a cocompact model of classifying space for proper actions of dimension $\cl^f(\Gamma)$, which concludes the proof.
\end{proof}

\subsection{The case of spherical parabolic subgroups}

\noindent In this subsection, we prove the following:

\begin{prop}\label{cor:gd_VA_complete}
    If $\Gamma$ is a complete graph, spherical, and locally reducible, then  $\VA[\Gamma]$ admits a cocompact model of classifying space for proper actions of dimension $|V(\Gamma)| =\cl^f(\Gamma).$
\end{prop}

\begin{defn}
    A Coxeter graph is said \textit{irreducible} if it cannot be written as a join of two subgraphs $\Gamma_1$ and $\Gamma_2$, where all the edges joining $\Gamma_1$ to $\Gamma_2$ are labelled by 2.

    \noindent Given a general Coxeter graph $\Gamma$, there is a decomposition 
    $$\VA[\Gamma] = \prod_i \VA[\Gamma_i]$$
    where the $\VA[\Gamma_i]$ are maximal irreducible standard parabolic subgroups.
\end{defn}

\noindent In order to prove Theorem~\ref{thm:cocompact_EG_bar}, we start with the  case of irreducible cliques:

\begin{lem}\label{lem:gd_VA_complete_irreducible}
    Let $\Gamma$ be an irreducible clique on $|V(\Gamma)|\leq 2$ vertices.  
    Then $\VA[\Gamma]$ admits a cocompact model of classifying space for proper actions of dimension $|V(\Gamma)| = \cl^f(\Gamma) $
\end{lem}

\begin{proof}
    If $\Gamma$ is empty, then $\VA[\Gamma]$ is trivial and the result holds. If $\Gamma$ is a single vertex, then $\VA[\Gamma]$ is isomorphic to the free product $\mathbb{Z}\ast \mathbb{Z}_2$, for which the Bass-Serre tree of the splitting is a cocompact model of classifying space for proper actions of dimension $1$.\medskip\\
\noindent    Let us now assume that $\Gamma=\Gamma_m$ is an edge with label $m\geq 3$. We consider the induced action of $\VA[\Gamma]$ on the 2-dimensional Deligne complex $\D_{\hGammam}$ of $\A[\hGammam]$ constructed in Lemma~\ref{prop-actionofVAm-on-Deligne}. Up to the action of the group, an $n$-simplex $\sigma$ of $\D_{\hGammam}$ corresponds to a chain of inclusions $\A_{{\cX}_0} \subset \cdots \subset \A_{{\cX}_n}$, with ${\cX}_i$ a clique of $\hGammam$ of size $|{\cX}_i|\leq 2$ for all $i$. Moreover, the stabiliser of $\sigma$ coincides with the stabiliser of the vertex $\A_{\cX_0}$. By Lemma~\ref{lem-stabilisers}, the stabiliser of $\A_{\cX_0}$ is isomorphic to $\A_{\cX_0}\times Q$ for some finite group $Q$. In particular, since Artin groups on at most two generators satisfy the $K(\pi,1)$-conjecture (they are either spherical-type or $\mathbb{F}_2$, hence follows from \cite{Deligne1972} and \cite{CharDav95} respectively), and such a stabiliser admits a cocompact model of classifying space for proper actions of dimension $|{\cX}_0|$, namely the universal cover of the Salvetti complex of $\A_{\cX_0}$. In particular, we get 
    $$\dim(\sigma) + \underline{\mathrm{gd}}(\mathrm{Stab}(\sigma)) \leq  n + |\cX_0| \leq  |\cX_n| \leq 2.$$
    \noindent Moreover, it follows from Proposition~\ref{propdihedralgammahat} that $\hGammam$ is 2 dimensional, so the Moussong metric on $\D_{\hGammam}$ turns it into a CAT(0) space by ~\cite[Theorem~B]{CharDav95}. It thus follows from Remark~\ref{rem:fixed-point-contractible} that, for every finite subgroup $H$ of $\VA[\Gamma]$, the corresponding fixed-point set is contractible. Thus, it  follows from Theorem~\ref{thm:combination_gdim} that $\VA[\Gamma]$ admits a cocompact model of classifying space for proper actions of dimension $2$.
\end{proof}

\begin{proof}[Proof of Proposition~\ref{cor:gd_VA_complete}]
    Each such virtual Artin group splits as a direct product of irreducible factors of the form $\VA[\Gamma_i]$ with $|V(\Gamma_i)| =1$ or $2$.  Each such factor admits a cocompact model $\underline{\mathrm{E}}\VA[\Gamma_i]$ of classifying space for proper actions of dimension $|V(\Gamma_i)|$ by Lemma~\ref{lem:gd_VA_complete_irreducible}. It is then straightforward to check that the direct product $\prod_i\underline{\mathrm{E}}\VA[\Gamma_i]$ is a cocompact model for proper actions for $\VA[\Gamma]$ of dimension $\sum_i |V(\Gamma_i)| = |V(\Gamma)|$.
\end{proof}

\noindent We are now ready to prove the main result of this section. 

\begin{proof}[Proof of Theorem~\ref{thm:cocompact_EG_bar}]
    If $\Gamma$ is locally reducible, then its virtual Deligne complex is a complete CAT(0) space by Lemma \ref{prop-vD-complete-action-no-inversion} and Corollary \ref{cor-CAT0-locred}. It thus follows from Remark~\ref{rem:fixed-point-contractible} that, for every finite subgroup $H$ of $\VA[\Gamma]$, the corresponding fixed-point set is contractible. Moreover, for every spherical clique $\Lambda$ of $\Gamma$, the corresponding parabolic subgroup $\VA[\Lambda]$ admits a cocompact model of classifying space for proper actions of dimension $|V(\Lambda)| = \cl^f(\Lambda)$ by Proposition~\ref{cor:gd_VA_complete}. The result now follows from Proposition~\ref{prop:reduction_Ebar}.
\end{proof}

\printbibliography

\vspace{0.4cm}

\bigskip\noindent
\textbf{Federica Gavazzi}, 

\noindent Address: Department of Mathematics and the Maxwell Institute for the Mathematical Sciences, Heriot-Watt University, Edinburgh EH14 4AS, UK.

\noindent Email: \texttt{F.Gavazzi@hw.ac.uk}

\smallskip

\bigskip\noindent
\textbf{Alexandre Martin}, 

\noindent Address: Department of Mathematics and the Maxwell Institute for the Mathematical Sciences, Heriot-Watt University, Edinburgh EH14 4AS, UK.

\noindent Email: \texttt{alexandre.martin@hw.ac.uk}

\end{document}